\documentclass[11pt]{article}

\usepackage[T1]{fontenc}
\usepackage{times}

\usepackage{bm,color,epsfig,enumerate,caption}
\usepackage{amsmath,amssymb}
\usepackage{amsthm}
\usepackage{amsrefs}
\usepackage{indentfirst}
\usepackage{mathrsfs}
\usepackage{graphicx}
\usepackage{hyperref}
\usepackage{xcolor}
\usepackage{fourier}
\usepackage[margin=1.2in]{geometry}
\usepackage{algpseudocode}
\usepackage{algorithm}
\usepackage{algorithmicx}

\newtheorem{theorem}{Theorem}
\newtheorem{lemma}[theorem]{Lemma}

\newtheorem{definition}[theorem]{Definition}

\DeclareMathOperator{\argmin}{arg min}

\DeclareMathOperator{\mass}{mass}

\newcommand{\grad}{\nabla}
\newcommand{\RR}{\mathbb{R}}
\newcommand{\NN}{\mathbb{N}}
\newcommand{\ZZ}{\mathbb{Z}}
\newcommand{\Z}{\mathbb{Z}}
\newcommand{\TT}{\mathrm{T}}

\newcommand{\wt}[1]{\widetilde{#1}}

\newcommand{\mm}{u}

\newcommand{\ud}{\,\mathrm{d}}

\newcommand{\mc}[1]{\mathcal{#1}}

\newcommand{\eps}{\epsilon}

\newcommand{\abs}[1]{\lvert#1\rvert}

\newcommand{\norm}[1]{\left\lVert#1\right\rVert}

\renewcommand{\Re}{\mathfrak{Re}}

\usepackage{mathtools}

\makeatletter
\newcommand*{\extendadd}{
  \mathbin{
    \mathpalette\extend@add{}
  }
}
\newcommand*{\extend@add}[2]{
  \ooalign{
    $\m@th#1\leftrightarrow$%
    \vphantom{$\m@th#1\updownarrow$}
    \cr
    \hfil$\m@th#1\updownarrow$\hfil
  }
}
\makeatother

\begin{document}
\title{$3D$ Crystal Image Analysis based
    on\\ Fast Synchrosqueezed Transforms}
    
    \author{Tao Zhang \\ Department of Mathematics\\ National University of Singapore\\  \href{mailtoz{\_}tao@u.nus.edu}{z{\_}tao@u.nus.edu} 
   \and Ling Li \\ Department of Mechanical Engineering\\Virginia Polytechnic Institute and State University\\ \href{mailto:lingl@vt.edu}{lingl@vt.edu}
     \and Haizhao Yang \\ Department of Mathematics\\ National University of Singapore, Singapore\\ \href{mailto:haizhao@nus.edu.sg}{haizhao@nus.edu.sg} }

\maketitle

\begin{abstract}
We propose an efficient algorithm to analyze $3D$ atomic resolution crystal images based on a fast $3D$ synchrosqueezed wave packet transform. The proposed algorithm can automatically extract microscopic information from $3D$ atomic resolution crystal images, e.g., crystal orientation, defects, and deformation, which are important information for characterizing material properties. The effectiveness of our algorithms is illustrated by experiments of synthetic datasets and real $3$D microscopic colloidal images. 
\end{abstract}



\maketitle

\section{Introduction}
\label{sec:intro}
The microstructure of materials, e.g. the dynamics of defects (e.g., grain boundaries and isolated dislocations) and grain deformation, is one of the key factors that determins the physical properties of crystalline materials  \cite{callister1991materials,hull2001introduction,snyder2011complex}. A thorough understanding of the role of the microscopic dynamics helps the design of advanced functional materials in many applications. To study the impact of the microscopic dynamics, it is necessary to obtain experimental data throughout the crystal fabrication process and monitor the change of microstructure. A major bottleneck in this process is then analyzing this dynamic information, which often involves time-consuming manual structural identification and classification, and sometimes is even beyond the capability of manual evaluation, e.g., grain deformation.

For the purpose of automatical and efficient data analysis, there has been extensive research in designing crystal image analysis tools \cite{belianinov2015,keen2015,sutton1995,wadhawan2014,elsey2013,elsey2014,Yang2015,LU2016,Lu2018,Zosso2017,GrainGeo,berkels2008,boerdgen2010} for $2D$ data coming from advanced imaging techniques (such as high resolution transmission electron microscopy (HR-TEM) \cite{king1998}), mean field models like phase field crystals \cite{PhysRevE.70.051605}, and the atomic simulation of molecular dynamics \cite{abraham2002}. These algorithms include the famous variational methods for texture classification and segmentation  (see \cite{Mumford,Meyer:2001, Chan:01, Vese:02, Sandberg:02, Vese:03,Aujol2006} for example) with adaptation to crystal image analysis \cite{Berkels:08,Berkels:10,Strekalovskiy:11,ElseyWirth:13,ElseyWirth:MMS}, the phase-space analysis methods \cite{Unser:95, SingerSinger:06,Yang2015,LU2016,Lu2018}, the algorithms based on atom positions \cite{stukowski2010,GrainGeo}, and more recently deep learning approaches \cite{DL1,DL2,DL3,DL4,DL5}.

However, the majority of existing crystal analysis methods in the literature are limited to $2$D data and hence cannot meet the demand of $3D$ materials synthesis \cite{3Dimportant}. To the best of our knowledge, only a few $3D$ crystal image analysis methods were previously reported (\cite{Elsey2015} for grain segmentation assuming that the crystal lattice was known, \cite{DL1} for the identification of lattice symmetry of simulated crystal structures). An automatic algorithm for a complete analysis including lattice classification, grain segmentation, defect detection, deformation estimation, etc., is still not available. This motivates the development of a complete framework to analyze $3D$ atomic resolution crystal images in this paper.

Our first contribution is the development of a fast $3D$ synchrosqueezed transform (SST) based on a $3D$ wave packet transform with geometric scaling parameters to control the support sizes of wave packets, making the $3D$ SST adaptive to complicated atomic configurations. A band-limited version of the SST significantly speeds up its application to $3D$ data making it practical for large images. Our second contribution is to extend the algorithms in \cite{Yang2015,LU2016,Lu2018} to $3D$ based on the $3D$ SST. In $3D$ space, the atomic configurations become much more complicated: the rotation of crystal lattices is characterized by two sphere angles, and there are much more classes of lattices in $3D$ than $2D$ as shown in Figure \ref{fig:Bravais} including triclinic, monoclinic, orthorhombic, tetragonal, hexagonal and cubic. Finally, the proposed method is a model-based method that could work for different kinds of atomic resolution crystal images from real experiments and computer simulations, which is different to data-based methods (e.g., deep learning methods \cite{DL1,DL2,DL3,DL4,DL5}) that are sensitive to training data (e.g., neural networks trained with synthetic data might not work for experimental data). Analysis results of $3D$ experimental data are usually limited (too expensive to obtain manually) and even not available (impossible for manual measurement). Hence, the proposed algorithm in this paper could serve as a useful tool to prepare training data from real experiments for data-based approaches.
  
\begin{figure}[ht!]
  \begin{center}
    \begin{tabular}{c}
      \includegraphics[height=3in]{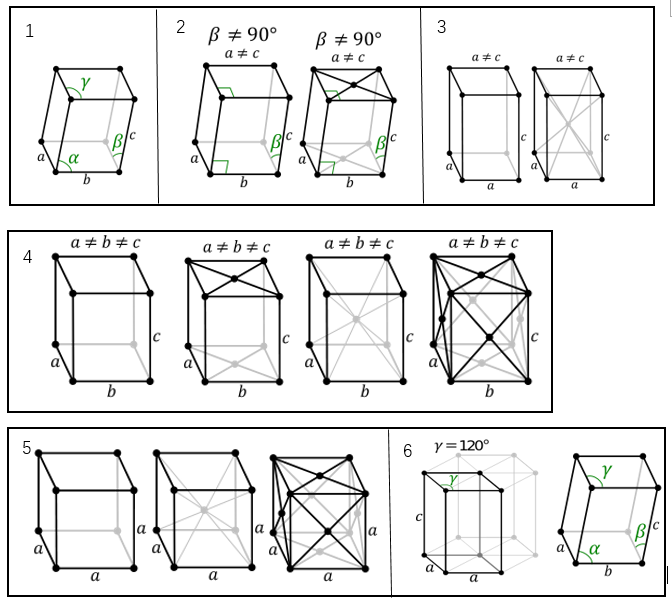}
    \end{tabular}
  \end{center}
  \caption{Six fundamental $3D$ Bravais lattices: 1 triclinic, 2 monoclinic, 3 tetragonal, 4 orthorhombic, 5 cubic, and 6 hexagonal. Courtesy of Wikipedia.}
  \label{fig:Bravais}
\end{figure}   

The rest of the paper is organized as follows. In section \ref{sec:model}, we introduce an atomic resolution crystal image model based on $3D$ general intrinsic mode type functions, and prove that the $3D$ SST is able to estimate the local properties of atomic resolution crystal images. In section \ref{sec:imp},  we present a fast $3D$ band-limited synchrosqueezed wave packet transform to detect crystal defects, estimate crystal rotations and elastic deformations. In Section \ref{sec:results}, several numerical examples of synthetic and real crystal images are provided to demonstrate the robustness and the reliability of our methods. Finally, we conclude with some discussion in Section \ref{sec:con}.

\section{Theory of $3D$ SST and atomic resolution crystal analysis}
\label{sec:model}
This section consists of three parts: In the first part, we introduce a new $3D$ SST based on $3D$ wave packet transform with a geometric scaling parameter $s\in (\frac{1}{2},1)$ to control the support sizes of wave packets. In the second part, we introduce $3D$ general intrinsic mode type functions to model atomic resolution crystal images. In the last part, we prove that the $3D$ SST is able to estimate local wave vectors of general intrinsic mode type functions accurately, providing useful information for crystal image analysis.   

\subsection{$3D$ SST}

One-dimensional SSTs are well-developed tools for empirical mode decomposition and time-frequency analysis \cite{Daubechies1996,Daubechies2011,Auger13,1DSSWPT,SSSecond,Daubechies20150193,app7080769} with better robustness in analyzing noisy signals than the short-time Fourier transform \cite{yang2018statistical,yang2015quantitative}. Two-dimensional SSTs have been proposed recently in \cite{SSWPT,SSCT,2Dwavelet}. However, $3D$ synchrosqueezed transforms have not been explored previously. Motivated by $3D$ crystal image analysis, we propose the $3D$ SST based on wave packet transforms (SSWPT) with a geometric scaling parameter $s$ as follows. From now on, we will use $n$ to denote the dimension in this section. The proposed transform and analysis later work for $n=1$, $2$, and $3$. Throughout this paper, the spatial variable would be denoted as $x$ or $b$, and the variable in the Fourier domain would be denoted as $\xi$, $a$, or $v$.

First, we introduce an $n$-dimensional mother wave packet $w(x)\in C^m(\RR^{n})$ of type $(\eps,m)$ such that $\widehat{w}(\xi)$ has an essential support in the unit ball $B_1(0)$ centered at the frequency origin with a radius $1$\footnote{In our numerical implementation, the mother wave packet has an essential support $B_d(0)$ in the Fourier domain, where $d$ is an adjustable parameter.}, i.e., 
\[
|\widehat{w}(\xi)|\leq \frac{\eps}{(1+|\xi|)^m},
\] 
for $|\xi|>1$ and some non-negative integer $m$. A family of $n$-dimensional wave packets is obtained by isotropic dilation, rotations and translations of the mother wave packet as follows, controlled by
a geometric parameter $s$.
\begin{definition}
  \label{2.2def:WA2d}
  Given the mother wave packet $w(x)$ of type $(\eps,m)$ and the parameter $s\in(1/2,1)$,
  the family of wave packets $\{w_{a b}(x): a,b\in \RR^{n}, |a|\ge 1\}$
  are defined as
  \[
  w_{a b}(x)=|a|^{ns/2} w\left(|a|^s(x-b)\right) e^{2\pi i (x-b)\cdot a},
  \]
  or equivalently in the Fourier domain
  \[
  \widehat{w_{ab}}(\xi) = |a|^{-ns/2} e^{-2\pi i b\cdot \xi}
  \widehat{w}\left(|a|^{-s}(\xi-a)\right).
  \]
\end{definition}
In this definition, we require $|a|\ge 1$. The reason is that, when
$|a|<1$, the above consideration regarding the shape of the wave
packets is no longer valid. However, since we are mostly concerned
with the high frequencies as the signals of interest here are
oscillatory, the case $|a|<1$ is essentially irrelevant.

Some properties can be seen immediately from the definition: the
Fourier transform $\widehat{w_{ab}}(\xi)$ is essentially supported in
$B_{|a|^s}\left(a\right)$, a ball centered at $a$ with a radius $|a|^s$; $w_{ab}(x)$ is centered in space at $b$ with an essential
support of width $O\left(|a|^{-s}\right)$. An $n$-dimensional SSWPT with a smaller $s$ value is better at distinguishing two intrinsic mode type functions with close propagating directions and is more robust \cite{yang2018statistical} against noise. This is the motivation to propose the wave packet transform here instead of adopting the wavelet transform corresponding to the case of $s=1$. With this family of wave packets, we define the wave packet transform as follows.
\begin{definition}
  \label{2.2def:WAT}
  The wave packet transform of a function $f(x)$ is a function
  \begin{align*}
    W_f(a,b) 
    &= \langle f,w_{a b}\rangle =  \int_{\RR^{n}} f(x) \overline{w_{a b}(x)}dx \label{2.2eqn:WAT}  \nonumber
  \end{align*}
  for $a,b\in \RR^{n}, |a|\ge 1$.
\end{definition}
If the Fourier transform $\widehat{f}(\xi)$ vanishes for $|\xi|< 1$, it is
easy to check that the $L^2$-norms of $W_f(a,b)$ and $f(x)$ are
equivalent, up to a uniform constant factor, i.e.,
\begin{equation}
  \int_{\RR^{2n}} |W_f(a,b)|^2 da db \eqsim \int_{\RR^{n}} |f(x)|^2 dx.  \label{2.2eqn:ENEEQ}
\end{equation}

\begin{definition} 
  \label{2.2def:LWV}
  The local wave vector estimation of a function $f(x)$ at
  $(a,b)\in\RR^{2n}$ is
  \begin{equation*}
    v_f(a,b)=\begin{cases}
      \frac{ \grad_b W_f(a,b) }{ 2\pi i W_f(a,b)},
      & \text{for }W_f(a,b)\neq 0;\\
      \left(\infty,\infty\right), 
      & otherwise.
    \end{cases}
\label{2.2eqn:IWE}
  \end{equation*}
\end{definition}
Given the wave vector estimation $v_f(a,b)$, the synchrosqueezing step
reallocates the information in the phase space and provides a sharpened
phase space representation of $f(x)$ in the following way.
\begin{definition}
  Given $f(x)$, the SST (or synchrosqueezed energy
  distribution), $T_f(v,b)$, is defined by
  \begin{equation*}
    T_f(v,b) = \int_{\RR^{n}\setminus B_1(0)} |W_f(a,b)|^2 \delta\left(\Re v_f(a,b)-v\right) da \label{2.2eqn:SED}
  \end{equation*}
  for $v,b\in \RR^{n}$.
\end{definition} 

As we shall see, for $f(x) = \alpha(x) e^{2\pi i N\phi(x)}$ with
a sufficiently smooth amplitude $\alpha(x)$ and a sufficiently steep phase
$N\phi(x)$, we can show that for each $b$, the estimation $v_f(a,b)$
indeed approximates $N\grad \phi(b)$ independently of $a$ as long as
$W_f(a,b)$ is non-negligible. As a direct consequence, for each $b$,
the essential support of $T_f(v,b)$ in the $v$ variable concentrates
near $N\grad \phi(b)$ (see Figure \ref{2.2fig:2} for an example). In
addition, we have the following property
\[
\int T_f(v,b) dv db =
\int |W_f(a,b)|^2 \delta(\Re v_f(a,b)-v) dv da db =
\int |W_f(a,b)|^2 dadb \eqsim \|f\|_2^2
\]
from Fubini's theorem and the norm equivalence \eqref{2.2eqn:ENEEQ}, for
any $f(x)$ with its Fourier transform vanishing for $|\xi|< 1$.
\begin{figure}[ht!]
  \begin{center}
    \begin{tabular}{cc}
      \includegraphics[height=1.6in]{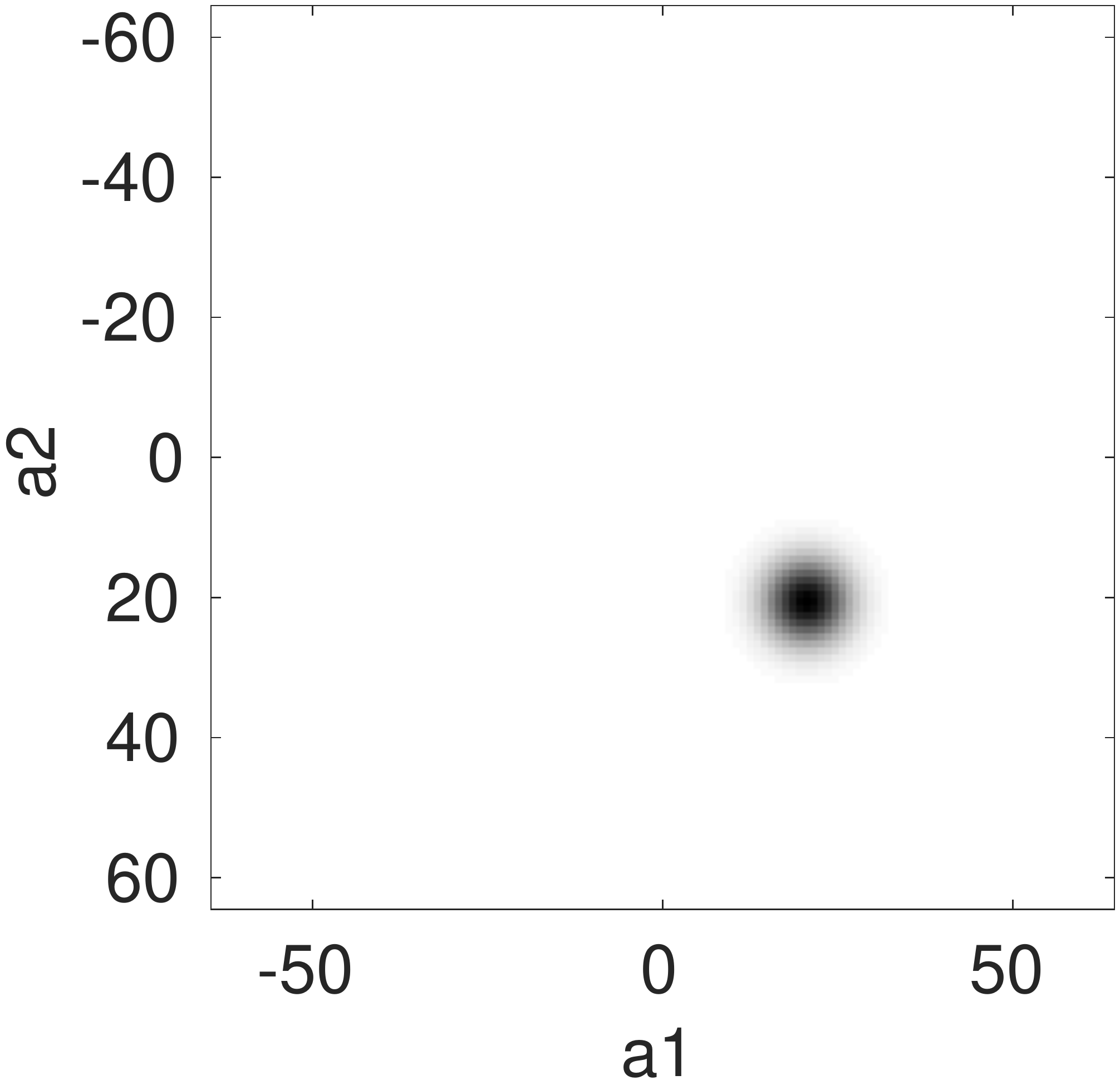}  & \includegraphics[height=1.6in]{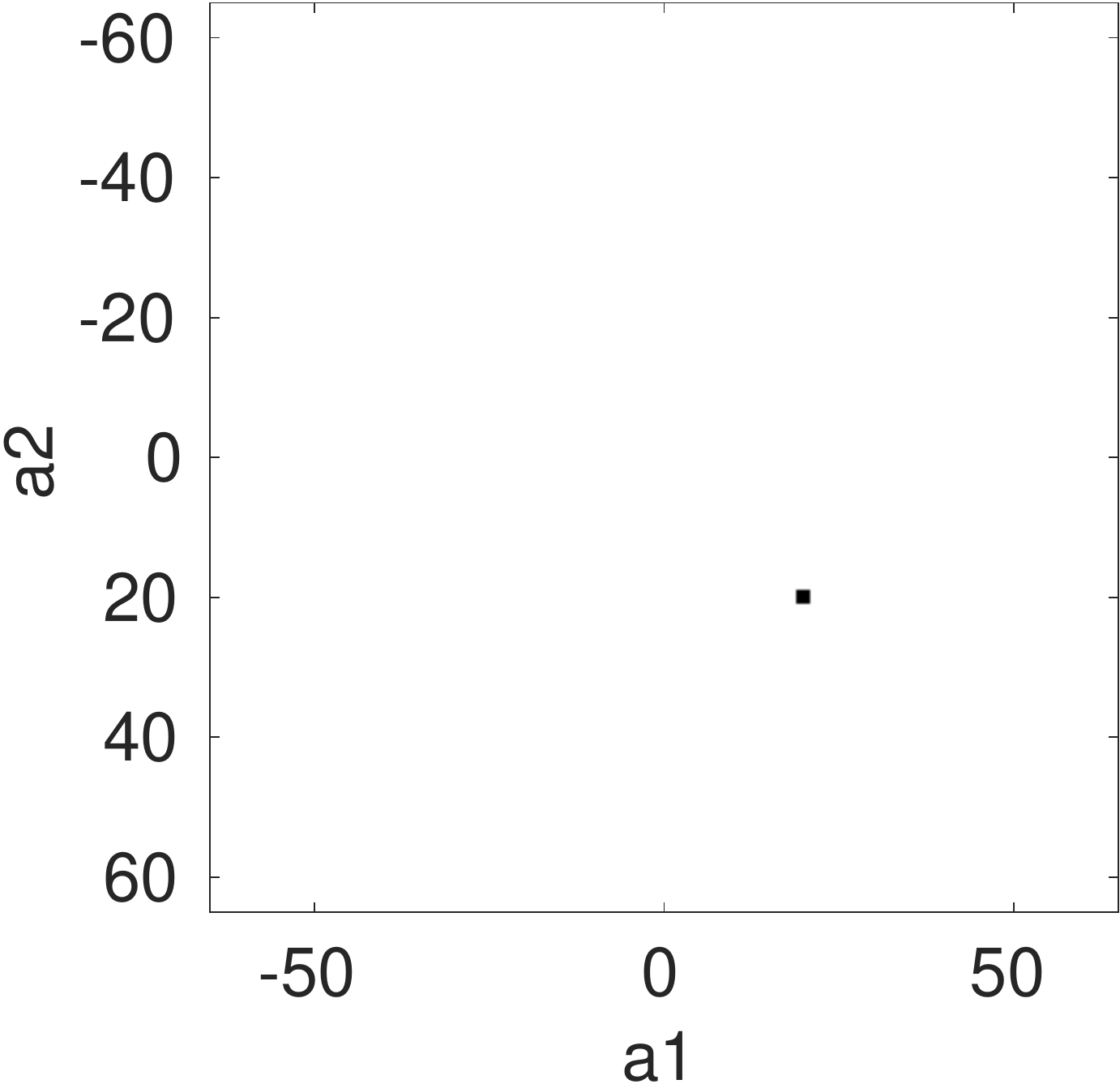}
    \end{tabular}
  \end{center}
  \caption{Suppose the data $f(x)=e^{2\pi iN(x_1+x_2+x_3)}$ with $N = 20$ for $x=(x_1,x_2,x_3)\in [0,1]^3$. Left: $|W_f(a,b)|$ at $b=(0.5,0.5,0.5)$ and $a_3=20$. Right: $T_f(a,b)$ at the same
    location $b$ and $a_3$. $|W_f(a,b)|$ has been reallocated to form a sharp phase
    space representation $T_f(a,b)$.}
  \label{2.2fig:2}
\end{figure}

Now we show that the SST can distinguish well-separated local wavevectors $\{N\nabla \phi_k(x)\}_{1\leq k\leq K}$ from a
superposition of multiple components $f(x)=\sum_{k=1}^K \alpha_k(x)e^{2\pi i N\phi_k(x)}$.

\begin{definition}
  \label{mSSWPT:def:IMTF}
  A function $f(x)=\alpha(x)e^{2\pi iN \phi(x)}$ is an intrinsic mode type
  function (IMT) of type $\left(M,N\right)$ if $\alpha(x)$ and $\phi(x)$ satisfy
  \begin{align*}
    \alpha(x)\in C^\infty, \quad |\grad \alpha(x)|\leq M, \quad 1/M \leq \alpha(x)\leq M \\
    \phi(x)\in C^\infty,  \quad  1/M \leq |\grad \phi(x)|\leq M, \quad |\grad^2 \phi(x)|\leq M.
   \end{align*}
\end{definition}
\begin{definition}
  \label{mSSWPT:def:SWSIMC}
  A function $f(x)$ is a well-separated superposition of type
  $\left(M,N,K,s\right)$ if
  \[
  f(x)=\sum_{k=1}^K f_k(x)
  \] 
  where each $f_k(x)=\alpha_k(x)e^{2\pi iN_k \phi_k(x)}$ is an IMT of type $\left(M,N_k\right)$ with $N_k\geq N$ and the phase functions satisfy the
  separation condition: for any $(a,b)\in \mathbb{R}^{2n}$, there exists at most one $f_k$ satisfying that 
  \[
\left|a\right|^{-s}\left| a-N_k\grad\phi_k(b) \right|\leq 1.
  \]
  We denote by $F\left(M,N,K,s\right)$ the set of all
  such functions.
\end{definition}

The following theorem illustrates the main results of the $n$-dimensional SST for a superposition of IMTs. 
In what follows, when we write $O\left(\cdot\right)$, $\lesssim$, or $\gtrsim$, the implicit constants may depend on $M$, $m$ and $K$.
\begin{theorem}
  \label{mSSWPT:thm:2d1}
  Suppose the n-dimensional mother wave packet is of type $(\epsilon,m)$, for any fixed $\eps\in(0,1)$ and any fixed integer $m\geq 0$.
  For a function $f(x)$, we define
  \[
  R_{\eps} = \{(a,b): |W_f(a,b)|\geq |a|^{-ns/2}\sqrt \eps\},
  \]
    \[
  S_{\eps} = \{(a,b): |W_f(a,b)|\geq \sqrt \eps\},
  \]
  and 
  \[
  Z_{k} = \{(a,b): |a-N_k\grad  \phi_k(b)|\leq |a|^s \}
  \]
  for $1\le k\le K$. For fixed $M$, $m$, and $K$ there
  exists a constant $N_0\left(M,m,K,s,\eps\right)\simeq \max\left\{ \epsilon^{\frac{-2}{2s-1}},\epsilon^{\frac{-1}{1-s}} \right\}$ such that for any
  $N>N_0$ and $f(x)\in F\left(M,N,K,s\right)$ the following statements
  hold.
  \begin{enumerate}
  \item $\{Z_{k}: 1\le k \le K\}$ are disjoint and $S_{\eps}\subset R_{\eps}
    \subset \bigcup_{1\le k \le K} Z_{k}$;
  \item For any $(a,b) \in R_{\eps} \cap Z_{k}$, 
    \[
    \frac{|v_f(a,b)-N_k\grad \phi_k(b)|}{ |N_k \grad \phi_k(b)|}\lesssim\sqrt \eps;
    \]
     \item For any $(a,b) \in S_{\eps} \cap Z_{k}$, 
    \[
    \frac{|v_f(a,b)-N_k\grad \phi_k(b)|}{ |N_k \grad \phi_k(b)|}\lesssim N_k^{-ns/2}\sqrt \eps.
    \]
  \end{enumerate}
\end{theorem}
\begin{lemma}
  \label{mSSWPT:lem:A2}
  Suppose $\Omega_a=\{k:|a|\in[\frac{N_k}{2M},2MN_k]\}$. Under the assumption of Theorem \ref{mSSWPT:thm:2d1}, we have
  \begin{equation*}
    \label{E1}
    W_f(a,b)=
      |a|^{-ns/2} \left( \sum_{k\in\Omega_a}\alpha_k(b)e^{2\pi iN_k \phi_k(b)}\widehat{w}\left(|a|^{-s}\left(a-N_k\grad \phi_k(b)\right)\right)
      +O\left(\eps\right) \right),
  \end{equation*}
  when $N>N_0\left(M,m,K,s,\eps\right)\simeq \max\left\{ \epsilon^{\frac{-2}{2s-1}},\epsilon^{\frac{-1}{1-s}} \right\}$.
\end{lemma}
\begin{proof}
  Let us first estimate $W_f(a,b)$ assuming that 
  $f(x)$ contains a single intrinsic mode function of type $(M,N)$
  \[
  f(x)=\alpha(x)e^{2\pi i N \phi(x)}.
  \]
  Using the definition of the wave packet transform, we have the
  following expression for $W_f(a,b)$.
  \begin{eqnarray*}
    W_f(a,b) &=& \int \alpha(x)e^{2\pi iN \phi(x)}|a|^{ns/2}w(|a|^s(x-b))e^{-2\pi i(x-b)\cdot a}dx\\
    &=&\int \alpha(b+|a|^{-s}y)e^{2\pi iN \phi(b+|a|^{-s}y)}|a|^{ns/2}w(y)e^{-2\pi i |a|^{-s} y\cdot a}d(|a|^{-s}y)\\
    &=&|a|^{-ns/2}\int \alpha(b+|a|^{-s}y)w(y) e^{2\pi i (N \phi(b+|a|^{-s}y)-|a|^{-s}y\cdot a)} dy.
  \end{eqnarray*}
  We claim that when $N$ is sufficiently large
  \begin{equation}
    W_f(a,b)=
    \begin{cases}
      |a|^{-ns/2} O(\eps),
      & |a|\notin[\frac{N}{2M},2MN]\\
      |a|^{-ns/2} \left(\alpha(b)e^{2\pi iN \phi(b)}\widehat{w}\left(|a|^{-s}(a-N\grad \phi(b))\right)+O(\eps) \right), 
      & |a|\in[\frac{N}{2M},2MN].
    \end{cases}
    \label{E1single}
  \end{equation}
  First, let us consider the case
  $|a|\notin[\frac{N}{2M},2MN]$. Consider the integral
  \[
  \int h(y) e^{i g(y)} dy 
  \]
  for smooth real functions $h(y)$ and $g(y)$, along with the
  differential operator
  \[
  L =\frac{1}{i}\frac{\langle\grad g,\grad\rangle}{ |\grad g|^2}.
  \]
  If $|\grad g|$ does not vanish, we have
  \[
  L e^{ig} = \frac{\langle \grad g, i\grad g e^{ig} \rangle}{i |\grad g|^2} = e^{i g}.
  \]
  Assuming that $h(y)$ decays sufficiently fast at infinity, we
  perform integration by parts $r$ times to get
  \[
  \int h e^{ig} dy = \int h (L^r e^{i g}) dy = \int ((L^*)^r h) e^{i g} dy,
  \]
  where $L^*$ is the adjoint of $L$. In the current setting,
  $W_f(a,b) = |a|^{-ns/2} \int h(y) e^{i g(y)} dy$ with
  \[
  h(y)=\alpha (b+|a|^{-s}y)w(y),\quad 
  g(y)=2\pi ( N \phi(b+|a|^{-s}y)-|a|^{-s} y\cdot a ),
  \]
  where $h(y)$ clearly decays rapidly at infinity since $w(y)$ is in
  the Schwartz class. In order to understand the impact of $L$ and
  $L^*$, we need to bound the norm of
  \[
  \grad g(y)=2\pi \left(N\grad \phi(b+|a|^{-s}y)-a\right) |a|^{-s} 
  \]
  from below when $|a|\notin[\frac{N}{2M},2MN]$. If $|a|<\frac{N}{2M}$,
  then
  \[
  |\grad g|\gtrsim(|N\grad\phi|-|a|)|a|^{-s}\gtrsim |N\grad\phi||a|^{-s}/2 \gtrsim N^{1-s}.
  \]
  If $|a|>2MN$, then
  \[
  |\grad g|\gtrsim(|a|-|N\grad\phi|)|a|^{-s}\gtrsim |a|\cdot|a|^{-s}/2 \gtrsim(|a|)^{1-s} \gtrsim
  N^{1-s}.
  \]
  Hence $|\grad g|\gtrsim N^{1-s}$ if $|a|\notin[\frac{N}{2M},2MN]$.
  Since $|\grad g|\not=0$ and each $L^*$ contributes a factor of order
  $1/|\grad g|$
  \[
  \left| \int e^{ig(y)} ((L^*)^r h)(y) dy \right| \lesssim N^{-(1-s)r}.
  \]
  When 
  \begin{equation}
  \label{2.2eqn:NR3}
  N \gtrsim \eps^{-1/((1-s)r)},
  \end{equation}
   we obtain
  \[
  \left| \int e^{ig(y)} ((L^*)^r h)(y) dy \right| \lesssim \eps.
  \]  
  Using the fact $W_f(a,b) = |a|^{-ns/2} \int h(y) e^{i g(y)} dy$, we
  have $|W_f(a,b)| \lesssim |a|^{-ns/2} \eps$.
  
  Second, let us address the case $|a|\in[\frac{N}{2M},2MN]$. We
  want to approximate $W_f(a,b)$ with
  \[
  |a|^{-ns/2} \alpha(b)e^{2\pi iN \phi(x)}\widehat{w}\left(|a|^{-s}(a-N\grad \phi(b))\right).
  \]
  Since $w(y)$ is in the Schwartz class, we can assume that
  $|w(y)|\leq\frac{C_\mm}{|y|^\mm}$ for some sufficient large $\mm$ with
  $C_\mm$ for $|y|\ge 1$. Therefore, the integration over
  $|y|\gtrsim\eps^{-1/\mm}$ yields a contribution of at most order
  $O(\eps)$. We can then estimate
  \[
  |W_f(a,b)|=|a|^{-ns/2}
  \left(\int_{|y|\lesssim\eps^{-1/\mm}} \alpha(b+|a|^{-s}y)w(y)e^{2\pi i (N \phi(b+|a|^{-s}y)-|a|^{-s} y\cdot a )}dy + O(\eps)\right).
  \]
  A Taylor expansion of $\alpha(x)$ and $\phi(x)$ shows that
  \[
  \alpha(b+|a|^{-s}y)=\alpha(b)+\grad \alpha(b^*) \cdot |a|^{-s}y
  \]
  and
  \[
  \phi(b+|a|^{-s}y)= \phi(b) + \grad \phi(b)\cdot (|a|^{-s}y) + \frac{1}{2} (|a|^{-s}y)^t \grad^2 \phi(b^*) (|a|^{-s}y),
  \]
  where in each case $b^*$ is a point between $b$ and
  $b+|a|^{-s}y$. We want to drop the last term from the above formulas
  without introducing a relative error larger than $O(\eps)$. We begin
  with the estimate
  \[
  \int_{|y|\lesssim\eps^{-1/\mm}}|\grad \alpha\cdot |a|^{-s}yw(y)|dy\lesssim\eps,
  \]
  which holds if $\eps^{-n/\mm}|\grad \alpha\cdot |a|^{-s}y|
  \lesssim\eps$, which is true when
  $|a|^{-s}\lesssim\eps^{1+(n+1)/\mm}$. Since $|a|\in[\frac{N}{2M},2MN]$,
  the above holds if
  \begin{equation}
    N\gtrsim \eps^{-(1+(n+1)/\mm)/s}. \label{2.2eqn:NR1}
  \end{equation}
  We also need
  \[
  \int_{|y|\lesssim\eps^{-1/\mm}} 
  | \alpha(b)w(y) e^{2\pi i (N\phi(b)+N\grad \phi(b)\cdot |a|^{-s}y-|a|^{-s}y\cdot a )}|
  \cdot | e^{2\pi i N/2 (|a|^{-s}y)^t \grad^2\phi (|a|^{-s}y)}-1
  |dy\lesssim\eps.
  \]
  Since $|e^{ix}-1|\leq|x|$, the above inequality is equivalent to
  \[
  \int_{|y|\lesssim\eps^{-1/\mm}} \alpha(b)w(y)e^{2\pi i (N \phi(b)+N\grad \phi(b)\cdot |a|^{-s}y-|a|^{-s}y\cdot a )} 
  | 2\pi N/2(|a|^{-s}y)^t \grad^2\phi (|a|^{-s}y) |dy\lesssim\eps,
  \]
  which is true if $\eps^{-n/\mm} N(|a|^{-s}y)^t \grad^2\phi
  (|a|^{-s}y)\lesssim \eps$, which in turn holds if $N |a|^{-2s} |y|^2
  \lesssim \eps^{1+n/\mm}$.  Because $|y|\lesssim\eps^{-\frac{1}{\mm}}$
  and $|a|\in[\frac{N}{2M},2MN]$, the above inequality is valid when
  \begin{equation}
    N\gtrsim \eps^{-(1+(n+2)/\mm)/(2s-1)}.  \label{2.2eqn:NR2}
  \end{equation}
  In summary, for $N$ larger than the maximum of the right hand sides
  of \eqref{2.2eqn:NR3}, \eqref{2.2eqn:NR1} and \eqref{2.2eqn:NR2}, if $|a|\in[\frac{N}{2M},2MN]$
  then we have
  \begin{align*}
    W_f(a,b)&=
    |a|^{-ns/2}\left(\int_{|y|\lesssim\eps^{-1/\mm}}\alpha(b)w(y)e^{2\pi i (N \phi(b)+N\grad \phi(b)\cdot |a|^{-s}y-|a|^{-s}y\cdot a)}dy + O(\eps)\right)\\
    &=|a|^{-ns/2}\left(\int_{|y|\lesssim\eps^{-1/\mm}}\left(\alpha(b)e^{2\pi iN \phi(b)}\right)
    w(y) e^{2\pi i (N\grad \phi(b) - a)\cdot |a|^{-s} y} dy + O(\eps)\right)\\
    &=|a|^{-ns/2}\left(\int_{\mathbb{R}^{n}}             \left(\alpha(b)e^{2\pi iN \phi(b)}\right) 
    w(y) e^{2\pi i (N\grad \phi(b) - a)\cdot |a|^{-s} y} dy + O(\eps)\right)\\
    &=|a|^{-ns/2}\left(\alpha(b)e^{2\pi iN \phi(b)}\widehat{w}\left(|a|^{-s}(a-N\grad \phi(b))\right)+O(\eps)\right),
  \end{align*}
  where the third line uses the fact that the integration of $w(y)$
  outside the set $\{y: |y|\lesssim\eps^{-1/\mm} \}$ is again of order
  $O(\eps)$.  

  Now let us return to the general case, where $f(x)$ is a
  superposition of $K$ well-separated intrinsic mode components:
  \[
  f(x)=\sum_{k=1}^K f_k(x)=\sum_{k=1}^K \alpha_k(x)e^{2\pi iN_k
    \phi_k(x)}.
  \]
  By linearity of the wave packet
  transform and \eqref{E1single}, we find:
   \[
   W_f(a,b)= |a|^{-ns/2}\left(\sum_{k\in \Omega_a=\{k:|a|\in[\frac{N_k}{2M},2MN_k]\}} \alpha_k(b)e^{2\pi iN \phi_k(b)}
      \widehat{w}\left(|a|^{-s}(a-N\grad \phi_k(b))\right)+O(\eps)\right).
      \]
\end{proof}
The next lemma estimates $\grad_b W_f(a,b)$ when
$\Omega_a$ is not empty, i.e., the case where $W_f(a,b)$ is
non-negligible.

\begin{lemma}
  \label{mSSWPT:lem:B2}
  Suppose $\Omega_a=\{k:|a|\in[\frac{N_k}{2M},2MN_k]\}$ is not empty. Under the assumption of Theorem \ref{mSSWPT:thm:2d1}, we have
  \begin{equation*}
    \grad_b W_f(a,b)=
    2\pi i |a|^{-ns/2}
    \left(\sum_{k\in\Omega_a} N_k\grad\phi_k(b)\alpha_k(b)e^{2\pi i N_k \phi_k(b)}\widehat{w}\left(|a|^{-s}\left(a-N_k\grad\phi_k(b)\right)\right) +|a|O\left(\eps\right)\right),
  \end{equation*}
  when $N>N_0\left(M,m,K,s,\eps\right)\simeq \max\left\{ \epsilon^{\frac{-2}{2s-1}},\epsilon^{\frac{-1}{1-s}} \right\}$.
\end{lemma}

\begin{proof}
  The proof is similar to that of Lemma \ref{mSSWPT:lem:A2}. Assume that
  $f(x)$ contains a single intrinsic mode function, i.e.,
  \[
  f(x)=\alpha(x)e^{2\pi iN \phi(x)},
  \]
  then
  \begin{align*}
    \grad_b W_f(a,b) =&
    \int_{\mathbb{R}^{n}}\alpha(x)e^{2\pi i N \phi(x)}|a|^{ns/2}\left(\grad w(|a|^s(x-b))(-|a|^s)+2\pi ipw(|a|^s(x-b))\right)
    e^{-2\pi i(x-b)\cdot a} dx\\
    =& \int_{\mathbb{R}^{n}}\alpha(b+|a|^{-s}y)e^{2\pi iN \phi(b+|a|^{-s}y)}|a|^{-ns/2}\grad w(y)(-|a|^s) e^{-2\pi i |a|^{-s} y \cdot a}dy\\
    &+\int_{\mathbb{R}^{n}}\alpha(b+|a|^{-s}y)e^{2\pi iN \phi(b+|a|^{-s}y)}|a|^{-ns/2} 2\pi i a w(y) e^{-2\pi i |a|^{-s} y\cdot a}dy.
  \end{align*}
  Forming a Taylor expansion and following the same argument as in the
  proof of Lemma \ref{mSSWPT:lem:A2} gives the following approximation for
  $|a|\in[\frac{N}{2M},2MN]$
  \begin{align*}
    \grad_b W_f(a,b) =& 
    \left(-2\pi i |a|^{-ns/2} (a-N\grad \phi(b)) \alpha(b)e^{2\pi iN \phi(b)} \widehat{w}(|a|^{-s}(a-N\grad \phi(b)))+O(\eps)\right)\\
    &+2\pi i|a|^{-ns/2}a\left(\alpha(b)e^{2\pi iN \phi(b)}\widehat{w}\left(|a|^{-s}(a-N\grad \phi(b))\right)+O(\eps)\right)\\
    =& 2\pi i |a|^{-ns/2}\left(N\grad\phi(b) \alpha(b)e^{2\pi i N \phi(b)} \widehat{w}(|a|^{-s}(a-N\grad \phi(b)))+|a|O(\eps) \right).
  \end{align*}
  For $f(x)=\sum_{k=1}^K f_k(x) = \sum_{k=1}^K \alpha_k(x)e^{2\pi iN
    \phi_k(x)}$, taking sum over $K$ terms gives
  \[
  \grad_b W_f(a,b)=2\pi i  |a|^{-ns/2} \left(\sum_{k\in\Omega_a}\left(N\grad\phi_k(b) \alpha_k(x)e^{2\pi i N \phi_k(b)} 
  \widehat{w}(|a|^{-s}(a-N\grad \phi_k(b)))\right)+|a|O(\eps)\right)
  \]
  for $|a|\in[\frac{N}{2M},2MN]$.
\end{proof}

We are now ready to prove the theorem.
\begin{proof}

  For $(i)$, the well-separation condition implies that $\{Z_{k}:1\leq k\leq K\}$ are disjoint.

  Let $(a,b)$ be a point in $R_{\eps} = \{(a,b): |W_f(a,b)| \ge
  |a|^{-ns/2} \sqrt{\eps} \}$.  From the above lemma, we have
  \[ 
  W_f(a,b) = |a|^{-ns/2}\left(\sum_{k\in\Omega_a} \alpha_k(b)e^{2\pi iN_k
    \phi_k(b)}\widehat{w}\left(|a|^{-s}(a-N_k\grad
  \phi_k(b))\right)+O(\eps)\right).
  \]
  Therefore, there exists $k$ between $1$ and $K$ such that
  $\widehat{w}\left(|a|^{-s}(a-N_k\grad \phi_k(b))\right)$ is non-zero. From
  the definition of $\widehat{w}(\xi)$, we see that this implies $(a,b) \in
  Z_{k}$. Hence $R_{\eps} \subset \bigcup_{k=1}^K Z_{k}$. It's obvious that $S_\eps\subset R_\eps$.
  
  To show $(ii)$, let us recall that $v_f(a,b)$ is defined as
  \[
  v_f(a,b) =        \frac{ \grad_b W_f(a,b) }{2\pi iW_f(a,b)  }
  \]
  for $W_f(a,b)\neq 0$. If $(a,b) \in R_{\eps} \cap
  Z_{k}$, then
  \[
  W_f(a,b)=|a|^{-ns/2}\left(\alpha_k(b)e^{2\pi iN_k \phi_k(b)}\widehat{w}\left(|a|^{-s}(a-N_k\grad \phi(b))\right)+O(\eps)\right)
  \]
  and
  \[
  \grad_b W_f(a,b) =2\pi i |a|^{-ns/2}
  \left(N_k\grad\phi_k(b)\alpha_k(b)e^{2\pi iN_k \phi_k(b)} \widehat{w}(|a|^{-s}(a-N_k\grad \phi_k(b)))+|a|O(\eps)\right)
  \]
  as the other terms drop out since $\{Z_{k}\}$ are disjoint. Hence
  \[
  v_f(a,b)=\frac{N_k\grad
    \phi_k(b) \left(\alpha_k(b)e^{2\pi iN_k
      \phi_k(b)}\widehat{w}\left(|a|^{-s}(a-N_k\grad
    \phi_k(b))\right)+O(\eps)\right)}
  { \left(\alpha_k(b)e^{2\pi
      iN_k \phi_k(b)}\widehat{w}\left(|a|^{-s}(a-N_k\grad
    \phi_k(b))\right)+O(\eps)\right)}.
  \]
  Let us denote the term $\alpha_k(b)e^{2\pi iN_k
    \phi_k(b)}\widehat{w}\left(|a|^{-s}(a-N_k\grad \phi_k(b))\right)$ by $g$. Then
  \[
  v_f(a,b)=\frac{N_k\grad \phi_k(b)\left(g+O(\eps)\right)}{ g+O(\eps)}.
  \]
  Since $|W_f(a,b)|\geq |a|^{-ns/2}\sqrt{\eps}$ for $(a,b)\in
  R_{\eps}$, $|g|\gtrsim\sqrt{\eps}$, and therefore
  \[
  \frac{|v_f(a,b)-N_k\grad \phi_k(b)|}{ |N_k\grad \phi_k(b)|}
  \lesssim \left|\frac{O(\eps)}{g+O(\eps)}\right|\lesssim \sqrt{\eps}.
  \]
  
   Similarly, if $(a,b)\in S_\eps\cap Z_k$, then 
   \[
  \frac{|v_f(a,b)-N_k\nabla\phi_k(b)|}{ |N_k \nabla\phi_k(b)|}
  \lesssim \left|\frac{O(\eps)}{g+O(\eps)}\right|\lesssim \frac{\sqrt{\eps}}{N_k^{ns/2}},
  \]
  since $|g|\gtrsim N_k^{ns/2}\sqrt{\eps}$ for $(a,b)\in S_\eps\cap Z_k$.
\end{proof}
In the next section, we will show that an atomic resolution crystal image can be considered as a superposition of multiple components $f(x)=\sum_{k=1}^K \alpha_k(x)e^{2\pi i N_k\phi_k(x)}$ and hence Theorem \ref{mSSWPT:thm:2d1} can be applied to analyze crystal images.

\subsection{Mathematical models for $3D$ atomic resolution crystal image}
 In this paper, we assume that the lattice type is known. In practical applications, the SST of crystal images can provide important features for crystal classification following the approach in \cite{Lu2018}. Without loss of generality, we assume the known lattice type is cubic. Consider an image of a polycrystalline material with atomic resolution.
Denote the perfect reference lattice as 
\begin{equation*}
\mathcal L=\{av_1+bv_2 +c v_3\,:\,a,\text{ }b,\text{ }c\text{ are integers}\}\,,
\end{equation*}
where $v_1$, $v_2$, and $v_3\in\RR^3$ represent three fixed lattice vectors. Let $S(2\pi Fx)$ be the image corresponding to a single perfect unit cell, extended periodically in $x$ with respect to the reference crystal lattice, where $F$ is an affine transform determined by the lattice type. $F$ is an identity matrix in the case of the cubic lattice here. We denote by $\Omega$ the domain occupied by the whole
image and by $\Omega_k$, $k = 1, \ldots, M$, the grains the system consists of.
Now we model a polycrystal image $f:\Omega\to\RR$ as
\begin{equation}
f(x)= \alpha_k(x) S(2\pi N \phi_k(x))+c_k(x) \qquad \text{if } x \in \Omega_k,
\label{eqcrystal}
\end{equation}
where $N$ is the reciprocal lattice parameter (or rather the relative
reciprocal lattice parameter as we will normalize the dimension of the
image).  The $\phi_k:\Omega_k \to \RR^3$ is chosen to map the atoms of
grain $\Omega_k$ back to the configuration of a perfect
crystal, in other words, it can be thought of as the inverse of the
elastic displacement field. The local inverse deformation gradient is
then given by $G_k = \nabla \phi_k$ in each $\Omega_k$.  
Possible variation of intensity and illumination may occur during the imaging process, leading to the smooth amplitude envelop $\alpha_k(x)$ and the smooth trend function $c_k(x)$ in \eqref{eqcrystal}.  By the
$3D$ Fourier series $\hat S$ of $S$ and the indicator functions
$\chi_{\Omega_k}$, we can rewrite \eqref{eqcrystal} as
\begin{equation}
\begin{aligned}
  f(x) =\sum_{k=1}^M \chi_{\Omega_k} (x)\left( \sum_{n \in
      \mc{L}^{\ast}}\alpha_k(x) \widehat{S}(n)e^{2\pi iN n \cdot
      \phi_k(x)}+c_k(x)\right),
\label{eqn:imagefunction}
\end{aligned}
\end{equation}
where $\mc{L}^{\ast}$ is the reciprocal lattice of $\mc{L}$ (recall that $S$ is
periodic with respect to the lattice $\mc{L}$). In each grain
$\Omega_k$, the image is a superposition of wave-like components
$\alpha_k(x) \widehat{S}(n) e^{2\pi i N n \cdot \phi_k(x)}$ with
local wave vectors $N \nabla(n \cdot \phi_k(x))$ and local amplitude
$\alpha_k(x)|\widehat{S}(n)|$.  

Our goal here is to apply the
$3D$ SST to estimate the
defect region and also $\nabla \phi_k$ in the interior of each grain $\Omega_k$. Grain boundaries are interpreted as $\cup \partial\Omega_k$ (in real crystal images, the grain boundaries would be a thin transition region instead of a sharp boundary $\cup \partial \Omega_k$. In the presence of local defects, e.g., an isolated defect and a terminating line of defects, $\cup\partial\Omega_k$ may include irregular boundaries and may contain point boundaries inside $\cup\Omega_k$.

\subsection{SST for crystal image analysis}

In this section, we will show that the $3D$ SST introduced previously can estimate well-separated local wavevectors from a superposition of multiple components in \eqref{eqn:imagefunction}.

\begin{definition}[$3D$ general shape function]
\label{def:GSF}

The $3D$ general shape function class ${S}_M$ consists of periodic functions $S(x)$ with a periodicity $(2\pi,2\pi,2\pi)$, a unit $L^{2}([-\pi,\pi]^{3})$-norm, and an $L^\infty$-norm bounded by $M$ satisfying the following conditions:
\begin{enumerate}
\item The $3D$  Fourier series of $S(x)$ is uniformly convergent;
\item $\sum_{n \in \ZZ^{3}} |\widehat{S}(n)|\leq M$ and  $\widehat{S}(0)=0$;
\item Let $\Lambda_i:=\{|n_i|\in \NN: n_i\neq 0,\text{ }\widehat{S}(n_1,n_2,n_3)\neq 0 \text{ for } n_1,\text{ }n_2,\text{ }n_3\in\ZZ\}$ for $i=1$, $2$, and $3$, and $\lambda_i$ be the greatest common divisor of all the elements in $\Lambda_i$. Then the greatest common divisor of $\{\lambda_i\}_{1\leq i\leq 3}$ is $1$.
\end{enumerate}
\end{definition}
\begin{definition}[$3D$ general intrinsic mode type function (GIMT)]
  \label{def:GIMTF}
  A function $f(x)=\alpha(x)S(2\pi N  \phi(x))$ is a $3D$  GIMT of type $(M,N)$, if $S(x)\in {S}_M$, $\alpha(x)$ and $\phi(x)$ satisfy the conditions below.
\begin{align*} 
&\alpha(x)\in C^\infty, \quad |\grad\alpha|\leq M, \quad 1/M \leq \alpha\leq M, \\
 &\phi(x)\in C^\infty, \quad 1/M \leq \left|  \nabla (n^{\TT}  \phi) / \abs{n^{\TT} } \right| \leq M,\quad \text{and}\\
  &\left| \nabla (n^{\TT}  \phi) / \abs{n^{\TT}} \right|\leq M, \quad \forall n \in \ZZ^{3} \quad \text{s.t.}\quad \widehat{S}(n)\neq 0.
   \end{align*}
\end{definition}

The following theorem illustrates the main results of the $3D$ SST for a superposition of  GIMTs. 
\begin{theorem}
  \label{mSSWPT:thm:2d2}
  Suppose the $3D$ mother wave packet is of type $(\epsilon,m)$, for any fixed $\eps\in(0,1)$ and any fixed integer $m\geq 0$.
  For a function $f(x)$, we define
  \[
  R_{\eps} = \{(a,b): |W_f(a,b)|\geq |a|^{-3s/2}\sqrt \eps\},
  \]
    \[
  S_{\eps} = \{(a,b): |W_f(a,b)|\geq \sqrt \eps\},
  \]
  and 
  \[
  Z_{n} = \{(a,b): |a-N\grad (n \phi(b))|\leq |a|^s \}.
  \]
  For fixed $M$, $m$, $s$, and $\epsilon$,  there
  exists a constant $N_0\left(M,m,s,\eps\right)\simeq \max\left\{ \epsilon^{\frac{-2}{2s-1}},\epsilon^{\frac{-1}{1-s}} \right\}$ such that for any
  $N>N_0$ and a $3D$  GIMT $f(x)=\alpha(x)S(2\pi N  \phi(x))$ of type (M,N) the following statements
  hold for $n=O(1)$.
  \begin{enumerate}
  \item $\left\{Z_{n}:\widehat{S}(n)\neq 0 \right\}$ are disjoint and $S_{\eps}\subset R_{\eps}
    \subset \bigcup_{\widehat{S}(n)\neq 0} Z_{n}$;
  \item For any $(a,b) \in R_{\eps} \cap Z_{n}$, 
    \[
    \frac{|v_f(a,b)-N\grad (n\phi(b))|}{ |N \grad (n\phi(b))|}\lesssim\sqrt \eps;
    \]
     \item For any $(a,b) \in S_{\eps} \cap Z_{n}$, 
    \[
    \frac{|v_f(a,b)-N\grad (n\phi(b))|}{ |N \grad (n\phi(b))|}\lesssim N^{-3s/2}\sqrt \eps.
    \]
  \end{enumerate}
\end{theorem}
The proof of Theorem \ref{mSSWPT:thm:2d2} is similar to that of Theorem \ref{mSSWPT:thm:2d1} and requires Lemmas \ref{mSSWPT:lem:A2} and \ref{mSSWPT:lem:B2}. 
%

\begin{proof}
By the uniform convergence of the $3D$ Fourier series of general shape functions, we have
\begin{equation*}
W_f(a,b)= \sum_{n\in \ZZ^3}W_{f_{n}}(a,b),
  \end{equation*}
where $f_{n}(x)=\widehat{S}(n)\alpha(x)e^{2\pi iN n\cdot\phi(x)}$. 
Introduce the short hand notation,  $\wt{\phi}_{n}(x)=n\cdot \phi(x) / \abs{n}$, then 
\begin{equation*}
  f_{n}(x)=\widehat{S}(n)\alpha(x)e^{2\pi iN \abs{n} \wt{\phi}_{n}(x)}.
\end{equation*}
By the property of $3D$ general intrinsic mode functions, $f_{n}(x)$ is a well-separated superposition of type $(M,N\abs{n},1,s)$ defined in Definition \ref{mSSWPT:def:SWSIMC}.

For each $n$, we estimate $W_{f_{n}}(a,b)$. By Lemma \ref{mSSWPT:lem:A2}, there exists a uniform $N_1(M,m,1,s,\eps)$ independent of $n$ such that, if $N\abs{n}>N_1$, 
\begin{equation*}
W_{f_{n}}(a,b)=|a|^{-\frac{3s}{2}} \left( f_{n}(b)\widehat{w}\left(|a|^{-s} \left(a-N \abs{n} \grad\wt{\phi}_{n}(b)\right)\right)+\left| \widehat{S}(n)\right|O(\eps) \right)
\end{equation*}
for $|a|\in[\frac{N|n|}{2M},2MN|n|]$, and
\begin{equation*}
W_{f_{n}}(a,b)=|a|^{-\frac{3s}{2}}\left| \widehat{S}(n)\right|O(\eps),
\end{equation*}
for $|a|\notin[\frac{N|n|}{2M},2MN|n|]$. 
%

For $(1)$, notice that for any $n\neq \widehat{n}$, the distant between wave vectors $N|n|\grad \tilde{\phi}_n(b)$ and $N|\widehat{n}| \grad \tilde{\phi}_{\widehat{n}}(b)$ are bounded below. In fact
  \[
    |N|n|\grad \tilde{\phi}_n(b)-N|\widehat{n}| \grad \tilde{\phi}_{\widehat{n}}(b)|= |N(n-\widehat{n})\grad \phi(b)| \\
    \geqslant  \frac{N}{M}|n-\widehat{n}|\geqslant \frac{N}{M}.
  \]
The first inequality above is due to the definition of $3D$  mode type functions. Observe that the support of a wave packet centered at $a$ is within a disk with a radius of length $|a|^s$. Because of the range of $a$ of interest is $|a|\leqslant2MN|n|$, where $|n|=O(1)$, the wave packets of interest have supports of size at most of $(2MN|n|)^s$. Hence, if $\frac{N}{M}\geqslant (2MN|n|)^s$, which is equivalent to $N\geqslant(2^{s}M^{1+s}|n|^s)^{\frac{1}{1-s}}=O(1)$, then for each $(a,b)$ of interest, there is at most one $n\in\ZZ^{3}$ such that
\[
   |a-N|n|\grad \tilde{\phi}_n(b))|\leq |a|^s 
\] 
This implies that $ \left\lbrace \ZZ_n\right\rbrace $ are disjoint sets. Notice that $\widehat{w}(x)$ decays to $O(\epsilon)$ when $|x|\geqslant 1$. The above statement also indicates that there is at most one $n\in \left\lbrace n: \widehat{S}(n) \neq 0\right\rbrace$ such that
\[
f_n(b)\widehat{w}(|a|^{-s}(a-N|n|\grad \tilde{\phi}_n(b))) \neq 0.
\]
Hence, if $(a,b)\in \RR_{\varepsilon}$ there must be some $n$ such that $\widehat{S}(n) \neq 0$ and
\begin{equation}
 W_f(a,b)= |a|^{-3s/2}\left(f_n(b)
      \widehat{w}\left(|a|^{-s}(a-N|n|\grad \tilde{\phi}_n(b))\right)+O(\eps)\right),
\end{equation}
by Lemma \ref{mSSWPT:lem:A2}. By the definition of $\ZZ_n$, we see $(a,b)\in \ZZ_n$. So, $S_{\varepsilon}\subset R_{\varepsilon}\subset\bigcup_{\widehat{S}(n)\neq 0}Z_n$, and $(1)$  is proved.

  To show $(2)$, let us recall that $v_f(a,b)$ is defined as
  \[
  v_f(a,b) =        \frac{ \grad_b W_f(a,b) }{2\pi iW_f(a,b)  }
  \]
  for $W_f(a,b)\neq 0$. If $(a,b) \in R_{\eps} \cap
  Z_{n}$, then by Lemmas \ref{mSSWPT:lem:A2}, \ref{mSSWPT:lem:B2}, and the above discussion,
  \[
  W_f(a,b)=|a|^{-3s/2}\left(f_n(b)\widehat{w}\left(|a|^{-s}(a-N|n|\grad \tilde{\phi}_n(b))\right)+\widehat{S}(n)O(\eps)\right)
  \]
  and
  \[
  \grad_b W_f(a,b) =2\pi i |a|^{-3s/2}
  \left(N|n|\grad \tilde{\phi}_n(b) f_n(b) \widehat{w}(|a|^{-s}(a-N|n|\grad \tilde{\phi}_n(b)))+|a|\widehat{S}(n)O(\eps)\right)
  \]
  as the other terms drop out since $\{Z_{n}\}$ are disjoint. Hence
  \[
  v_f(a,b)=\frac{N|n|\grad \tilde{\phi}_n(b) \left(f_n(b)\widehat{w}\left(|a|^{-s}(a-N|n|\grad \tilde{\phi}_n(b))\right)+\widehat{S}(n)O(\eps)\right)}
  { \left(f_n(b)\widehat{w}\left(|a|^{-s}(a-N|n|\grad \tilde{\phi}_n(b))\right)+\widehat{S}(n)O(\eps)\right)}.
  \]
  Let us denote the term $f_n(b)\widehat{w}\left(|a|^{-s}(a-N|n|\grad \tilde{\phi}_n(b))\right)$ by $g$. Then
  \[
  v_f(a,b)=\frac{N|n|\grad \tilde{\phi}_n(b)\left(g+O(\eps)\right)}{ g+O(\eps)}.
  \]
  Since $|W_f(a,b)|\geq |a|^{-3s/2}\sqrt{\eps}$ for $(a,b)\in
  R_{\eps}$, $|g|\gtrsim\sqrt{\eps}$, and therefore
  \[
  \frac{|v_f(a,b)-N|n|\grad \tilde{\phi}_n(b)|}{ |N|n|\grad \tilde{\phi}_n(b)|}
  \lesssim \left|\frac{O(\eps)}{g+O(\eps)}\right|\lesssim \sqrt{\eps}.
  \]
 Similarly, if $(a,b)\in S_\eps\cap Z_k$, then 
   \[
  \frac{|v_f(a,b)-N|n|\grad \tilde{\phi}_n(b)|}{ |N|n|\grad \tilde{\phi}_n(b)|}
  \lesssim \left|\frac{O(\eps)}{g+O(\eps)}\right|\lesssim \frac{\sqrt{\eps}}{N^{3s/2}},
  \]
  since $|g|\gtrsim N^{3s/2}\sqrt{\eps}$ for $(a,b)\in S_\eps\cap Z_k$.
\end{proof}

Theorem \ref{mSSWPT:thm:2d2} indicates that the underlying wave-like components $\widehat{S}(n)\alpha(x)e^{2\pi i Nn\cdot\phi(x)}$ of the $3D$  general intrinsic mode type function $f(x)=\alpha(x)S(2\pi N\phi(x))$ are well-separated, if $N$ is sufficiently large. By the definition of the SST, $T_f(v,b)$ would concentrate around their local wave vectors $N\nabla(n\cdot \phi(x))$. 

\section{Crystal Analysis Algorithm and Implementations}
\label{sec:imp}
In this section, we analyze $3D$ crystal images using the $3D$ SST in the previous section. We introduce several fast algorithms to analyze local defects, crystal rotations, and deformations. 
We will still focus on $3D$ cubic crystal images for simplicity. The generalization to other types of crystal images is simple and there will be numerical examples of other types of crystal images in the numerical section. 
%
To make our presentation more transparent, the algorithm and implementation are introduced with a toy example in Figure \ref{fig:bump} (left).

\begin{figure}[ht!]
  \begin{center}
    \begin{tabular}{ccc}
      \includegraphics[width=1.8in]{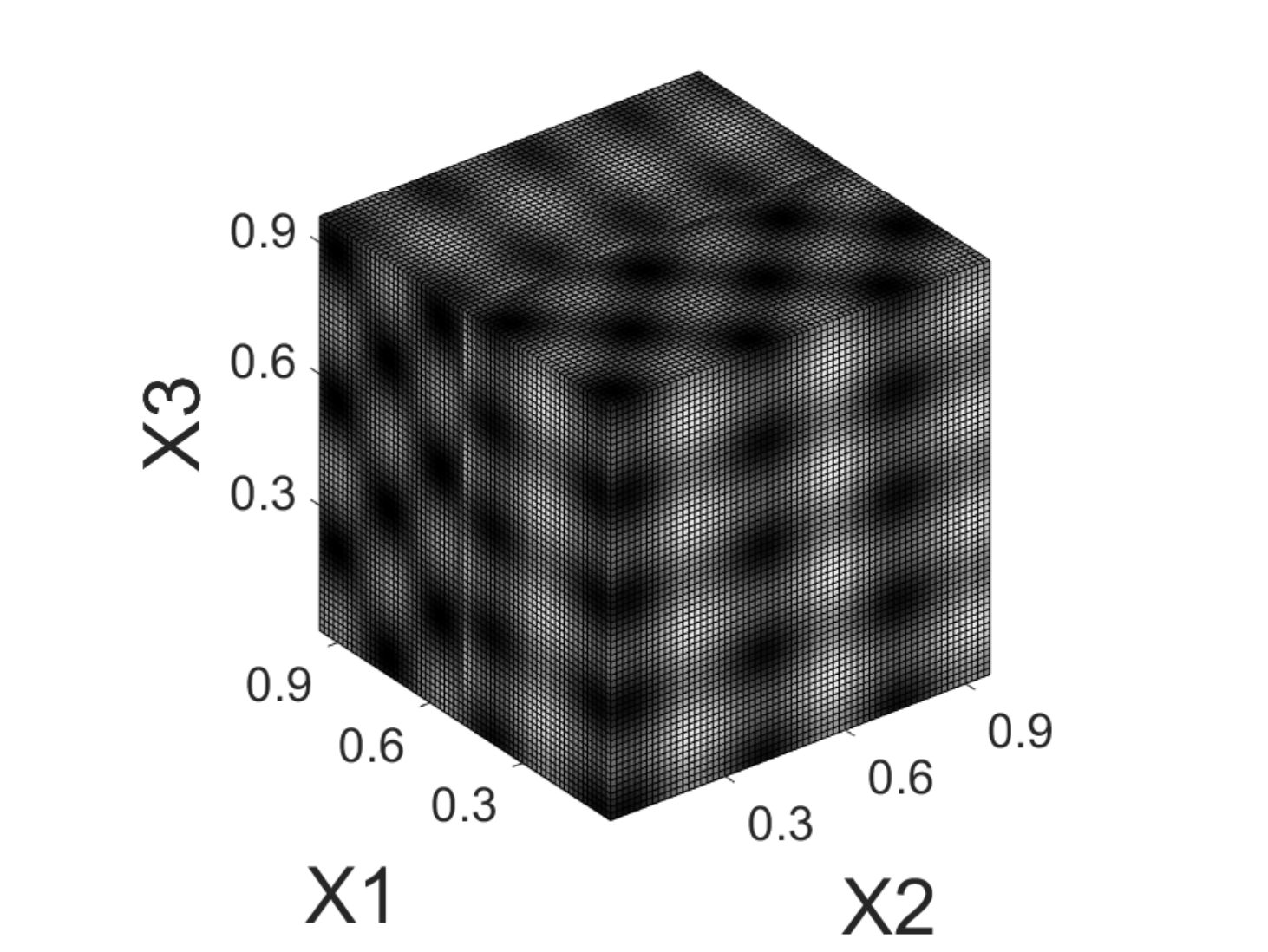} &
      \includegraphics[width=2in]{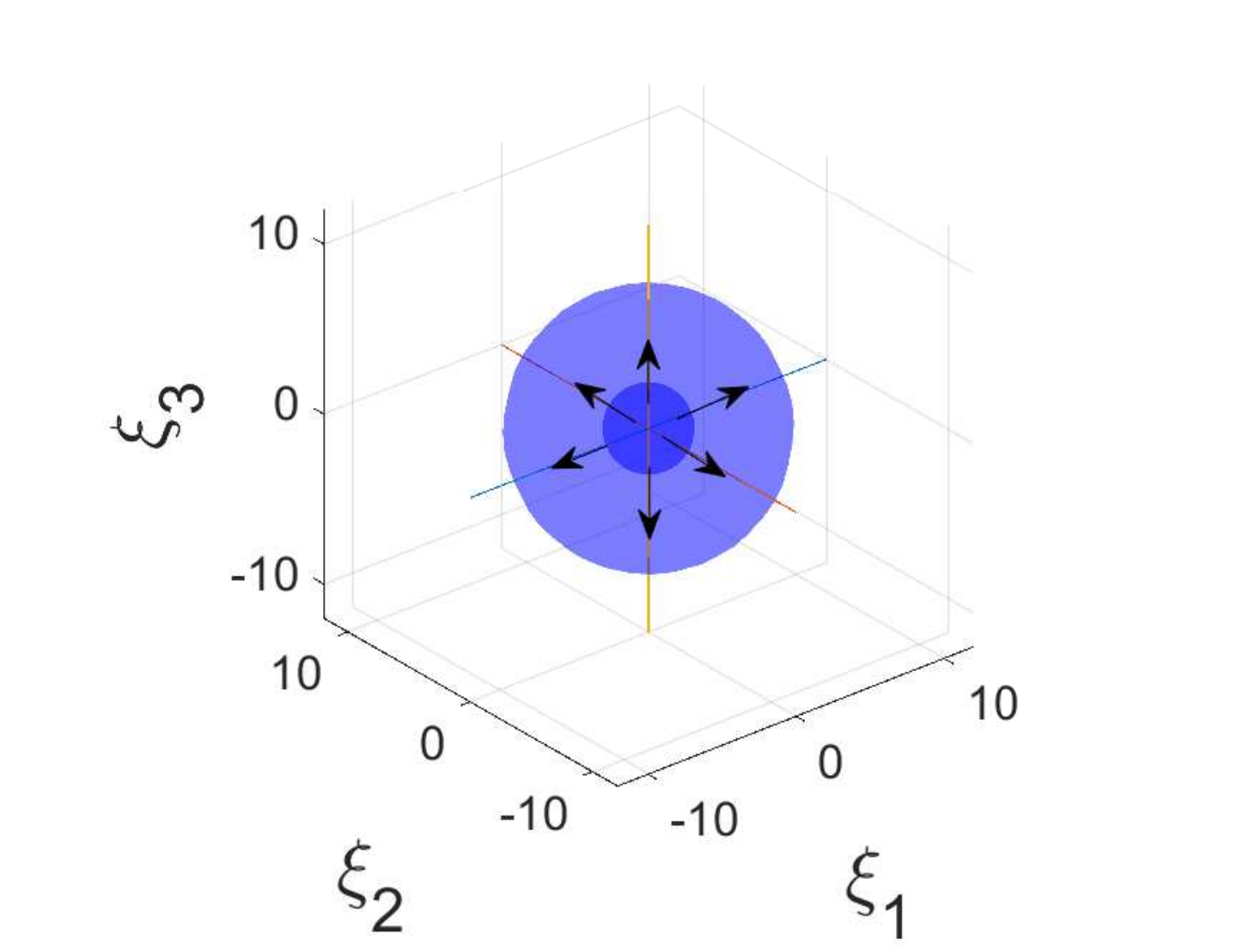} &
      \includegraphics[width=1.4in]{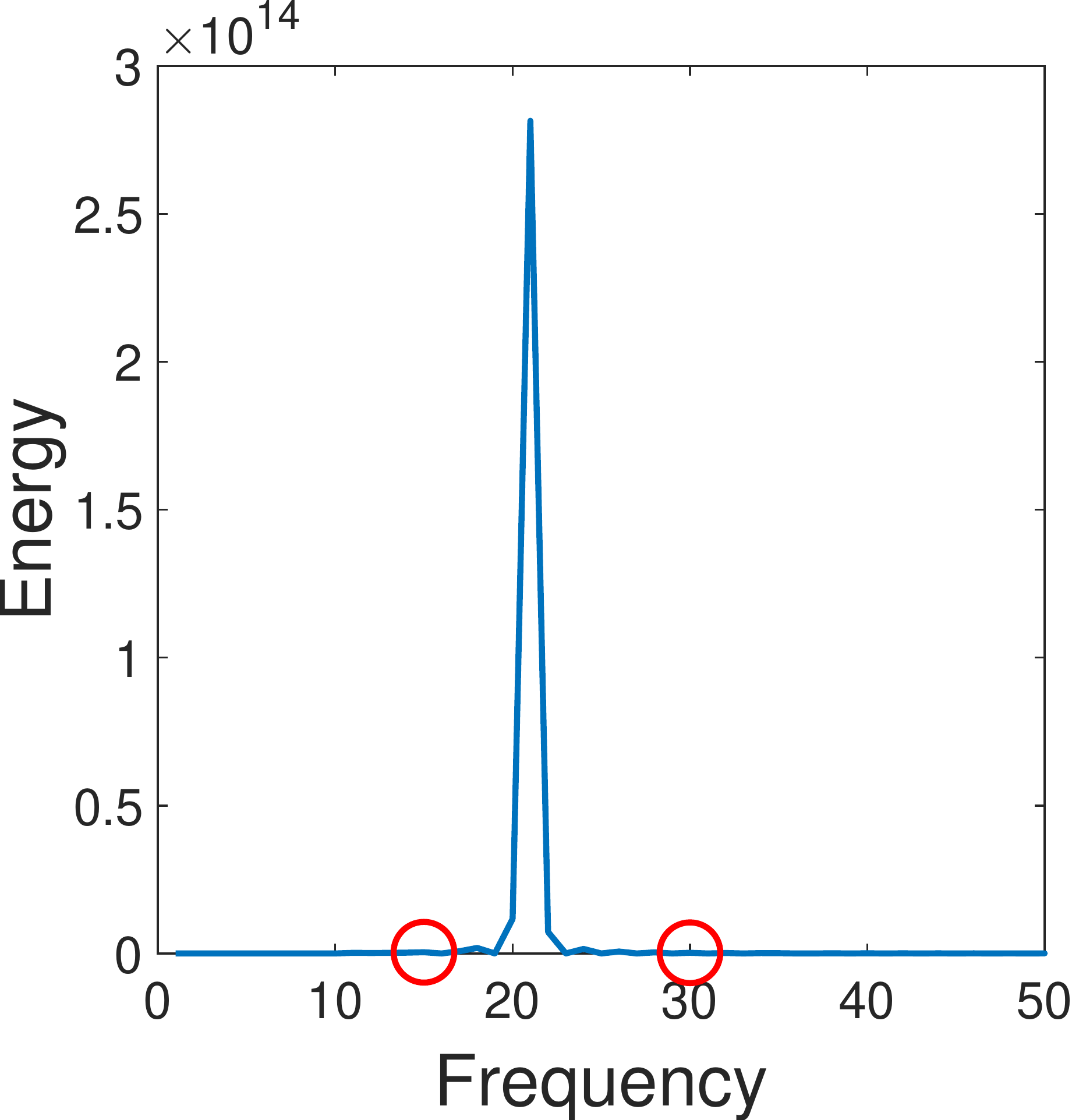} 
    \end{tabular}
  \end{center}
  \caption{Left: an undeformed example of two cubic grains with a vertical
boundary surface. Middle: six dominant local wave vectors in the interior of a grain. All vectors are located in the area in light blue determined by the frequency range parameters $[r_1,r_2]$. Right: the radially average Fourier power spectrum $E(r)$ with the identified most dominant energy bump in the frequency range $[r_1,r_2]$ indicated by two red circles.}
  \label{fig:bump}
\end{figure}

\subsection{Band-limited $3D$ SST in the spherical coordinate}
Typically, each grain
\[
\chi_{\Omega_k} (x) (\alpha_k(x) S(2\pi N \phi_k(x))+c_k(x))  
\]
in a polycrystalline crystal image can be considered as a $3D$ general intrinsic mode type function of type-$(M,N)$ with a small $M$ near $1$, unless the strain is too large. Hence, the $3D$ Fourier power spectrum of a multi-grain image would have several well-separated non-zero energy annuli centered at the origin due to crystal rotations. See Figure \ref{fig:bump} (left and middle) for an example. Suppose the radically average Fourier power spectrum is defined as 
\[
E(r) =\frac{1}{r}\int_0^{2\pi}\int_0^\pi |\widehat{f}(r,\theta,\psi)|d\psi d\theta,
\]
where $\widehat{f}(r,\theta,\psi)$ is the Fourier transform in the spherical coordinate, i.e., \[\widehat{f}(r,\theta,\psi)=\widehat{f}(\xi)\quad\text{for}\quad\xi=(r\sin\theta\cos\psi,r\sin\theta\sin\psi,r\cos\theta).\] Then there would be several well-separated energy bumps in $E(r)$. See Figure \ref{fig:bump} (right) for an energy bump of the example of Figure \ref{fig:bump} (left).

To realize this step, a grid of step size $\Delta$ is generated to discretize the domain $[0,\infty)$ in the variable $r$ as follows:
\[
R = \{n\Delta: n\in \NN\}.
\]
At each $r = n\Delta \in R$, we associate a cell $D_r$
started at $r$
\[
D_r = 
\left[n\Delta,(n+1)\Delta\right).
\]
Then $E(r)$ is estimated by
\[
E(r) = \frac{1}{r}\sum_{\xi\in\Xi: \left|\xi\right|\in D_r }\left| \widehat{f}(\xi)\right|.
\]

According to the structure of cubic lattices, we know that a cubic crystal image with a single grain 
\[
f(x)=\alpha(x) S ( 2\pi N\phi(x) ) + c(x)
\] 
has six dominant local wave vectors close to $\nu_j(\phi(x))$, $j=1$, $2$, $\cdots,6$, which are the vertices of a cubic prisms centered at the origin in the Fourier domain:
\[
\nu_j(\phi(x))=N\nabla(n_j\cdot\phi(x)),\quad j=1,2,\dots, 6,\quad\text{where }\{n_j\}=\{\pm e_1,\pm e_2,\pm e_3\},
\]
and $\{e_1,e_2, e_3\}$ is the standard basis of $\mathbb{R}^3$ in the Fourier domain with axes $\xi=(\xi_1,\xi_2,\xi_3)$. 


To reduce the computational cost of $3D$ SST, we restrict the computation to the spectrum domain that contains the most dominant local wave vectors. The range of such a domain can be identified via the most dominant energy bump in the radially average Fourier power spectrum $E(r)$. Figure \ref{fig:bump} (middle) visualize the range of frequency domain of interest identified by the dominant energy bump in Figure \ref{fig:bump} (right). Suppose the support of the most dominant energy bump of $E(r)$ (i.e. frequency band) is $[r_1,r_2]$. Then a band-limited $3D$ fast SST can be introduced to estimate local wave vectors with wave numbers in $[r_1,r_2]$. 

To be more specific, we consider images that are periodic over the unit cubic $[0,1)^3$ in $3D$. If it is not the case, the images will be periodized by padding zeros around the image boundary.  The key idea of the band-limited $3D$ SST is to restrict the class of wave packets to a band-limited class
\[
\left\lbrace w_{ab}(x): a,b\in \RR^3, |a|\in [r_1,r_2]\right\rbrace.
\]
And the $3D$ SST $T_f(a,b)$ of an image $f$ is only evaluated in the domain $\left\lbrace (a,b): a,b\in \RR^3, |a|\in [r_1,r_2]\right\rbrace$.

To design discrete wave packets, let
\[
X = \{ (n/L: n \in \Z^3, \quad 0 \le n_j<L, \text{ for }1\leq j \leq 3\}
\]
be the spatial grid of size $L$ in each dimension at which these functions are sampled.
The corresponding Fourier grid is
\[
\Xi = \{ \xi\in \Z^3: -L/2\le \xi_j <L/2, \text{ for } 1\leq j\leq 3\}.
\]
For a function $f(x)\in \ell^3(X)$, the discrete forward Fourier
transform is defined by
\[
\hat{f}(\xi) = \frac{1}{L^{3/2}} \sum_{x\in X} e^{-2\pi i x\cdot \xi} f(x),
\]
while the discrete inverse Fourier transform of $g(\xi)\in
\ell^3(\Xi)$ is
\[
\check{g}(x) = \frac{1}{L^{3/2}} \sum_{\xi\in \Xi} e^{2\pi i  x\cdot \xi} g(\xi).
\]

Due to the localization requirement of wave packets in the frequency domain in the analysis of crystal images via SST, a filterbank-based time-frequency transform is applied to design the discrete wave packet transform. In the Fourier domain, $\widehat{w_{ab}}(\xi)$ has the profile
\begin{equation}
  \label{DSSWPT:eqn:wnd}
  |a|^{-3s/2} \widehat{w}(|a|^{-s}(\xi-a)),
\end{equation}
modulo complex modulation. We sample the
Fourier domain $[-L/2,L/2]^3$ with a set $A$ of points $a$ such that $|a|$ is essentially in $[r_1,r_2]$ (as shown as ``*" symbol in
Figure \ref{DSSWPT:fig:tiling1D}) and associate with each $a$ a smooth non-negative window
function $g_a(\xi)$  with a rectangular compact support of length $L_a=O(|a|^s)$ (see Figure \ref{DSSWPT:fig:tiling1D} in blue) that behaves qualitatively as $\widehat{w}(|a|^{-s}(\xi-a))$ essentially centered at $a$. These window functions form a partition of unity of the Fourier domain.

\begin{figure}[ht!]
  \begin{center}
    \begin{tabular}{c}
     \includegraphics[height=2.5in]{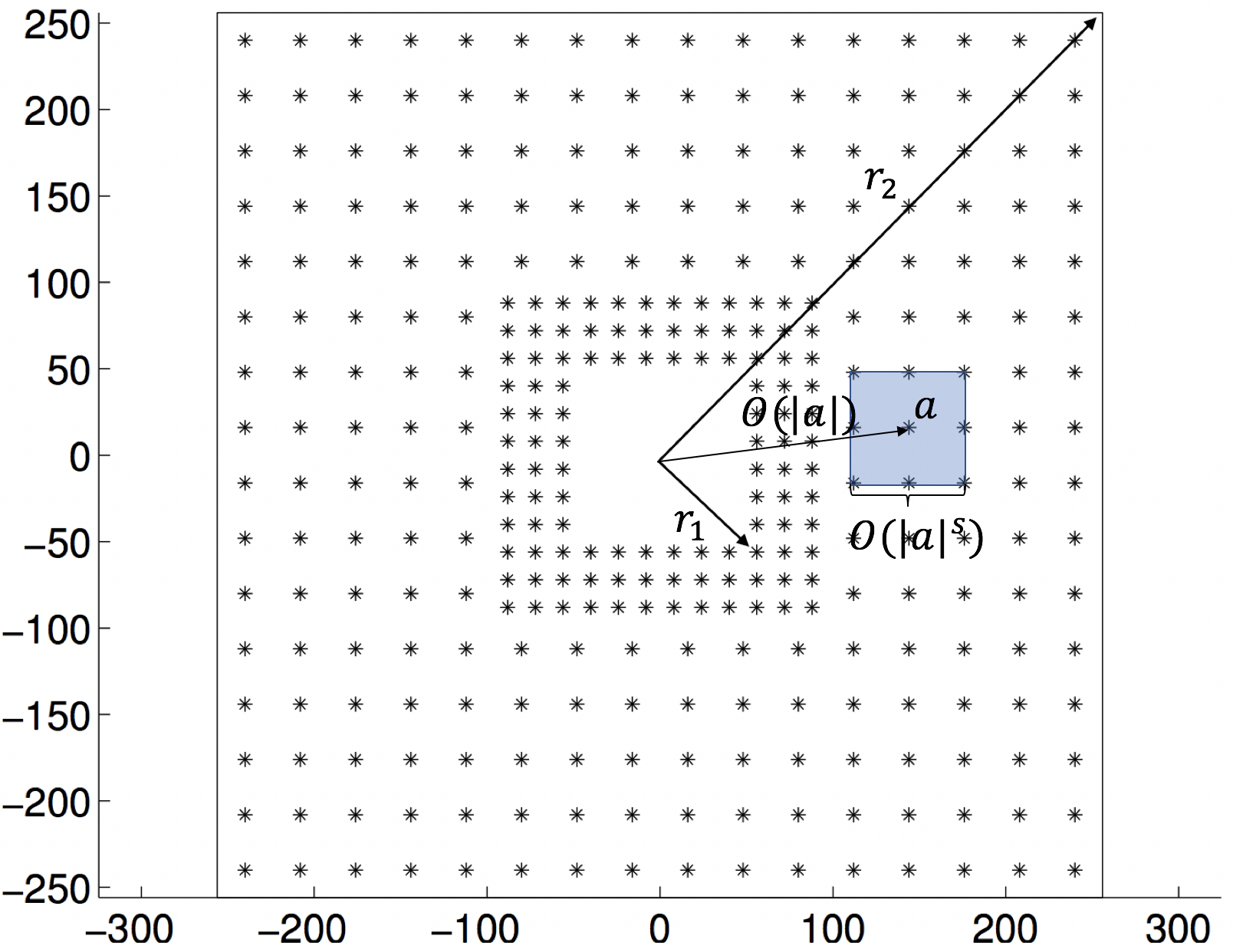} 
    \end{tabular}
  \end{center}
  \caption{Sampled set $A$ of $a$'s in the Fourier domain in the range specified by $[r_1,r_2]$ (projected to $2D$ on the $a_1$-$a_2$ plane for visualization) for an image of size $512\times512\times512$. Each point represents the center of the support of a window function with a rectangular compact support of length $L_a=O(|a|^s)$ in blue.}
  \label{DSSWPT:fig:tiling1D}
\end{figure}

The way we constructed the wave packets is similar to the constructions of the wave atom frame in \cite{Demanet2007}. When $s=1/2$, our wave packets become wave atoms. A straightforward calculation shows that the total number of
samples $a$'s is of order $O(L^{3(1-s)})$.

In the spatial doamin, we simply discretize it with a uniform grid of size $L_B$ in each dimension as follows:
\[
B = \{ n/L_B: n\in \Z^3,\quad 0 \le n_j<L_B, \text{ for } 1\leq j\leq 3\}.
\]

For each fixed $a\in A$ and $b\in B$, the discrete wave packet $w_{ab}(x)$ is defined as
\begin{equation}
\label{DSSWPT:eqn:DSSWPTFFT}
\widehat{w_{ab}}(\xi) = \frac{1}{L_a^{3/2}} e^{-2\pi i b\cdot \xi} g_a(\xi)
\end{equation}
for $\xi \in \Xi$. 

For a function $f(x)$ defined on $x\in X$, the discrete wave packet
transform is a map from $\ell_2(X)$ to $\ell_2(A\times B)$, defined by
\begin{equation}
  W_f(a,b) = \langle w_{ab}, f \rangle =  \langle \widehat{w_{ab}}, \hat{f} \rangle = \int \overline{\widehat{w_{ab}}(\xi)} \hat{f}(\xi)d\xi=
  \frac{1}{L_B^{3/2}} \sum_{\xi\in\Xi} e^{2\pi i b\cdot\xi} g_a(\xi) \hat{f}(\xi).
  \label{DSSWPT:eqn:D1}
\end{equation}

Note that
\[
\grad_b W_f(a,b) = \grad_b \langle \widehat{w_{ab}},\hat{f}\rangle = 
\langle -2\pi i \xi\widehat{w_{ab}}(\xi),\hat{f}(\xi) \rangle.
\]
Therefore, the discrete gradient $\grad_b W_f(a,b)$ can be evaluated 
in a similar way
\begin{equation}
  \grad_b W_f(a,b) = \sum_{\xi\in\Xi} \frac{1}{L_B^{3/2}} 2\pi i \xi e^{2\pi
    i b\cdot\xi} g_a(\xi) \hat{f}(\xi).
  \label{DSSWPT:eqn:D3}
\end{equation}

Applying the fast Fourier transform (FFT) to evaluate the summation in \eqref{DSSWPT:eqn:D1} and \eqref{DSSWPT:eqn:D3}, the wave packets admit a fast transform with complexity of $O(L^3\log L + L^{3(1-s)} L_B^3 \log L_B)$ with
$L_B \ge \max_{a\in A} L_a = O(L^s)$, where $L$ is bounded by $r_2$, the frequency range parameter. 

For a given crystal image $f(x)$ defined on $x\in X$, we compute $W_f(a,b)$ and $\grad_b W_f(a,b)$ via the FFT. The approximate local wavevector
$v_f(a,b)$ is then estimated by
\[
v_f(a,b) = \frac{\grad_b W_f(a,b)}{2\pi i W_f(a,b)}
\]
for $a$ and $b$ in the domain $R_{\eps} = \{(a,b): a\in A, b\in B, |W_f(a,b)|\geq \sqrt \eps\}$.

To specify the SST $T_f(v,\psi,\theta,b)$  in the spherical coordinate $(v,\psi,\theta)$ of the Fourier domain, we use step-sizes $\Delta_v$, $\Delta_\psi$, and $\Delta_\theta$ in $v$, $\psi$, and $\theta$, respectively, to construct grids $G_v = \{n\Delta_v: n\in \Z, n\Delta_v\in[r_1,r_2]\}$, $G_\psi = \{n\Delta_\psi: n\in \Z, n\Delta_\psi\in[0,2\pi]\}$, and $G_\theta = \{n\Delta_\theta: n\in \Z, n\Delta_\theta\in[0,\pi]\}$.
At each $v = n\Delta_v \in G_v$, we associate a cell $D_v$
centered at $v$, i.e., $D_v = \left[(n-\frac{1}{2})\Delta_v,(n+\frac{1}{2})\Delta_v\right)$. Similarly, we have cells $D_\psi = \left[(n-\frac{1}{2})\Delta_\psi,(n+\frac{1}{2})\Delta_\psi\right)$ and $D_\theta = \left[(n-\frac{1}{2})\Delta_\theta,(n+\frac{1}{2})\Delta_\theta\right)$ at the grid $\psi=n\Delta_\psi$ and $\theta=n\Delta_\theta$, respectively. 
Then the discrete SST in a spherical coordinate is defined as
\[
T_f(v,\psi,\theta,b) = \sum_{(a,b)\in R_{\eps}:\Re v_f(a,b)\in D_v\times D_\psi\times D_\theta} | W_f(a,b)|^2.
\]

\begin{figure}[ht!]
  \begin{center}
    \begin{tabular}{ccc}
      \includegraphics[width=1.6in]{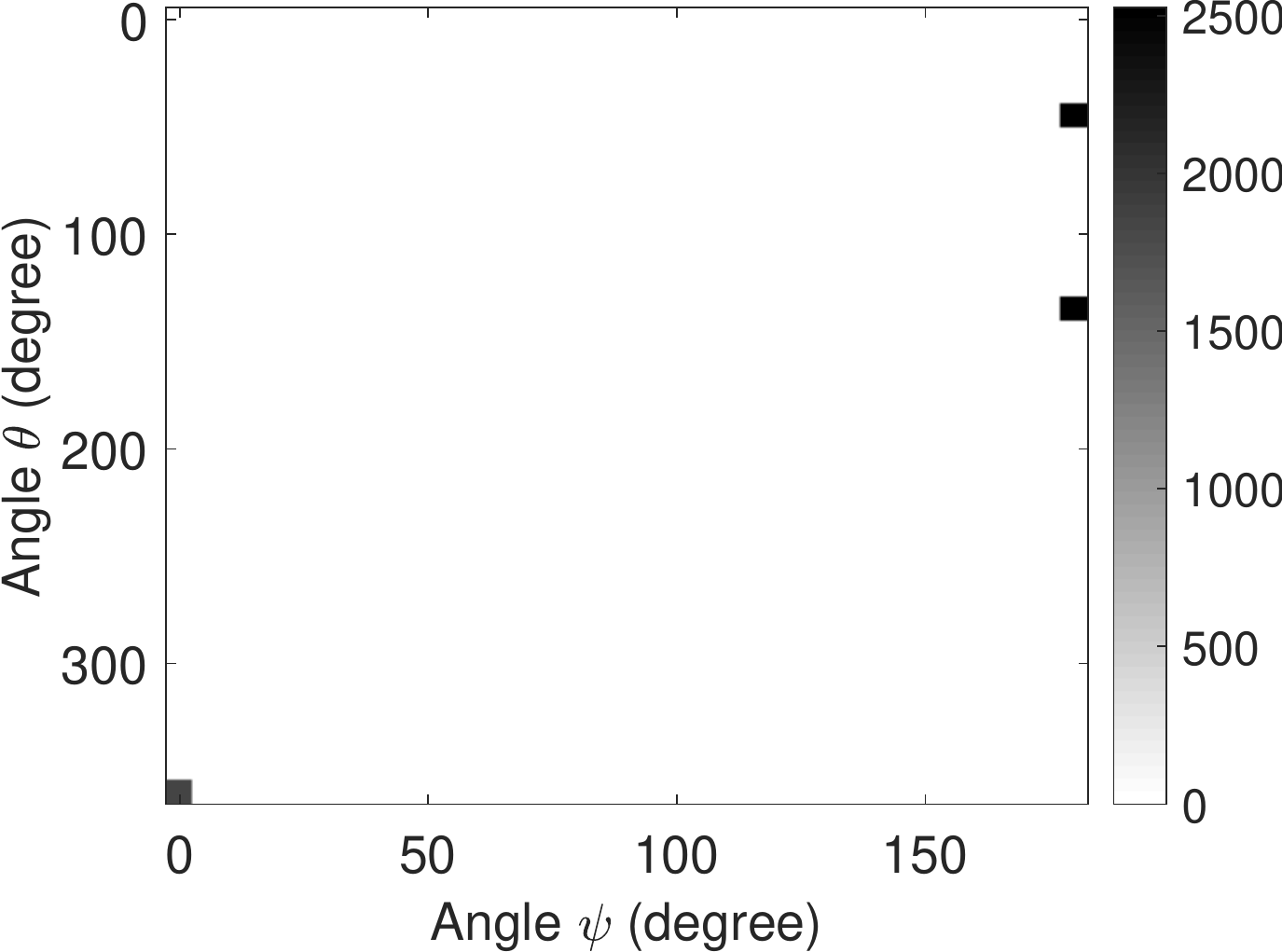} &
      \includegraphics[width=1.6in]{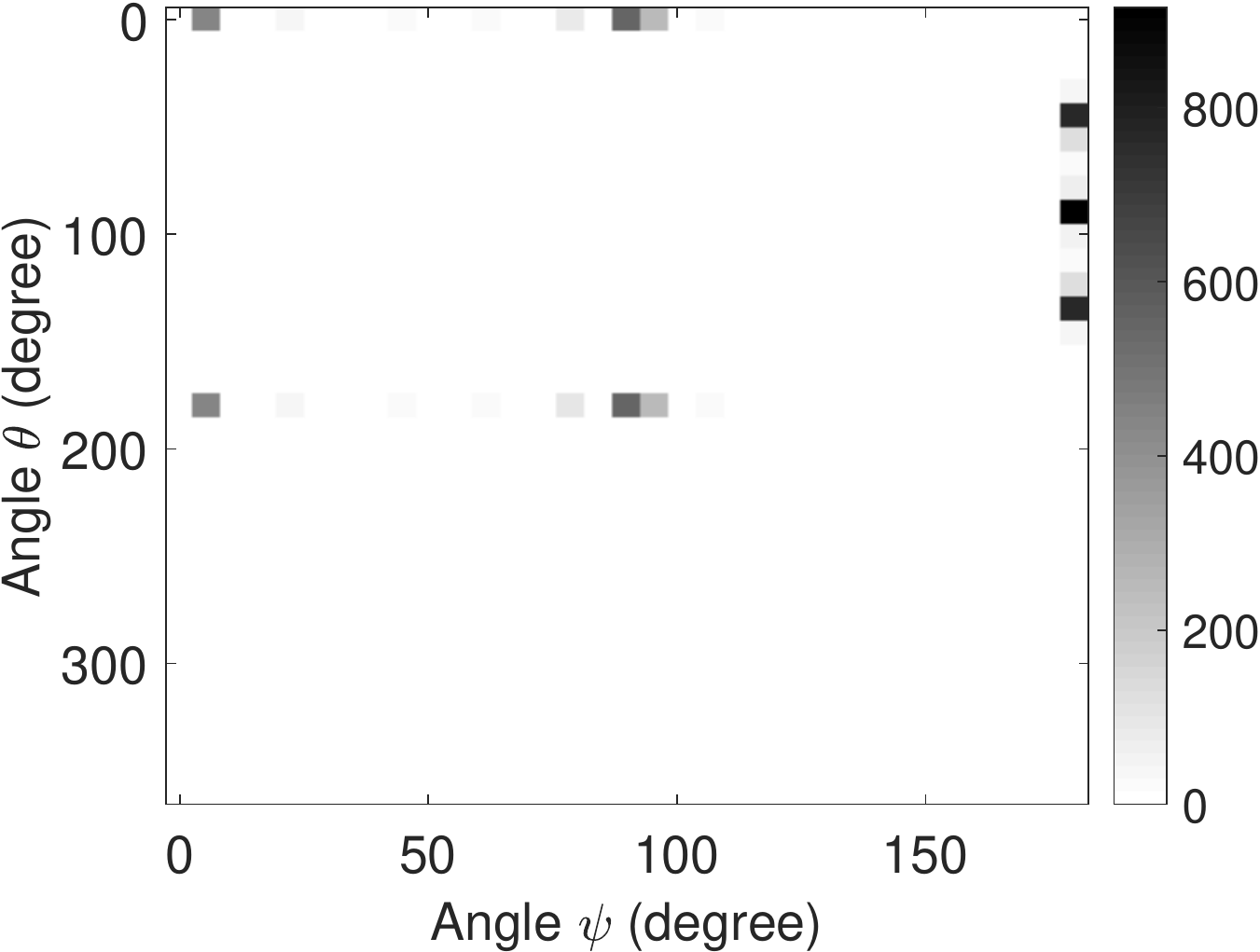} &
      \includegraphics[width=1.6in]{fig/sample7(cubic)/Dis_bump_Interior.pdf} 
    \end{tabular}
  \end{center}
  \caption{The SST $T_f(v,\psi,\theta,b)$ of Figure \ref{fig:bump} (left) in the spherical coordinate at three different points $b_i$, $i = 1,$ $2$, and $3$ when $v$ is set to be the length of local wave vectors. Left: $b_1 = (0.25, 0.5,0.5)$ is in the middle of the right grain. Middle:
$b_2 = (0.5, 0.5,0.5)$ is at the boundary. Right: $b_3 = (0.75, 0.5,0.5)$ is in the middle of the left grain. }
  \label{fig:defc}
\end{figure}

As an example, Figure \ref{fig:defc} shows the SST $T_f(v,\psi,\theta,b)$ in a spherical coordinate at three different positions of a crystal image in Figure \ref{fig:bump} (left). The results show that the essential support of $T_f(v,\psi,\theta,b)$ concentrates around local wave vectors of a grain $\alpha(x) S ( 2\pi N\phi(x) ) + c(x)$:
\[
\nu_j(\phi(b))=N\nabla(n_j\cdot\phi(b)),\quad j=1,2,\dots, 6,\quad\text{where }\{n_j\}=\{\pm e_1,\pm e_2,\pm e_3\},
\]
when the location $b$ is not at the boundary. When $b$ is at the boundary, the essential support of $T_f(v,\psi,\theta,b)$ can still provide some information, e.g. crystal rotations.

\subsection{Defect detection algorithm}
\label{sec:def}
As we can see in Figure \ref{fig:defc}, the synchrosqueezed energy distribution has supports around local wave vectors $\{\nu_j(\phi(b))\}_{j=0,1,\dots, 5}$ with energy of order $|\widehat{S}(n)|\alpha(b)$, when $b$ is in the grain interior. Moreover, the energy would decrease fast near defects. This motivates the application of SST to identify defects by detecting the irregularity of energy distribution as follows. Due to  the symmetry of crystal lattice, we only consider the domain when $(\psi,\theta)\in[0,\pi)\times[0,\pi)$, in which we have three dominate local wave vectors, denoted as $\{\nu_j(\phi(b))\}_{j=1,2,3}$. Let $\delta>0$ be a small parameter for the size of the supports of the SST around $\{\nu_j(\phi(b))\}_{j=1,2,3}$ and $B_\delta(\nu_j(\phi(b)))$ be the ball centered at $\nu_j(\phi(b))$ with a radius $\delta$. As we can see in Figure \ref{fig:defc}, when $b$ is in the grain interior, these balls can be easily estimated; while when $b$ is near the defect location, we could only identify three balls that most possibly contain the energy of $T_f(v,\psi,\theta,b)$. We define
\[
\mass(b) :=\frac{\displaystyle \sum_{j=1}^3\int_{(v,\psi,\theta)\in B_\delta(\nu_j(\phi(b)))} T_f(v,\psi,\theta,b) \ud v\ud \psi\ud\theta}{\displaystyle \int_{(v,\psi,\theta)\in [r_1,r_2]\times [0,\pi)\times[0,\pi)} T_f(v,\psi,\theta,b) \ud v \ud \psi\ud\theta}\,.
\]
Then $\mass(b)$ will be close to $1$ in the interior of a grain; while $\mass(b)\ll 1$ when $b$ is near the defect location. Figure~\ref{fig:BD} (left) visualizes $\mass(b)$ of the crystal image in Figure \ref{fig:bump} (left).

Hence, the estimate of defect regions can be
obtained by thresholding $\mass(b)$ at some value $\eta\in(0,1)$ according to
\begin{equation*}
\Omega_d=\{b\in\Omega\ :\ \mass(b)<\eta\}\,,
\end{equation*}
an illustration of which is shown in Figure~\ref{fig:BD} (right). 

\begin{figure}[ht!]
  \begin{center}
    \begin{tabular}{cc}
      \includegraphics[width=1.8in]{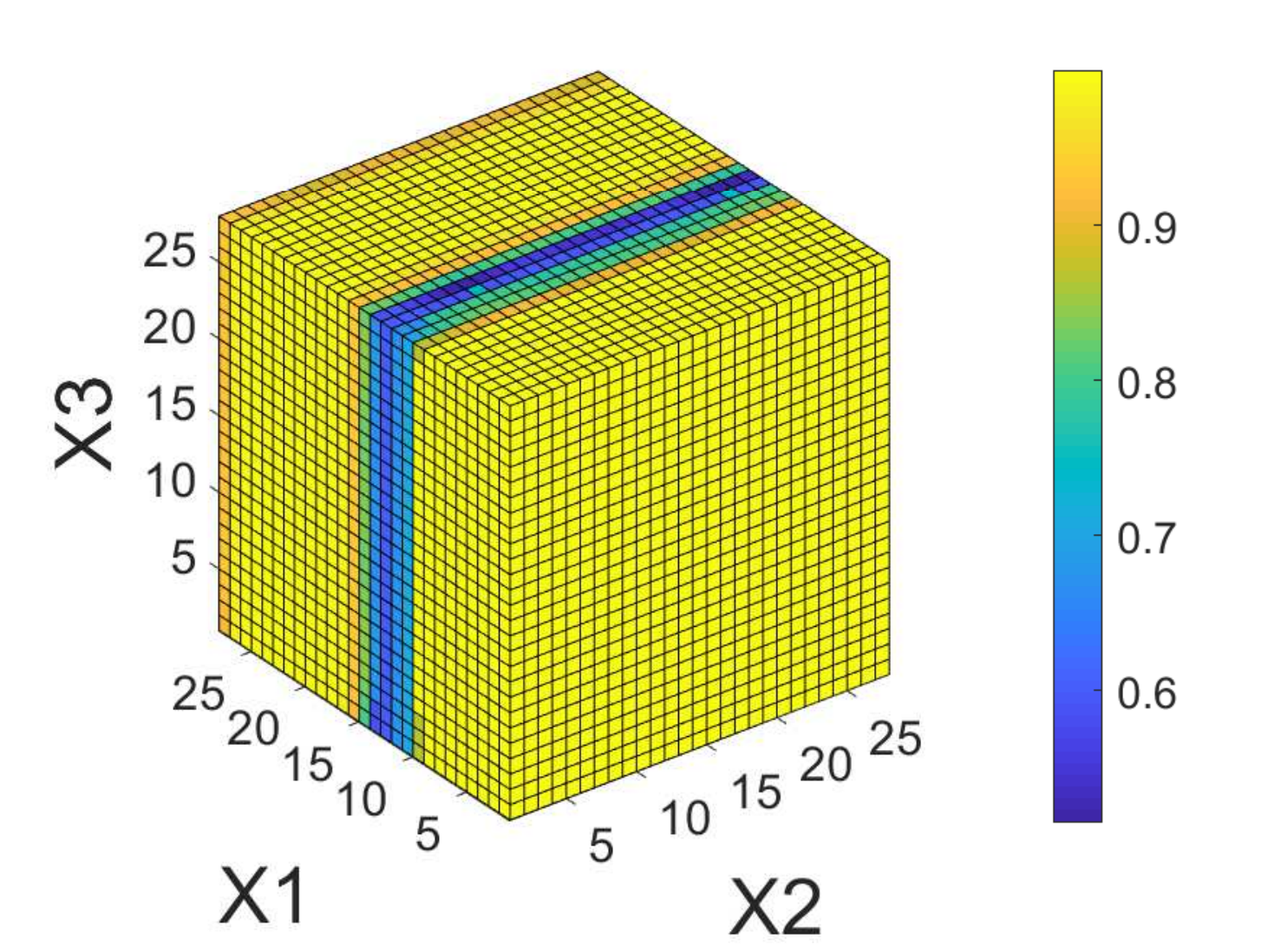} &
      \includegraphics[width=1.8in]{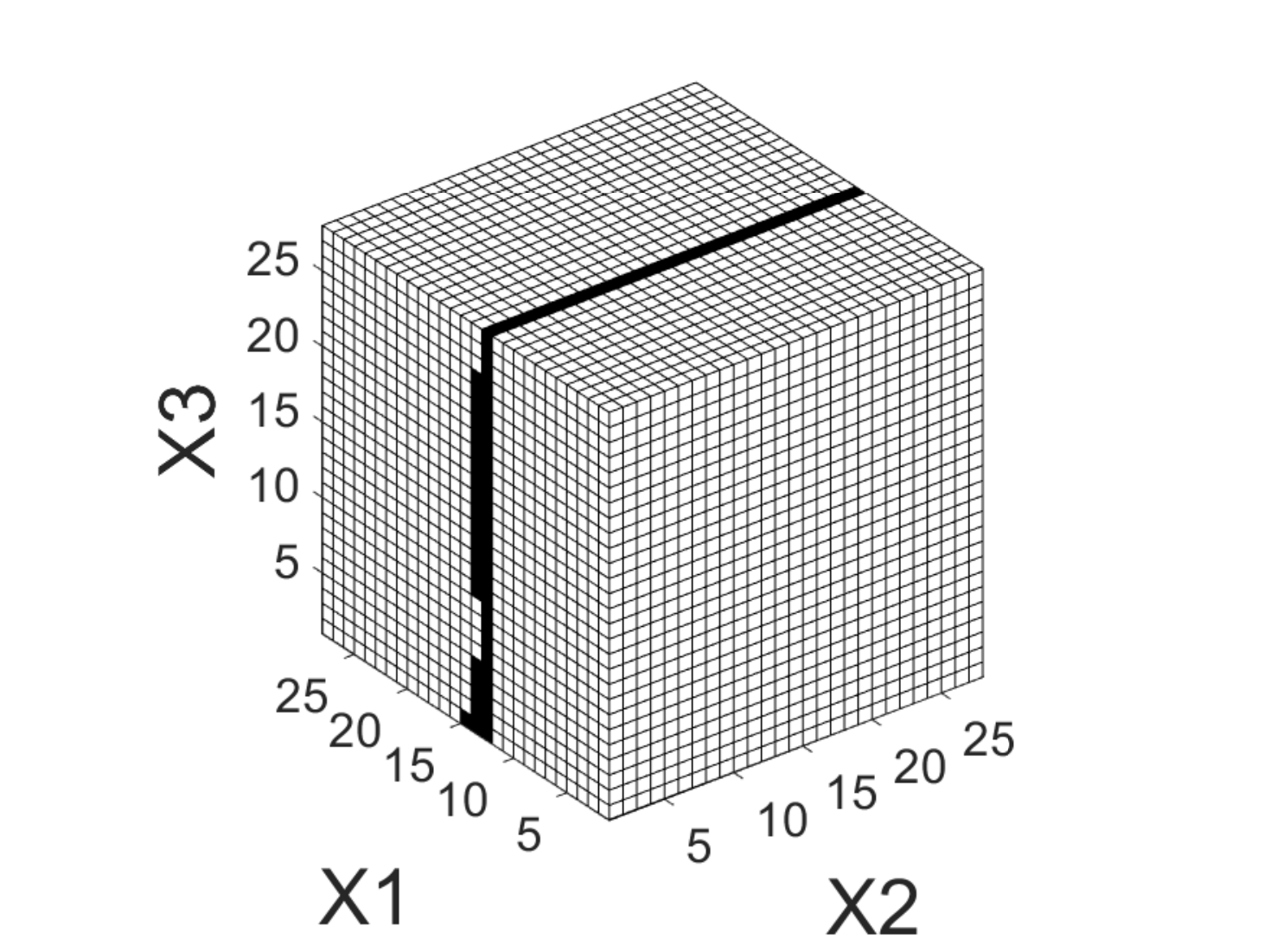} 
    \end{tabular}
  \end{center}
  \caption{Left: the $\mass(b)$ of the crystal image in Figure \ref{fig:bump} (left). Right: the area in black indicates the estimated defect area $\Omega_d$ of the crystal image in Figure \ref{fig:bump} (left). }
  \label{fig:BD}
\end{figure}

\subsection{Recovery of local inverse deformation gradient}
\label{sec:def}
Instead of estimating the elastic deformation $\phi(x)$ of a grain directly, we would equivalently estimate the local inverse deformation gradient $\nabla\phi(x)\in\mathbb{R}^{3\times 3}$, where
\[
\grad \phi(x)=\left(
\begin{array}{ccc}
              \partial_{x1}\phi_1(x) & \partial_{x2}\phi_1(x) &  \partial_{x3}\phi_1(x)\\
              \partial_{x1}\phi_2(x) & \partial_{x2}\phi_2(x) & \partial_{x3}\phi_2(x)\\
             \partial_{x1}\phi_3(x) & \partial_{x2}\phi_3(x) &\partial_{x3}\phi_3(x)\\
\end{array} 
              \right)  ,
\] 
since  $\nabla\phi(x)$ leads to the estimation of crystal rotations and principal stretches.

With the three balls $\{B_\delta(\nu_j(\phi(x)))\}_{j=1,2,3}$ identified in Section \ref{sec:def}, we can estimate their corresponding local wave vectors $\{\nu_j(\phi(x))\}_{j=1,2,3}$ via the weighted average of the location of non-zero SST as follows. Let
\[
(v^{est}_j (x),\psi^{est}_j (x),\theta^{est}_j (x)):=\frac{\displaystyle \int_{(v,\psi,\theta)\in B_\delta(\nu_j(\phi(x)))} (v,\psi,\theta)T_f(v,\psi,\theta,b) \ud v\ud \psi\ud\theta}{\displaystyle \int_{(v,\psi,\theta)\in B_\delta(\nu_j(\phi(x)))} T_f(v,\psi,\theta,b) \ud v \ud \psi\ud\theta}\,,
\]
then the estimated local wave vector in the Cartesian coordinate is
\begin{equation*}
\nu_j^{est}(\phi(x)):=\left(v^{est}_j (x)\sin(\psi^{est}_j (x)) \cos(\theta^{est}_j (x)) ,  v^{est}_j (x) \sin(\psi^{est}_j (x)) \sin(\theta^{est}_j (x)) , v^{est}_j (x) \cos(\psi^{est}_j (x)\right)
\end{equation*}
for $ j=1$, $2$, and $3$.

Note that the local wave vectors of an undeformed reference cubic grain $\alpha(x) S ( 2\pi Nx ) + c(x)$ are $\nu_j^{ref}(x)=N e_j$ for $j=1$, $2$, and $3$. Hence, the inverse deformation gradient $\nabla \phi(x)$ is determined by a least squares fitting of the deformed local wave vectors $\nu_j^{est}(\phi(x))$ to the reference local wave vectors $N e_j$:
\[
\nabla \phi(x)\approx G(x) = \underset{\tilde{G}}{\argmin} \sum_{j=1}^3 \norm{\nu_j^{est}(\phi(x)) - N \tilde{G} e_j}_2^2,
\]
where $N$ can be estimated by the weighted average of the radically average Fourier power spectrum via $N\approx \int_{r_1}^{r_2}r E(r)\ud r$.

With the local inverse deformation gradient estimation $G(x)\in\mathbb{R}^{3\times 3}$ for each $x$, we can estimate crystal rotations and principle stretches by the polar decomposition of $G(x)$, i.e., $G(x)=R(x)U(x)$, where $R(x)\in\mathbb{R}^{3\times 3}$ is a unitary matrix implying how to rotate vectors $\{N G e_j\}$ to match reference vectors $\{\nu_j^{est}(\phi(x))\}$, and $U(x)$ is a positive-semidefinite Hermitian matrix describing how to stretch $\{N G e_j\}$ to match $\{\nu_j^{est}(\phi(x))\}$.

In the $3$D space, the orientation of a crystal lattice compared to a reference lattice can be described by a series of elemental rotations, that is, rotations around the axes along $e_1$, $e_2$, and $e_3$ in the Cartesian coordinate. Generally, the composition of these three elemental rotations is capable of recovering any rotation in the $3$D space.  Suppose $\alpha(x)$, $\beta(x)$, and $\gamma(x)$ are the Euler angles of elemental rotations along $e_1$, $e_2$, and $e_3$ to generate the orientation of a crystal lattice, then $\alpha(x)$, $\beta(x)$, and $\gamma(x)$ can be computed via the unitary matrix $R(x)$ in the polar decomposition $G(x)=R(x)U(x)$, since
\[
R=\left(
\begin{array}{ccc}
\cos\beta\cos\gamma & -\cos\beta\sin\gamma & \sin\beta\\
\cos\alpha + \cos\beta\sin\alpha\sin\beta & \cos\alpha\cos\gamma - \sin\alpha\sin\beta\sin\gamma & -\cos\beta\sin\alpha\\
\sin\alpha\sin\gamma - \cos\alpha\cos\gamma\sin\beta & \cos\gamma\sin\alpha + \cos\alpha\sin\beta\sin\gamma & \cos\alpha\cos\beta
\end{array}
\right).
\]
As an example, Figure \ref{fig:agl} shows the estimated Euler angles $\alpha(x)$, $\beta(x)$, and $\gamma(x)$ of the crystal orientation of the crystal image in Figure \ref{fig:bump}. 

     \begin{figure}[ht!]
  \begin{center}
    \begin{tabular}{ccc}
      \includegraphics[width=1.6in]{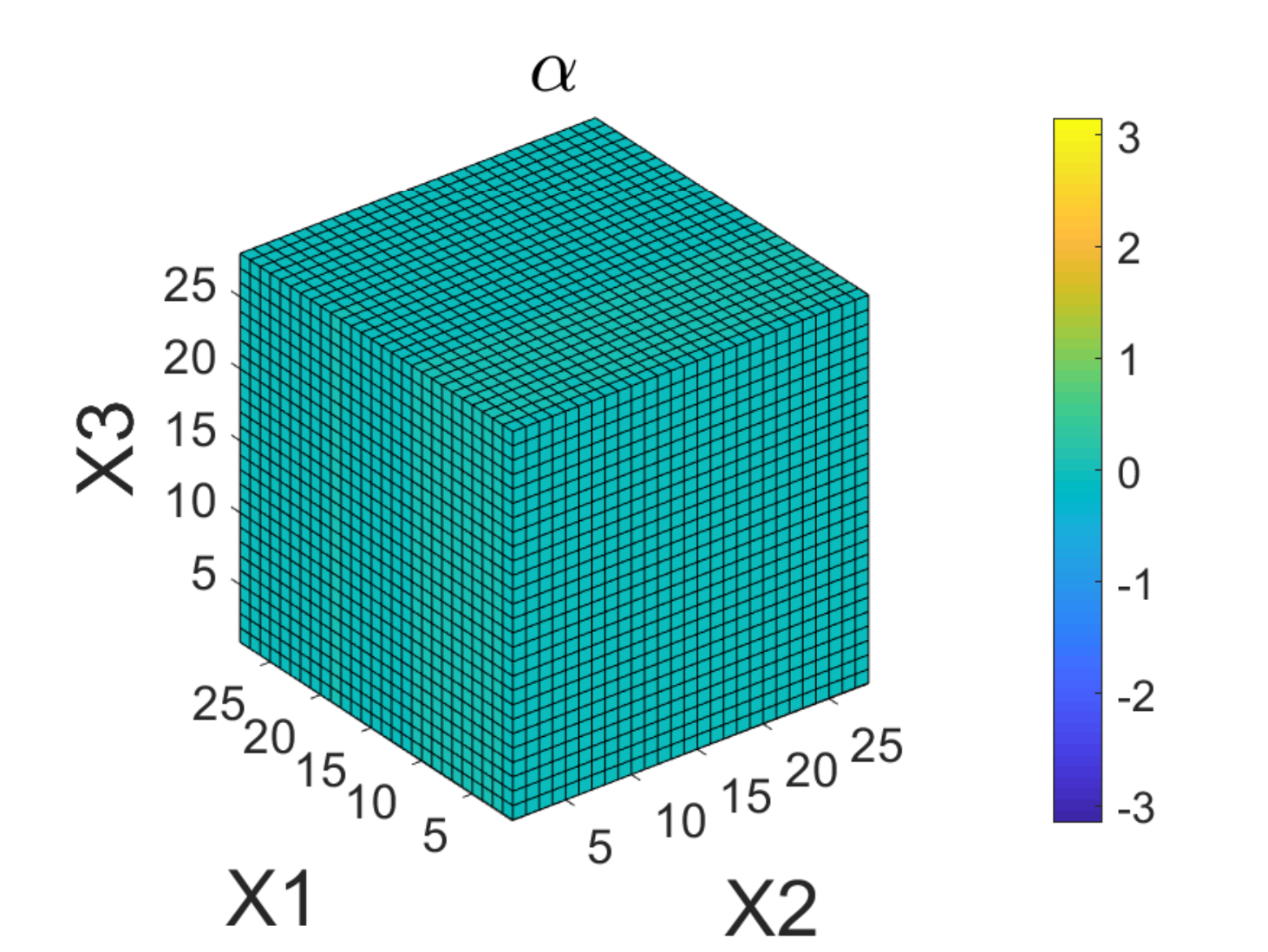}  &
      \includegraphics[width=1.6in]{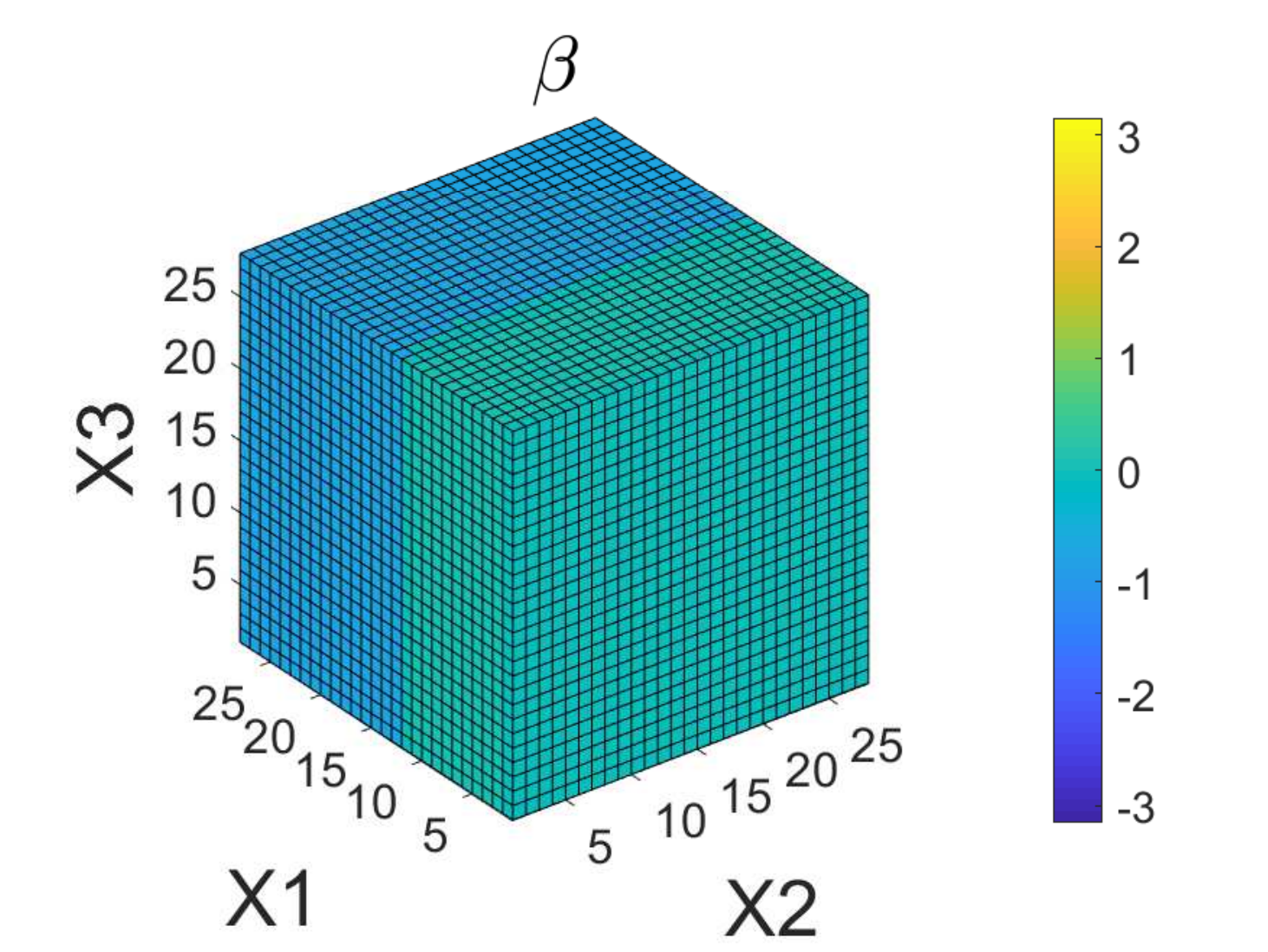} &
      \includegraphics[width=1.6in]{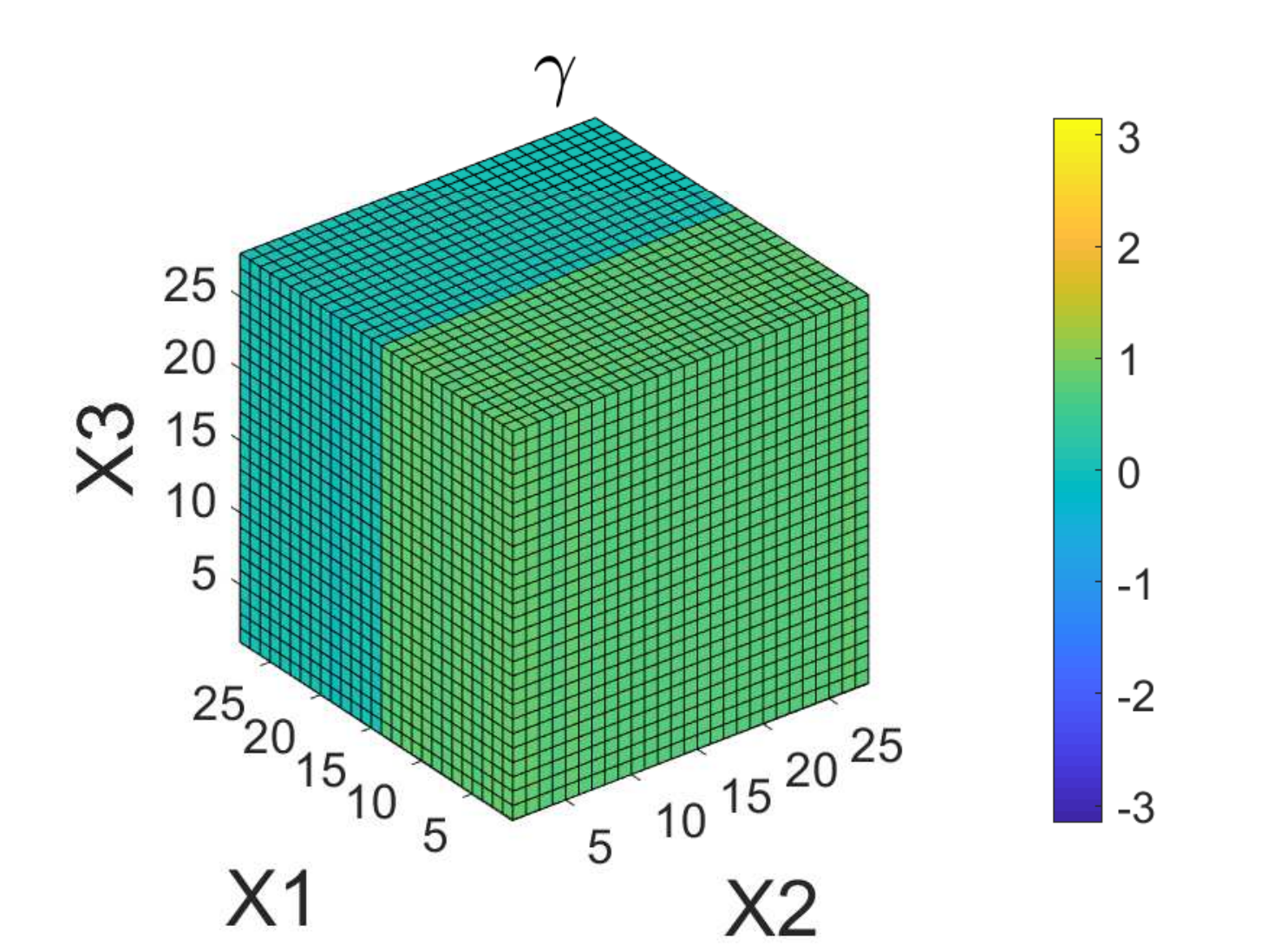} 
    \end{tabular}
  \end{center}
  \caption{Estimated Euler angles of the crystal orientation of the crystal image in Figure \ref{fig:bump}. From left to right: $\alpha(x)$, $\beta(x)$, and $\gamma(x)$, respectively.}
  \label{fig:agl}
\end{figure}

\section{Examples And Discussions}
\label{sec:results}
In this section, we present several numerical examples of synthetic and real images to illustrate the performance of our method. The corresponding code is open source and 
will be available as SynCrystal at \url{https://github.com/SynCrystal/SynCrystal}.
We first apply our method to analyze synthetic crystal images with triple junction grain boundaries and isolated defects. To show the robustness of our method, both noiseless
and noisy examples are presented. In the second part, we apply our algorithm to crystal images from real experiments.

The main parameters for the $3D$ SST is the geometric scaling
parameter $s=\frac{1}{2}$\footnote{Although in theory we cannot guarantee the performance of SST when $s=\frac{1}{2}$, it works well in practice.}. The accuracy of our algorithm is not sensitive to the discretization grids in the SST and the crystal image analysis, while the speed and memory cost of our code depends on the grid sizes. Hence, we adaptively choose the grid sizes according to the size of data and the resolution we need for visualization.

\subsection{Synthetic atomic resolution crystal images}

\subsubsection{Triple junction grain boundaries} 

Our first example is a synthetic hexagonal crystal image with triple junction grain boundaries in Figure \ref{fig:tri1} (left). It contains three grains with triple junction grain boundaries. As shown in Figure \ref{fig:tri1} (middle), $\text{mass}(x)$ clearly visualizes the grain boundaries and after thresholding it gives the location of the boundaries in Figure \ref{fig:tri1} (right). Figure \ref{fig:tri2} shows the estimated Euler angles of the crystal orientation of the crystal image in Figure \ref{fig:tri1} (left) (from left to right: $\alpha(x)$, $\beta(x)$, and $\gamma(x)$, respectively). A noisy version of the triple junction example and its analysis results are given in Figure \ref{fig:tri3} and \ref{fig:tri4}. The results in Figure \ref{fig:tri3} and \ref{fig:tri4} are comparable to those in Figure \ref{fig:tri1} and \ref{fig:tri2}, even though the atom structure is hardly seen in the presence of heavy noise (see the comparison of the zoomed-in picture in the corner of Figure \ref{fig:tri3} (left) and that in the corner of Figure \ref{fig:tri1} (left)).

\begin{figure}[ht!]
  \begin{center}
    \begin{tabular}{ccc}
      \includegraphics[width=1.8in]{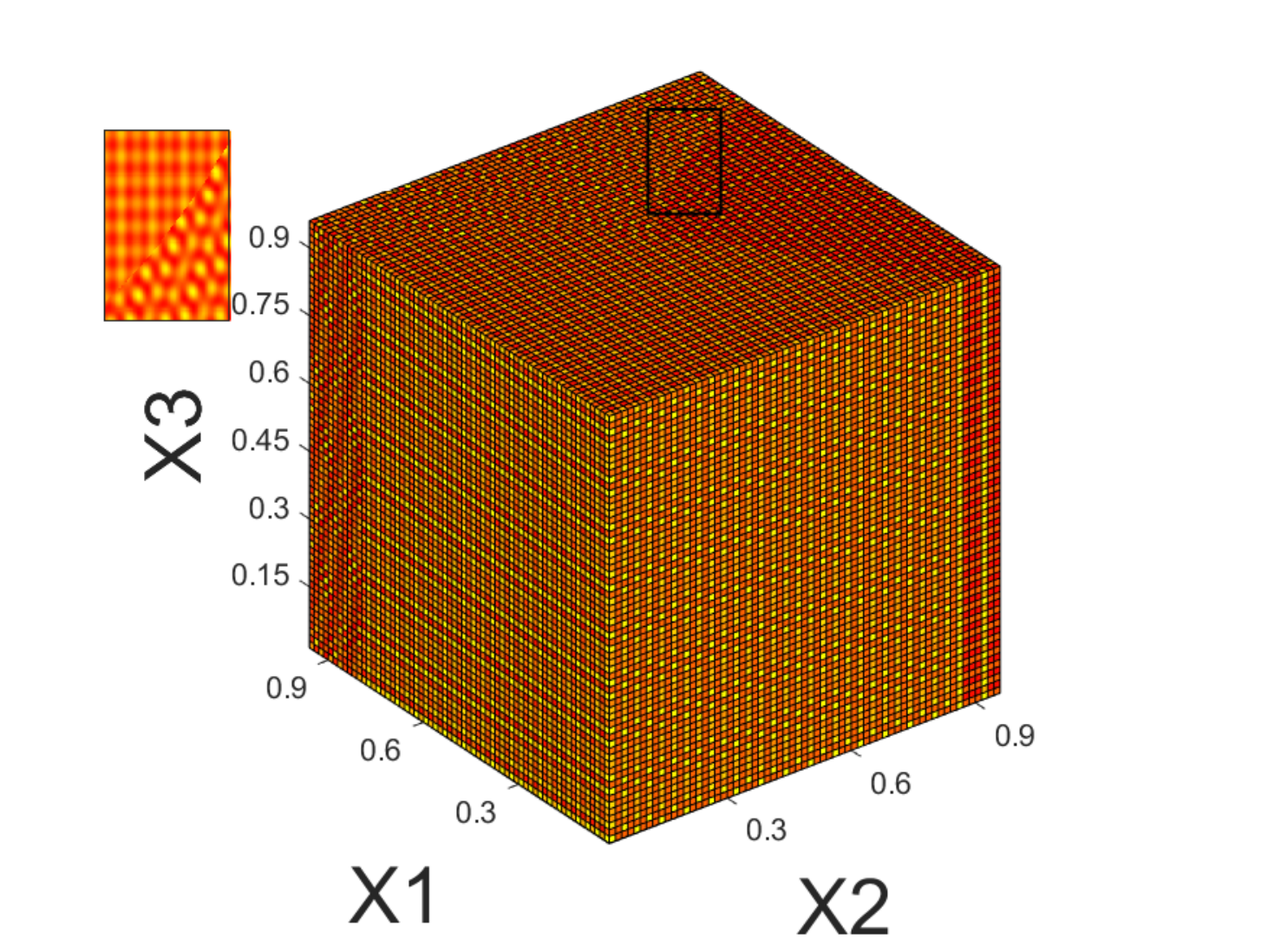} &
      \includegraphics[width=1.8in]{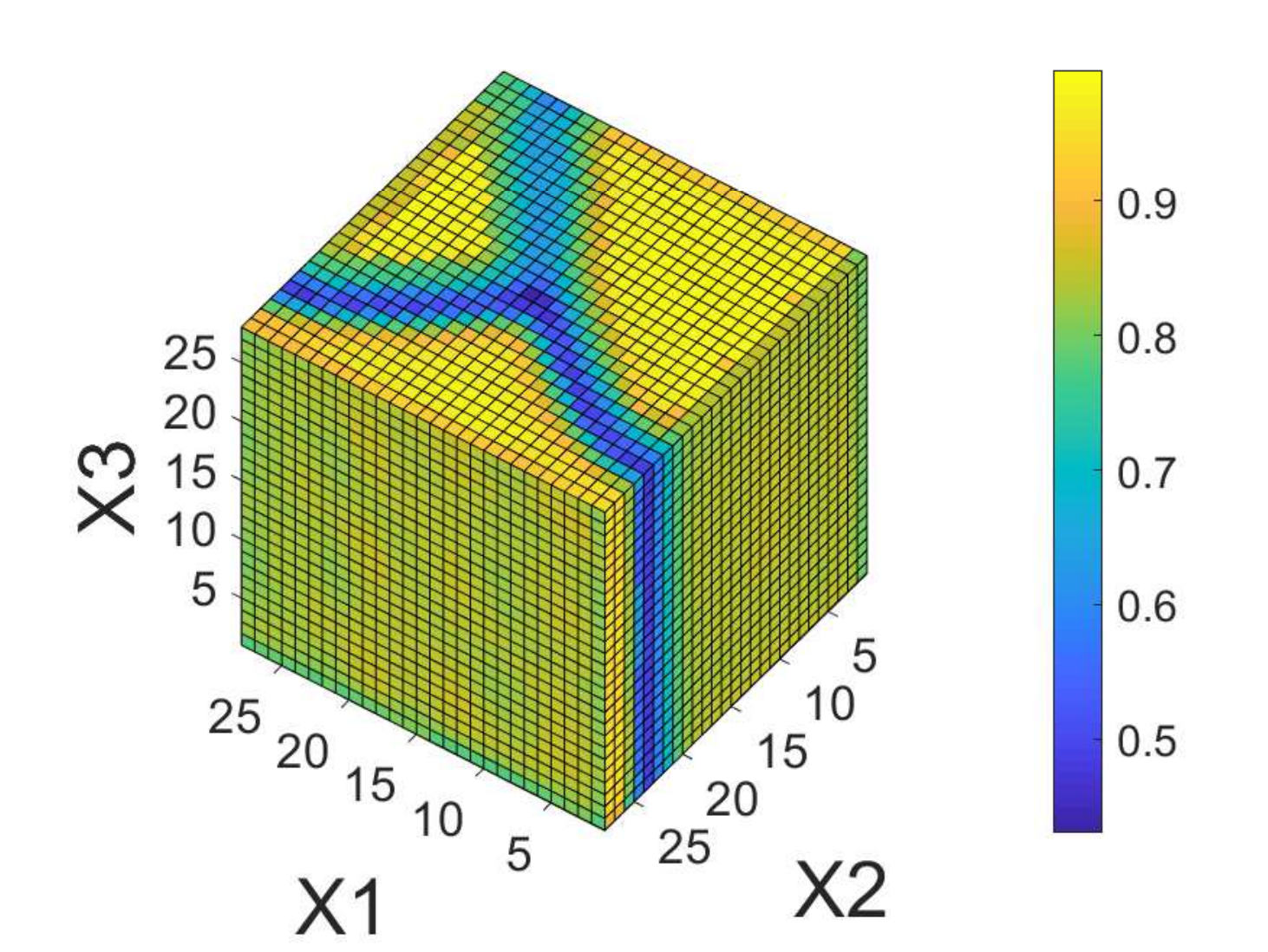} &
      \includegraphics[width=1.8in]{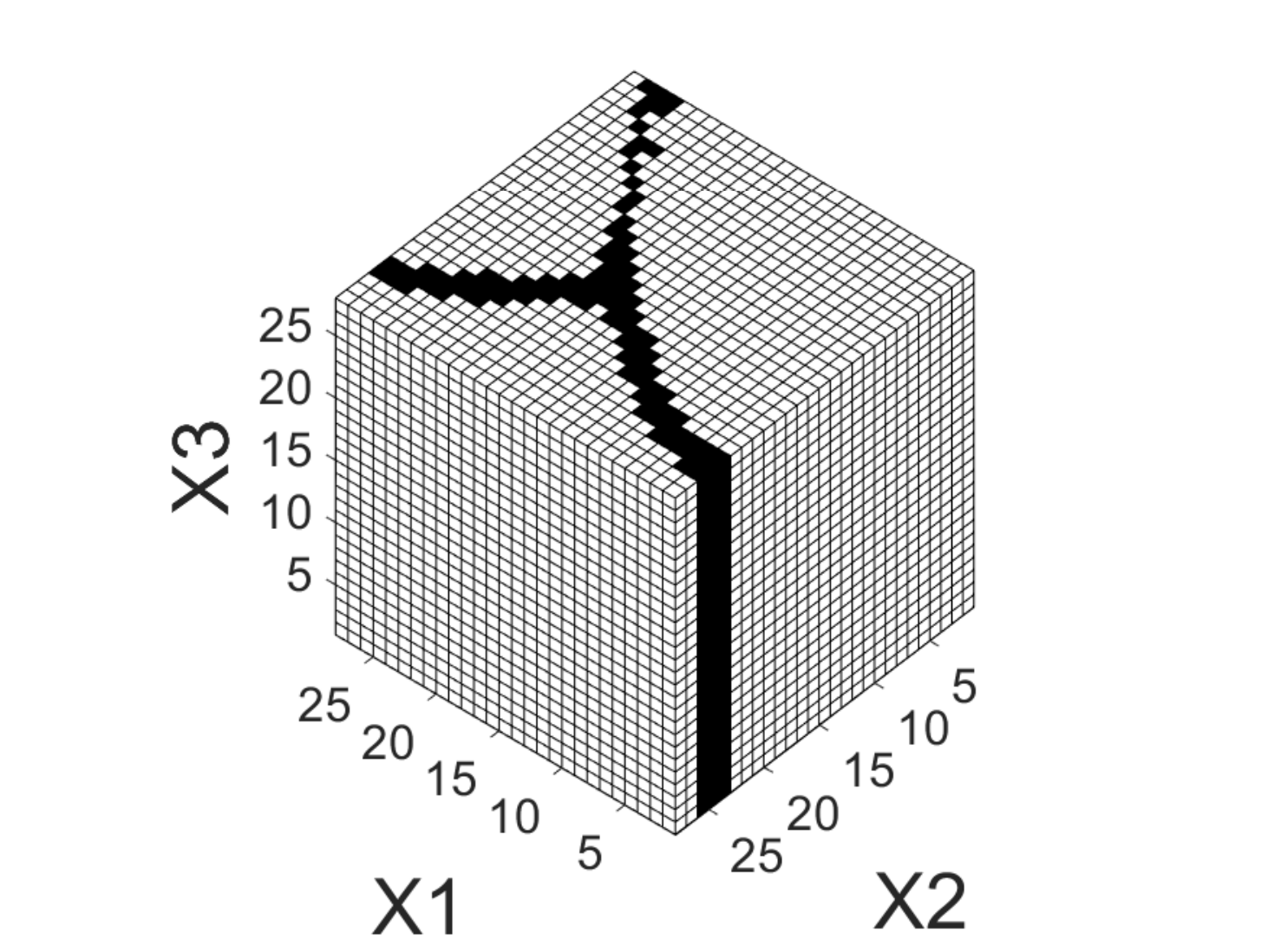} 
    \end{tabular}
  \end{center}
  \caption{Left: an undeformed hexagonal crystal image $f(x)$ with triple junction grain boundaries. Middle: $\text{mass}(x)$ of $f(x)$. Right: identified grain boundaries by thresholding $\text{mass}(x)$.}
  \label{fig:tri1}
\end{figure}

\begin{figure}[ht!]
  \begin{center}
    \begin{tabular}{ccc}
      \includegraphics[width=1.8in]{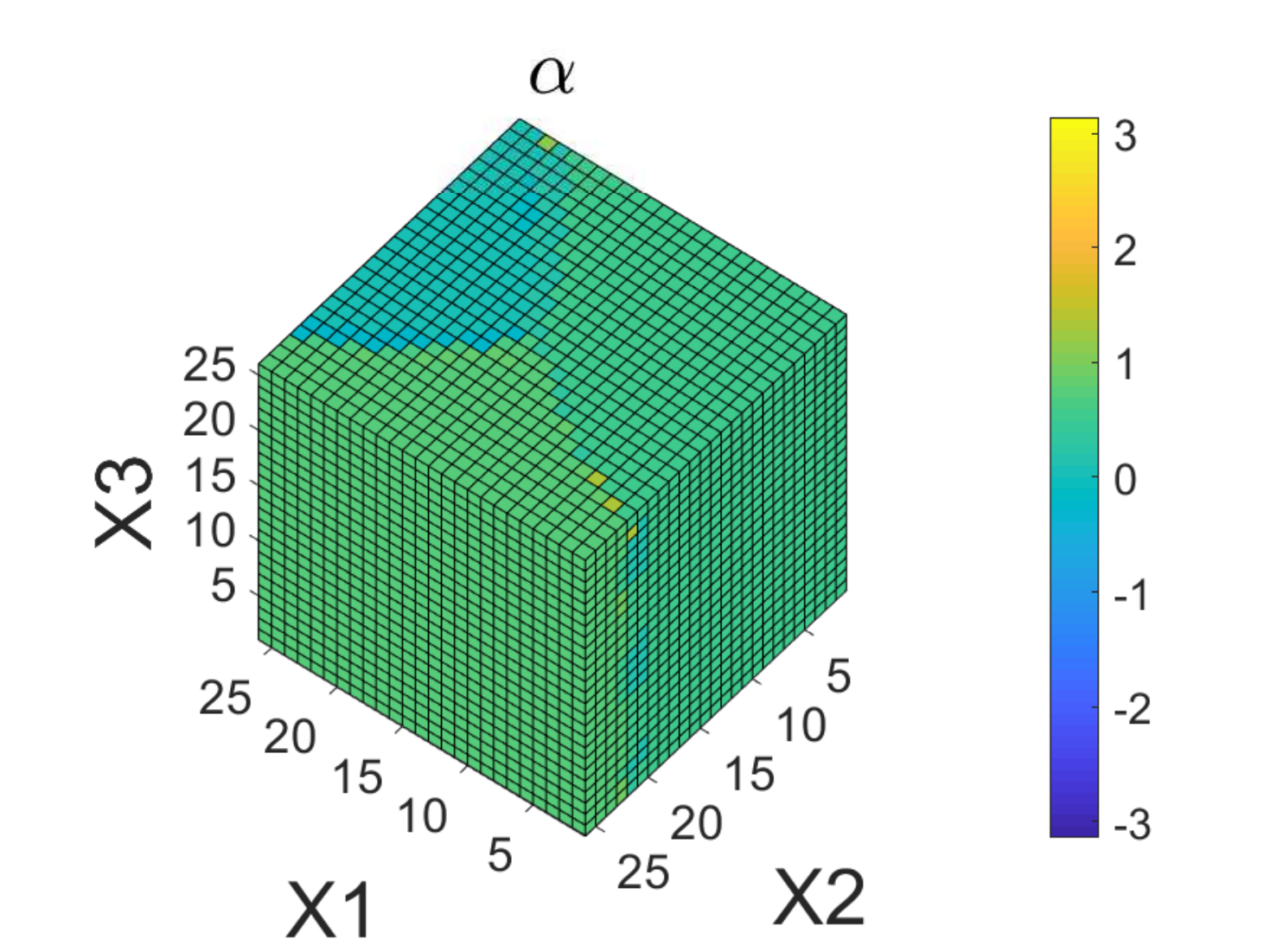} &
      \includegraphics[width=1.8in]{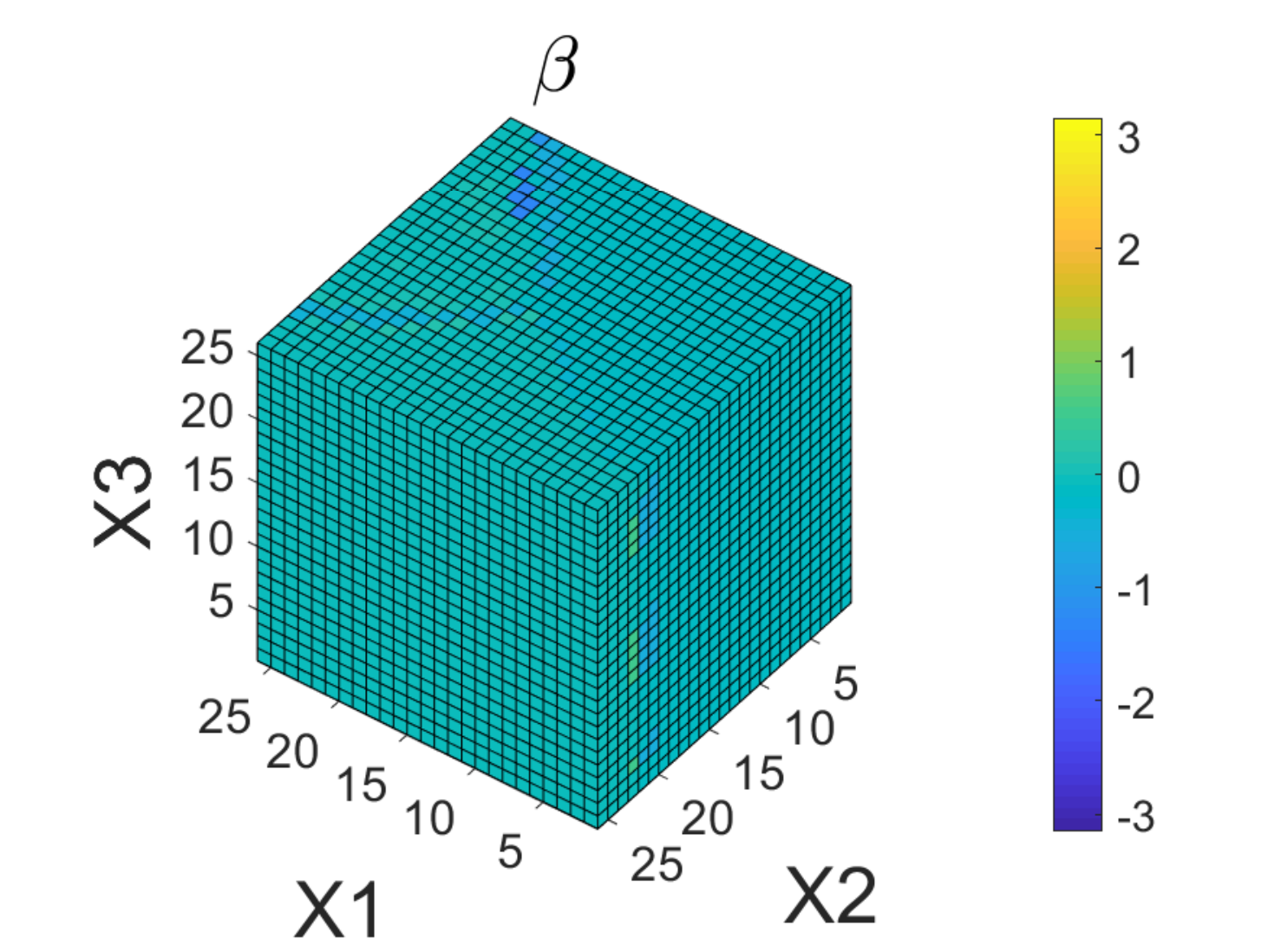} &
      \includegraphics[width=1.8in]{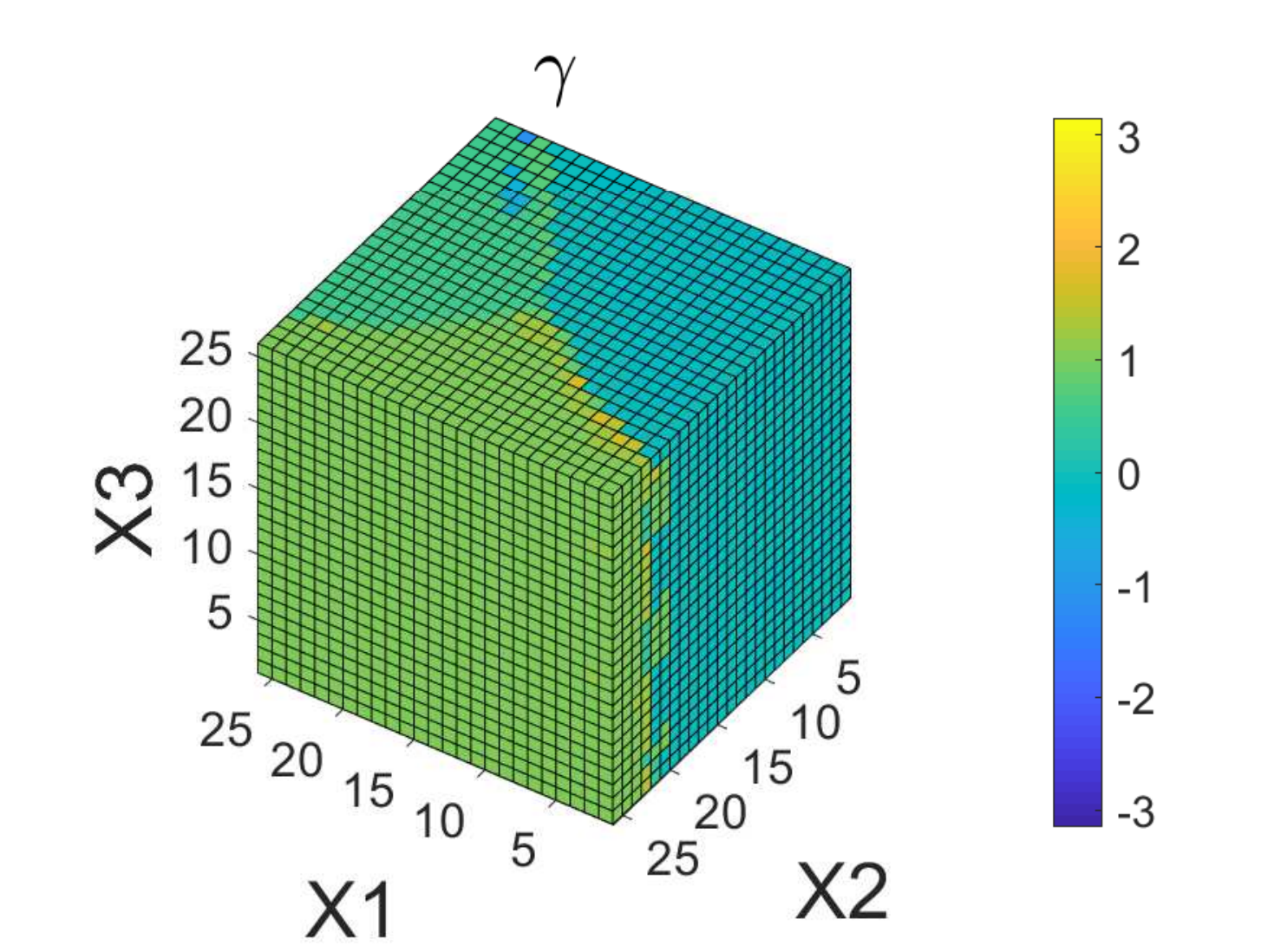} 
    \end{tabular}
  \end{center}
  \caption{Estimated Euler angles of the crystal orientation of the crystal image in Figure \ref{fig:tri1} (left). From left to right: $\alpha(x)$, $\beta(x)$, and $\gamma(x)$, respectively.}
  \label{fig:tri2}
\end{figure}

\begin{figure}[ht!]
  \begin{center}
    \begin{tabular}{ccc}
      \includegraphics[width=1.8in]{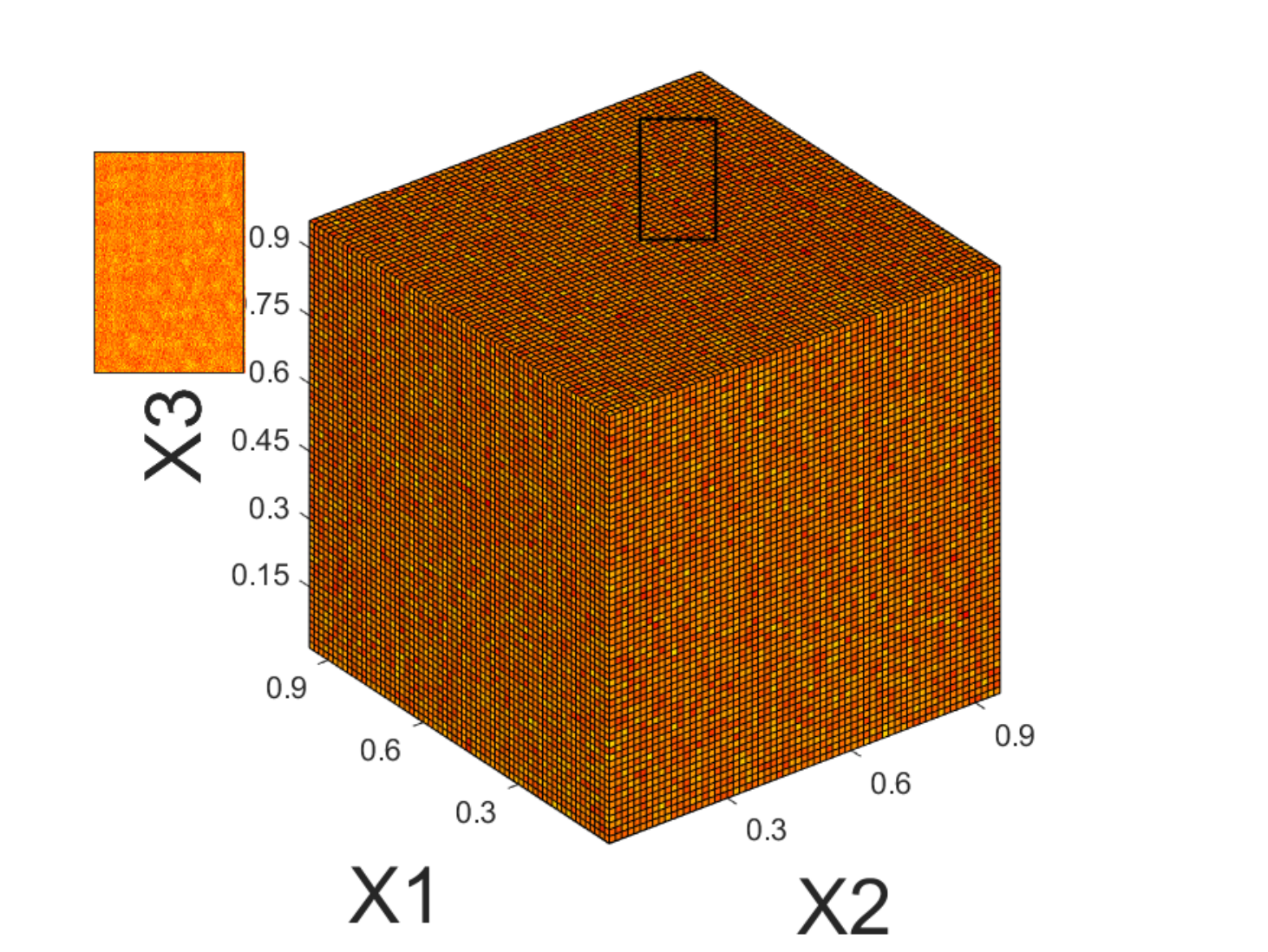} &
      \includegraphics[width=1.8in]{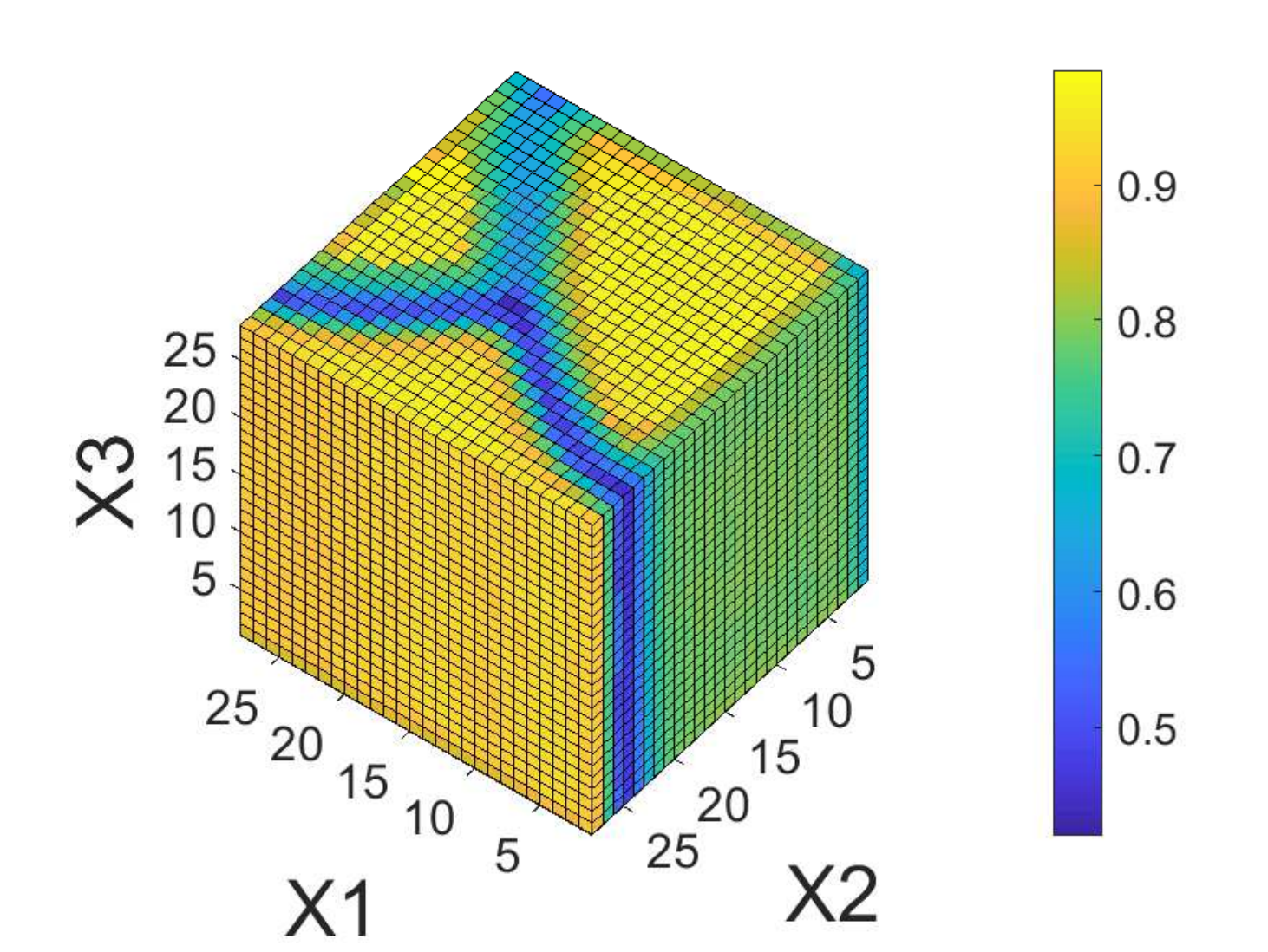} &
      \includegraphics[width=1.8in]{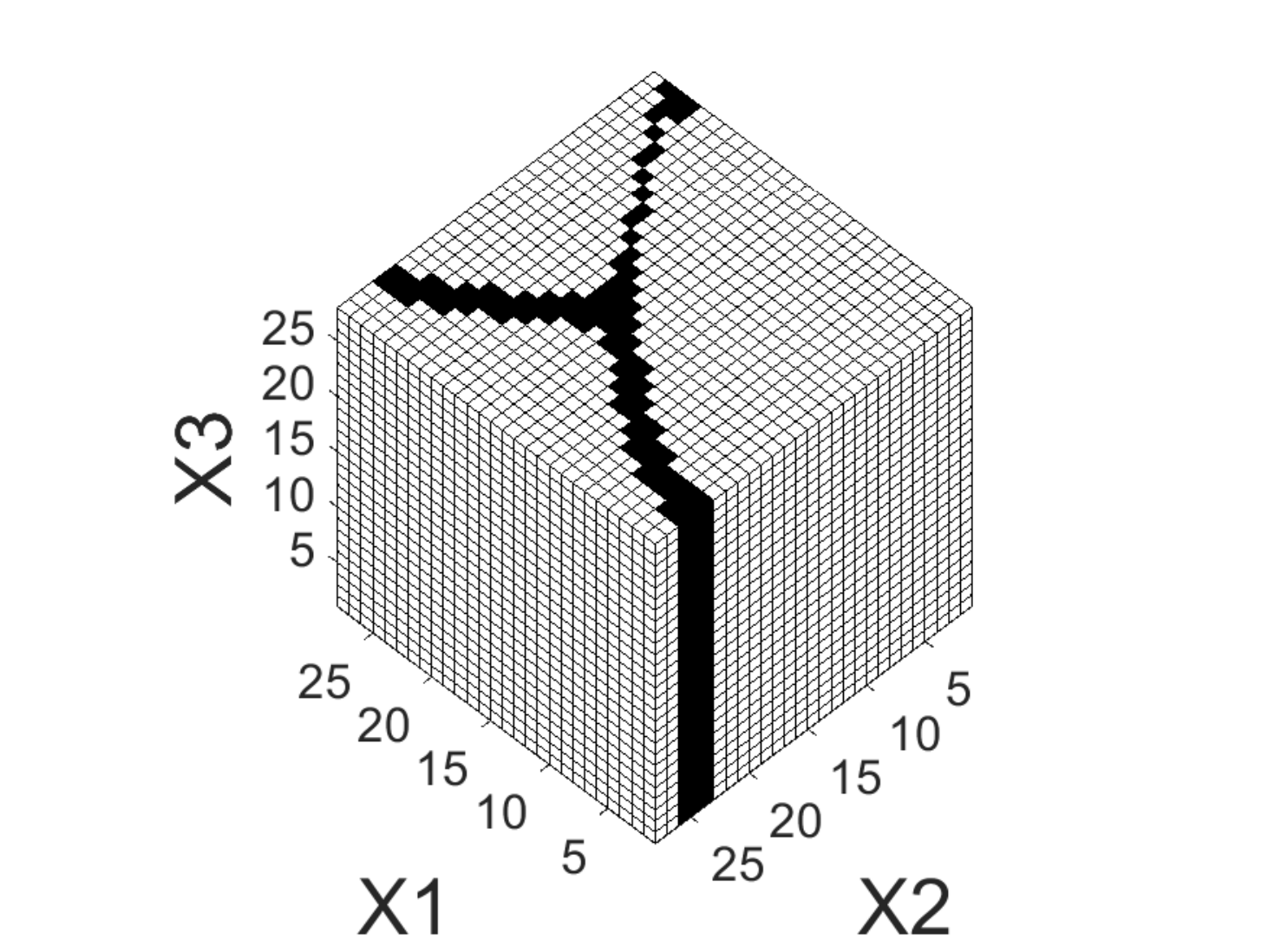} 
    \end{tabular}
  \end{center}
  \caption{Left: an undeformed hexagonal crystal image $f(x)$ with triple junction grain boundaries. Mean-zero Gaussian random noise $ns(x)$ with variance $1$ is added to the crystal image. Middle: $\text{mass}(x)$ of $f(x)+ns(x)$. Right: identified grain boundaries by thresholding $\text{mass}(x)$.}
  \label{fig:tri3}
\end{figure}

\begin{figure}[ht!]
  \begin{center}
    \begin{tabular}{ccc}
      \includegraphics[width=1.8in]{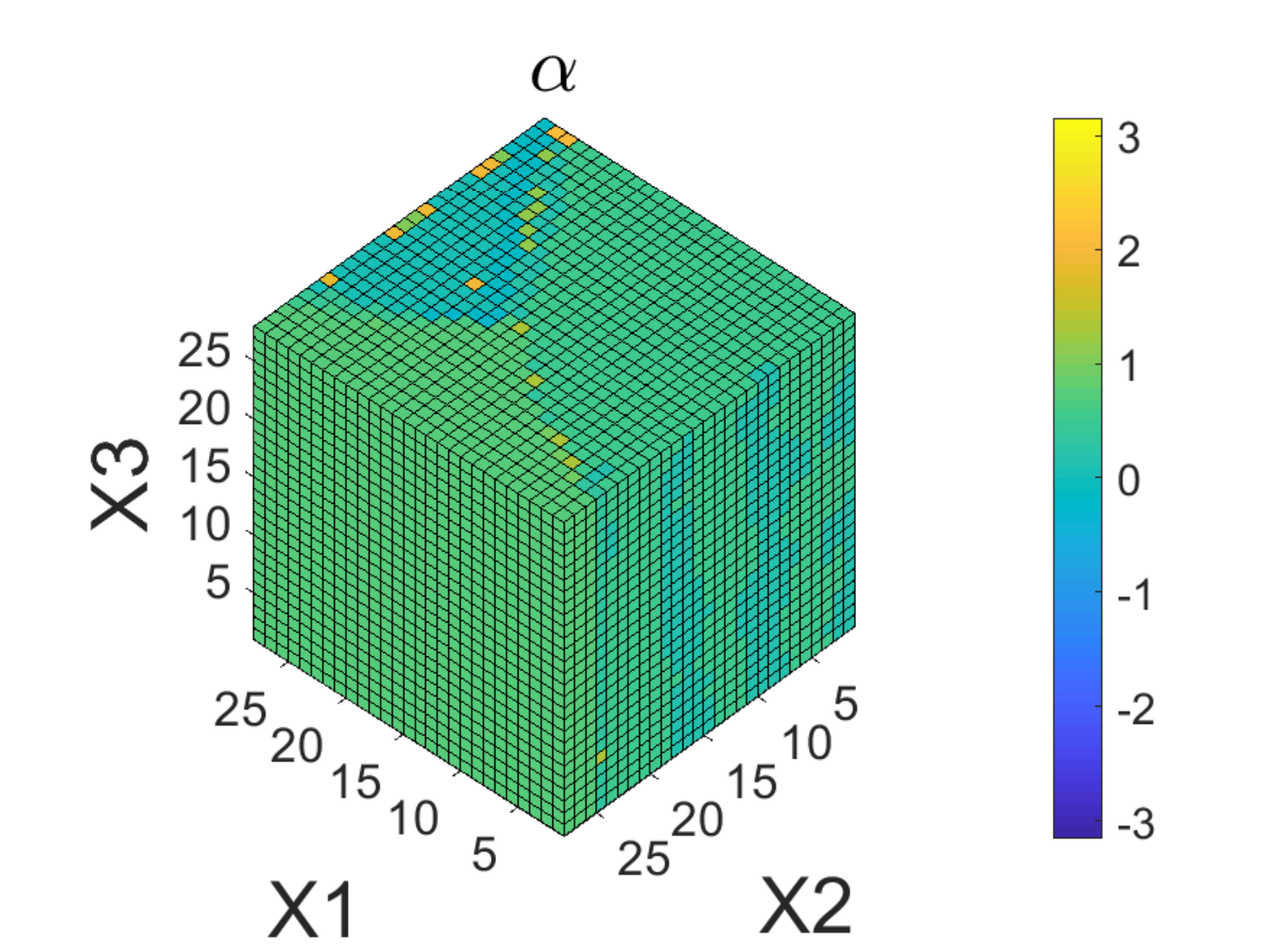} &
      \includegraphics[width=1.8in]{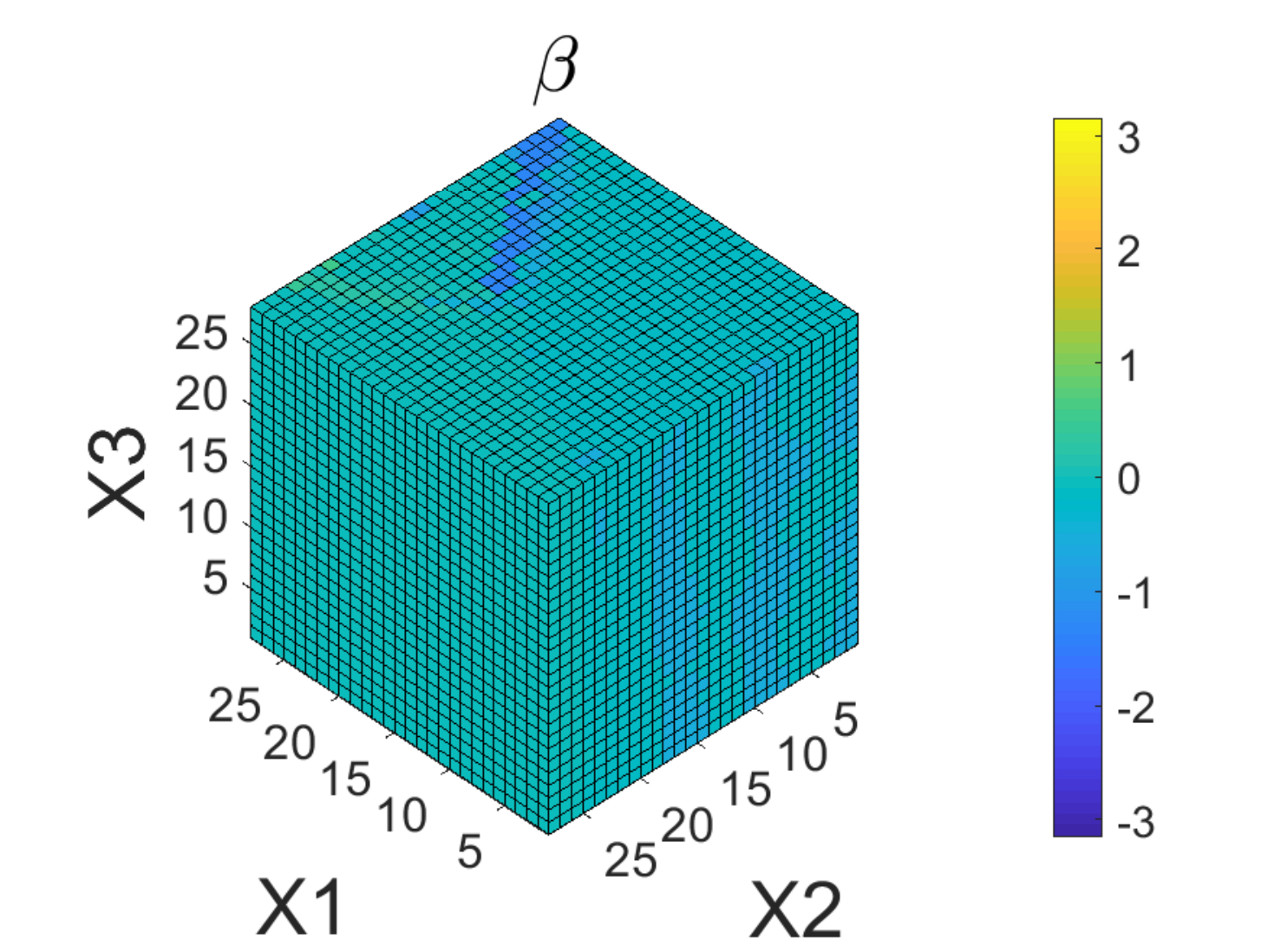} &
      \includegraphics[width=1.8in]{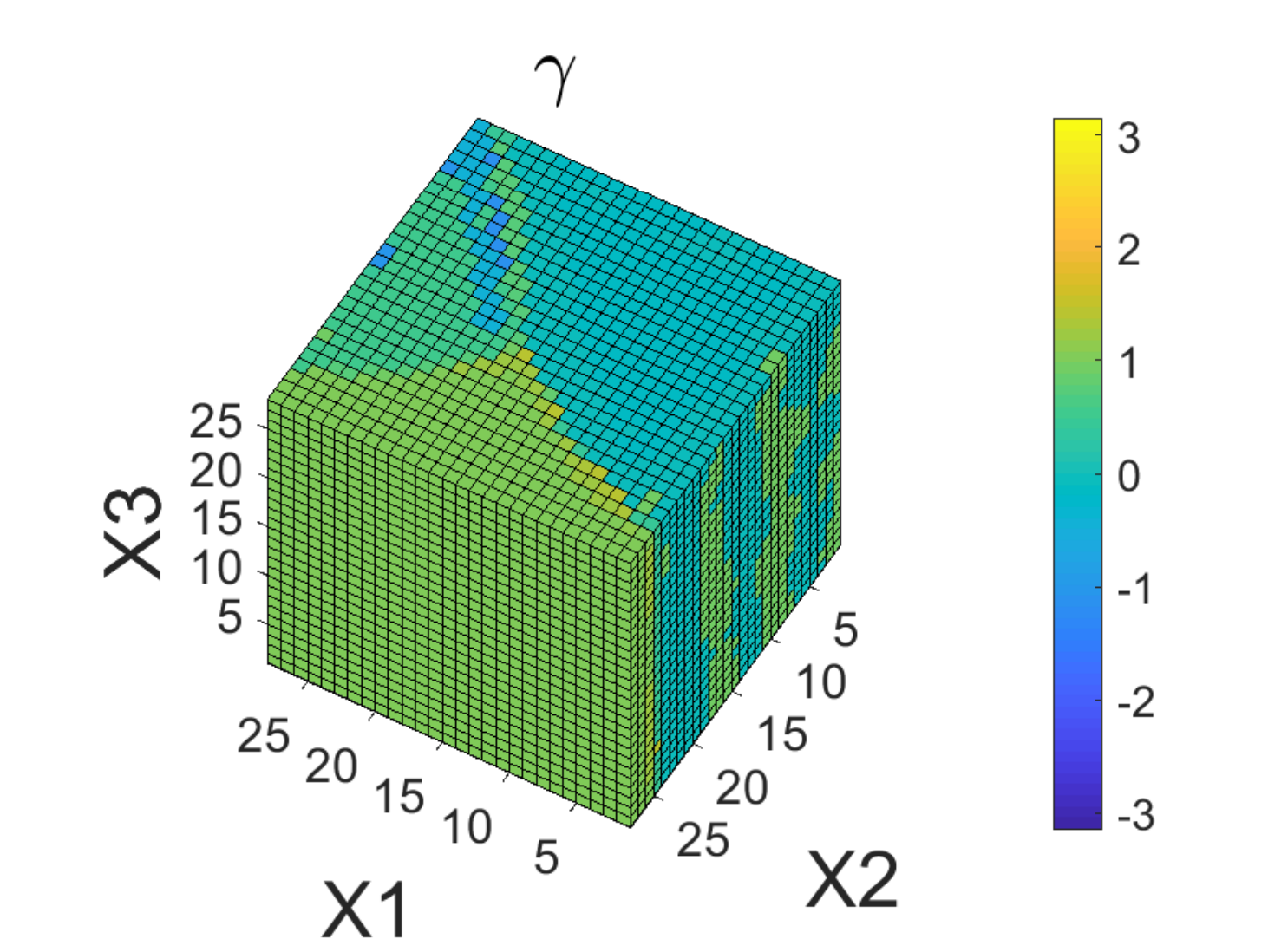} 
    \end{tabular}
  \end{center}
  \caption{Estimated Euler angles of the crystal orientation of the crystal image in Figure \ref{fig:tri3} (left). From left to right: $\alpha(x)$, $\beta(x)$, and $\gamma(x)$, respectively.}
  \label{fig:tri4}
\end{figure}

\subsubsection{Isolated defects}

Figure \ref{fig:is1} (left) shows an example of two isolated defects: one is a vacancy and the other one is a dislocation. As shown in Figure \ref{fig:is1} (middle), $\text{mass}(x)$ clearly visualizes the defect locations and after thresholding it gives the location of the defects in Figure \ref{fig:is1} (right). Figure \ref{fig:is2} shows the estimated Euler angles of the crystal orientation of the crystal image in Figure \ref{fig:is1} (left) (from left to right: $\alpha(x)$, $\beta(x)$, and $\gamma(x)$, respectively). A noisy version of the example of isolated defects and its analysis results are given in Figure \ref{fig:is3} and \ref{fig:is4}. The results in Figure \ref{fig:is3} and \ref{fig:is4} are comparable to those in Figure \ref{fig:is1} and \ref{fig:is2}, even though the atom structure is not clear with heavy noise.

\begin{figure}[ht!]
  \begin{center}
    \begin{tabular}{ccc}
      \includegraphics[width=1.6in]{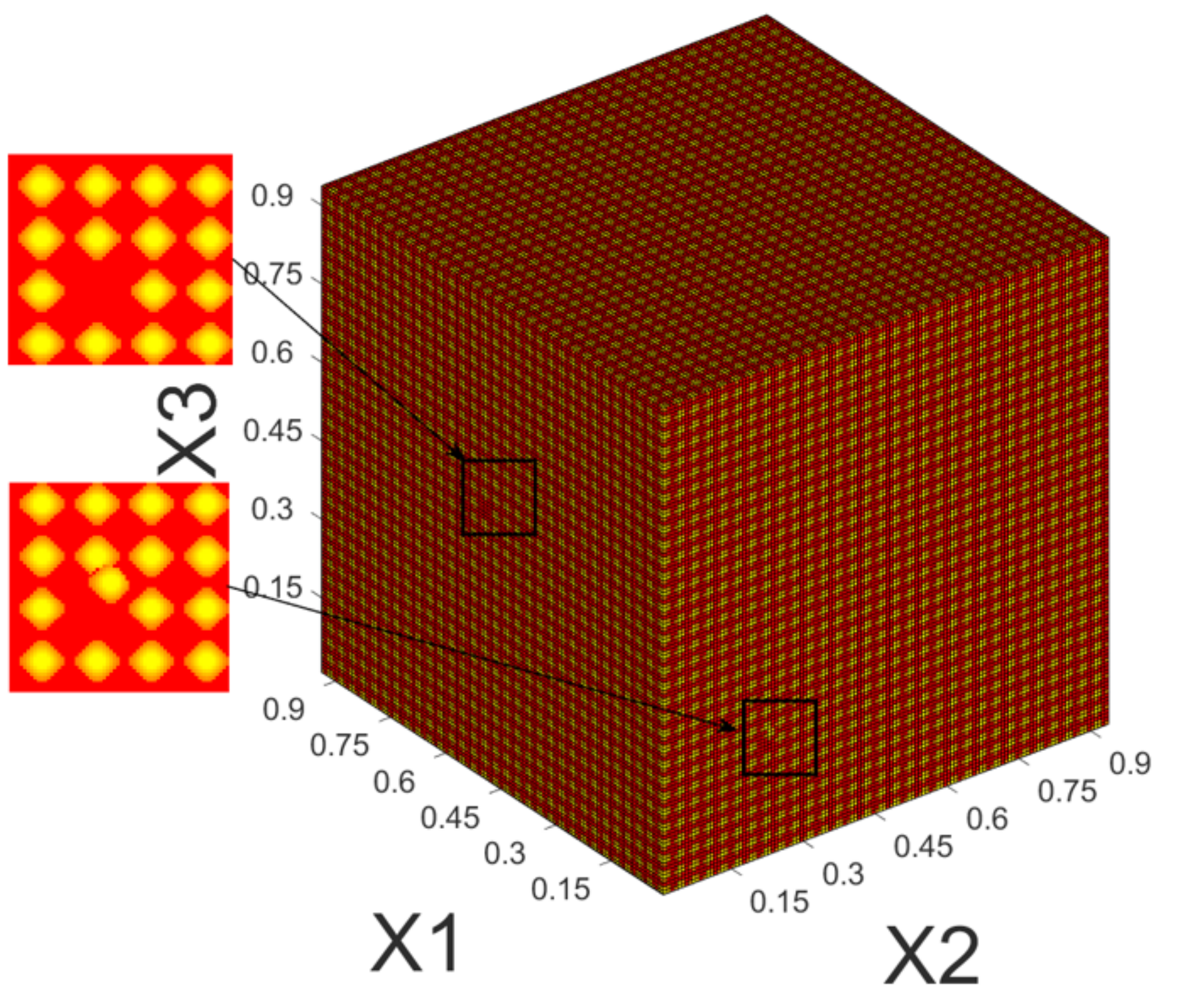} &
      \includegraphics[width=1.8in]{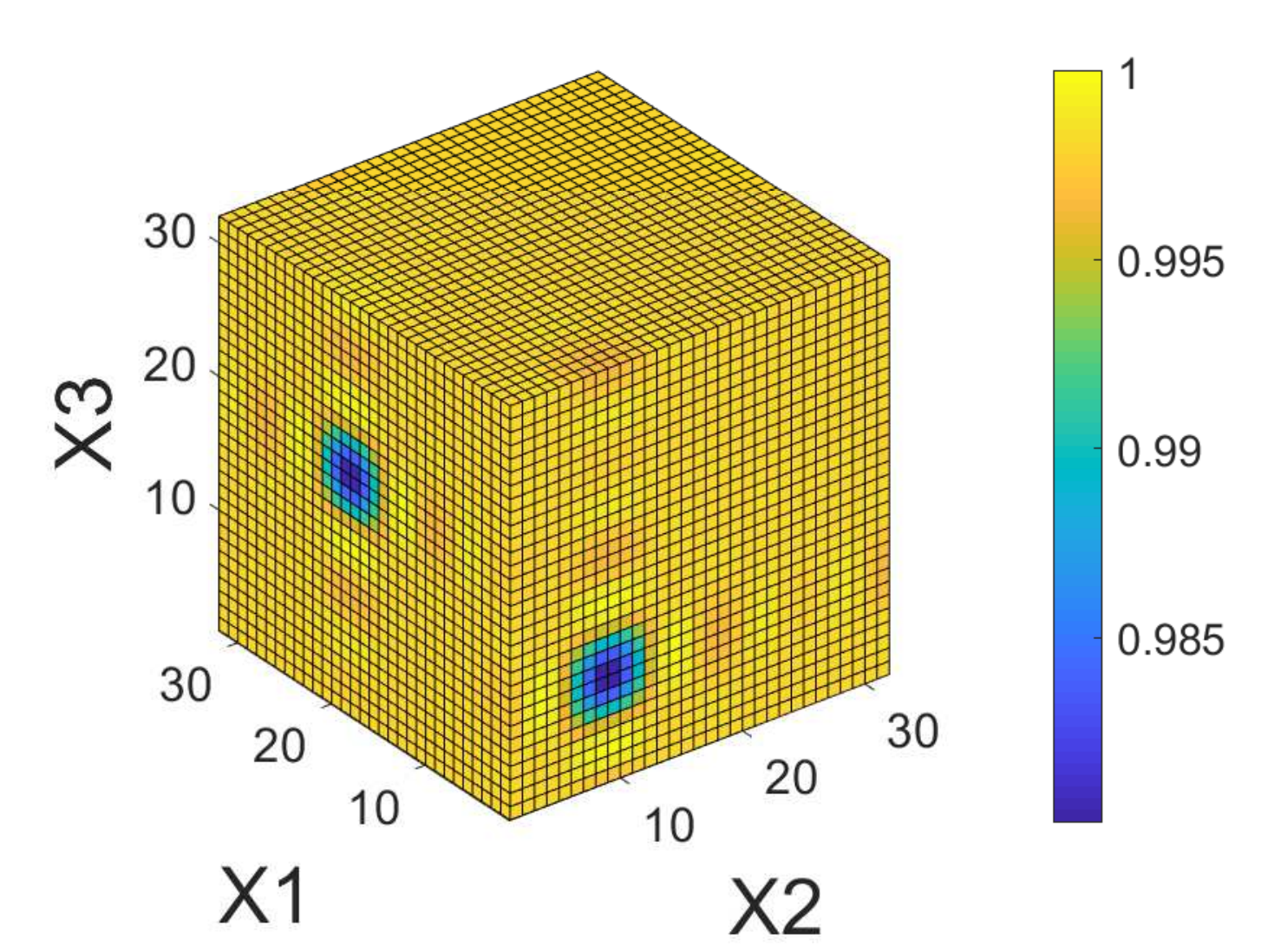} &
      \includegraphics[width=1.8in]{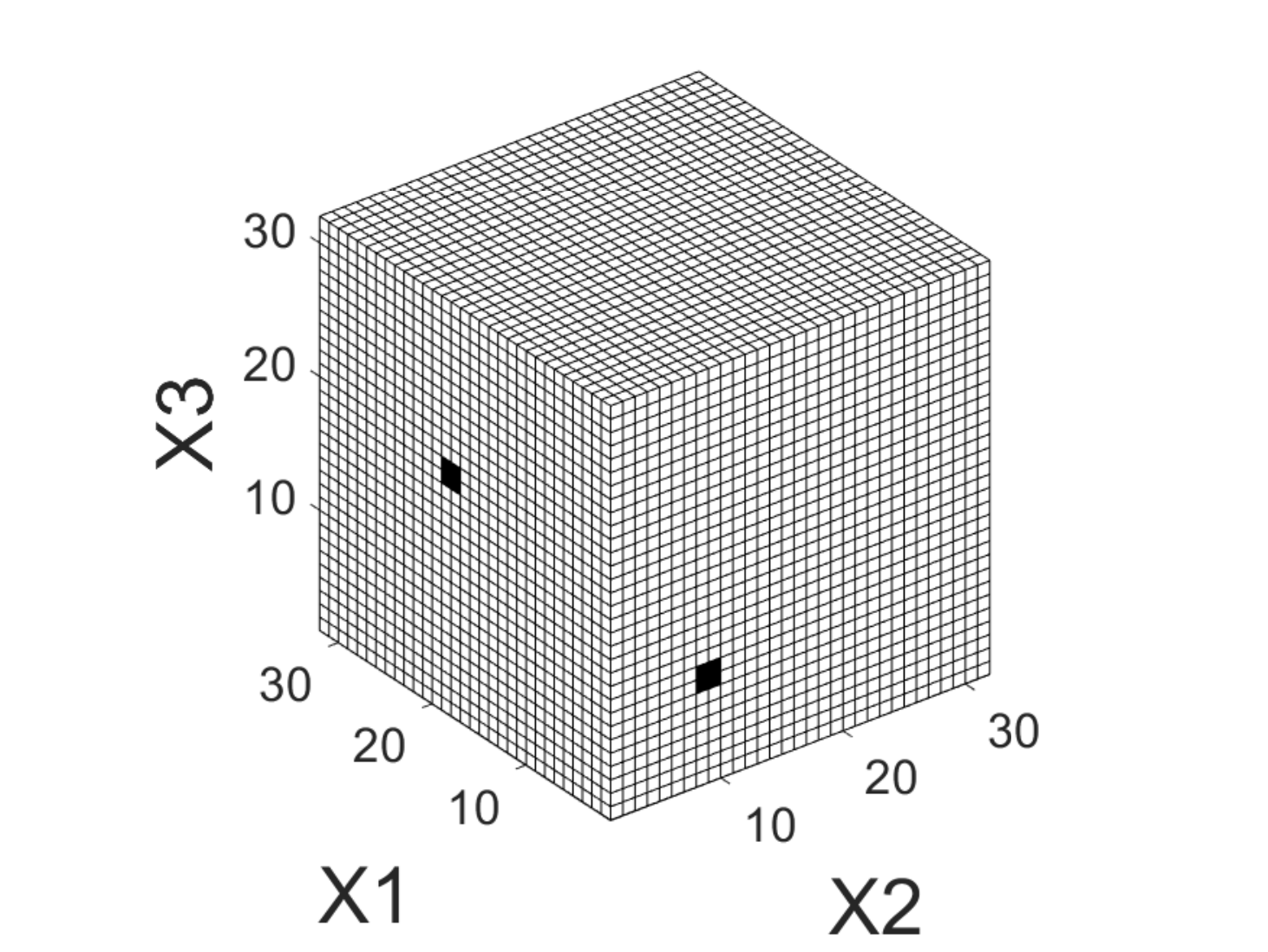} 
    \end{tabular}
  \end{center}
  \caption{Left: an undeformed cubic crystal image $f(x)$ with two isolated defects. Middle: $\text{mass}(x)$ of $f(x)$. Right: identified grain boundaries by thresholding $\text{mass}(x)$.}
  \label{fig:is1}
\end{figure}

\begin{figure}[ht!]
  \begin{center}
    \begin{tabular}{ccc}
      \includegraphics[width=1.8in]{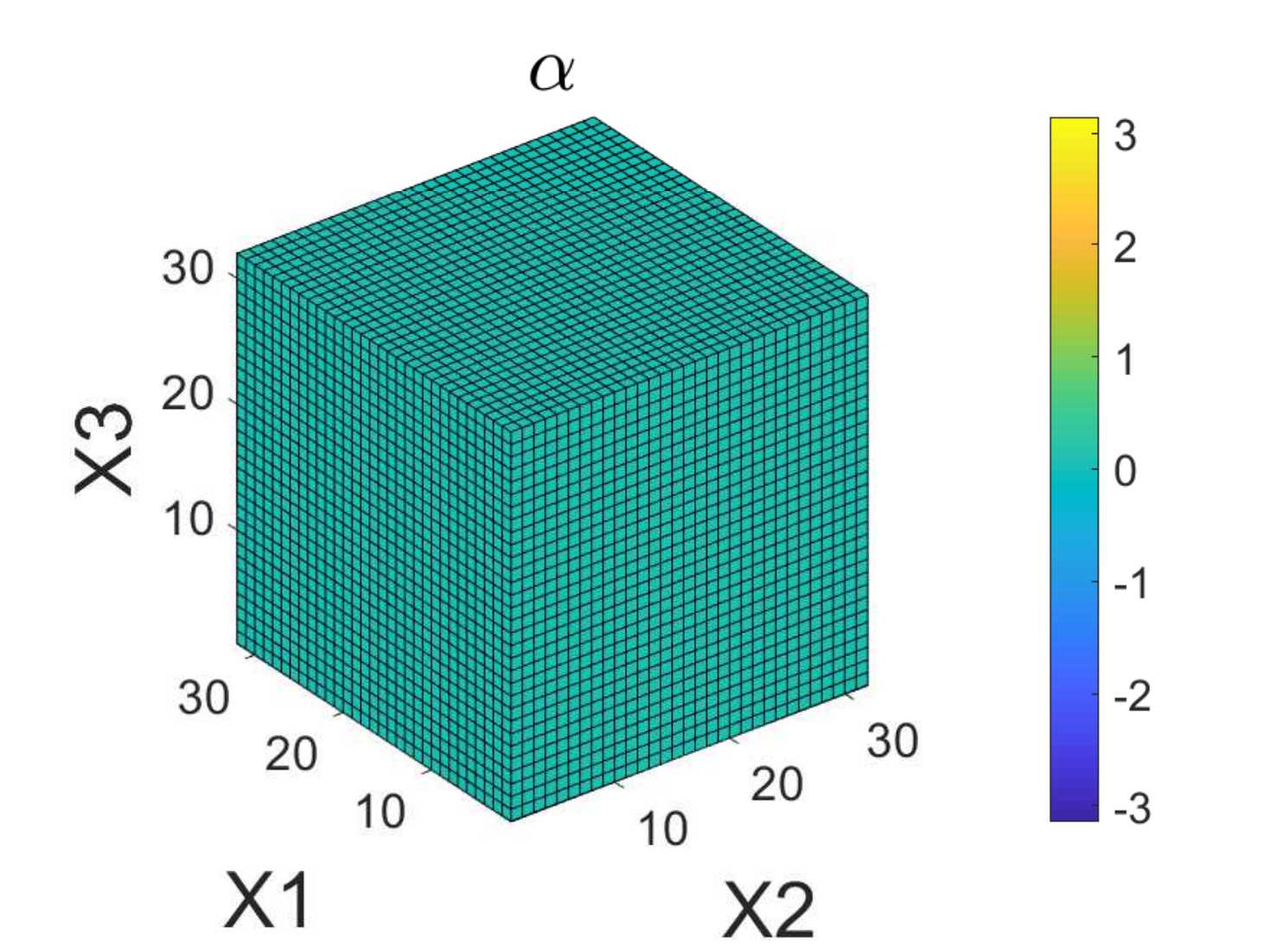} &
      \includegraphics[width=1.8in]{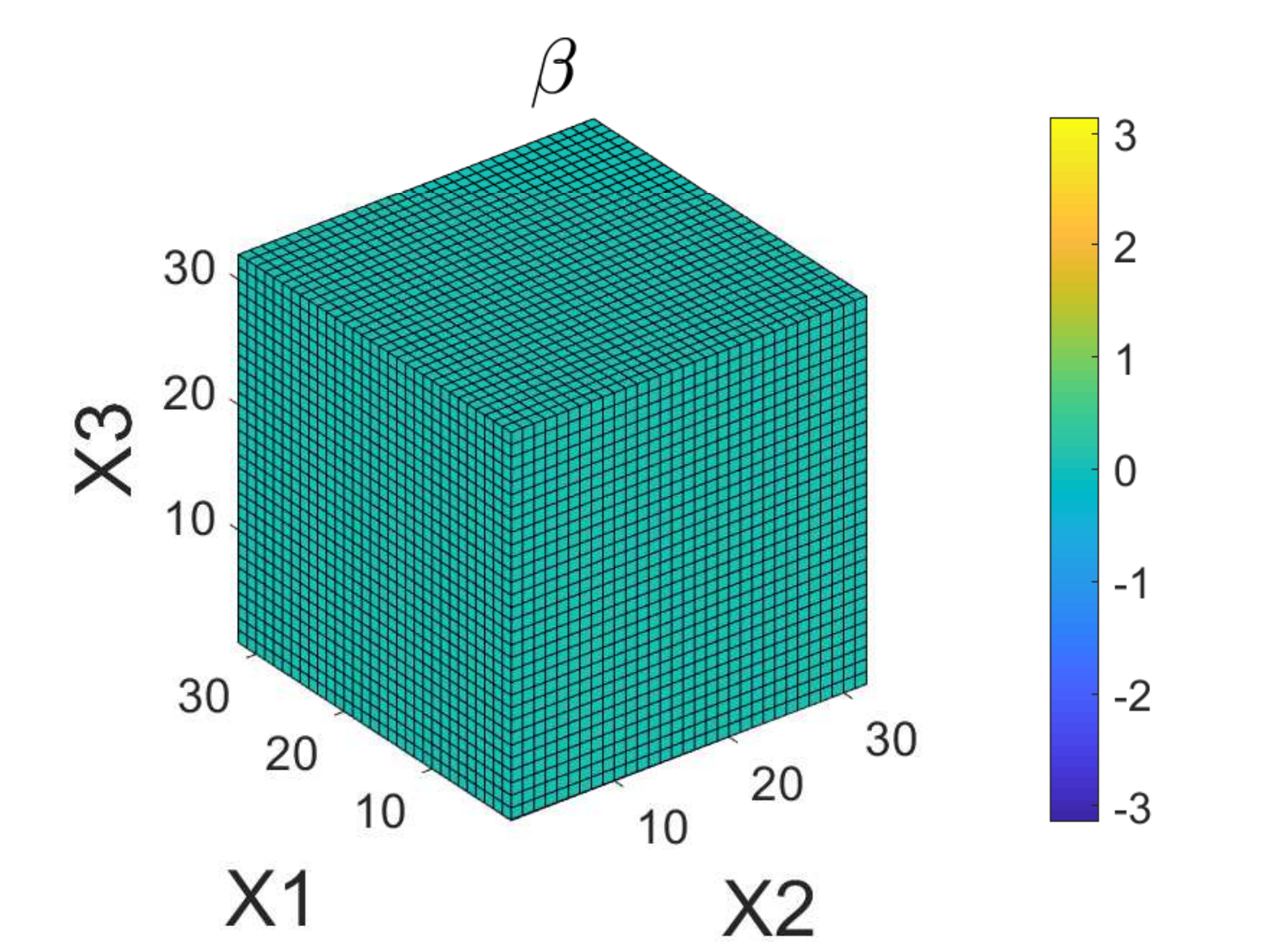} &
      \includegraphics[width=1.8in]{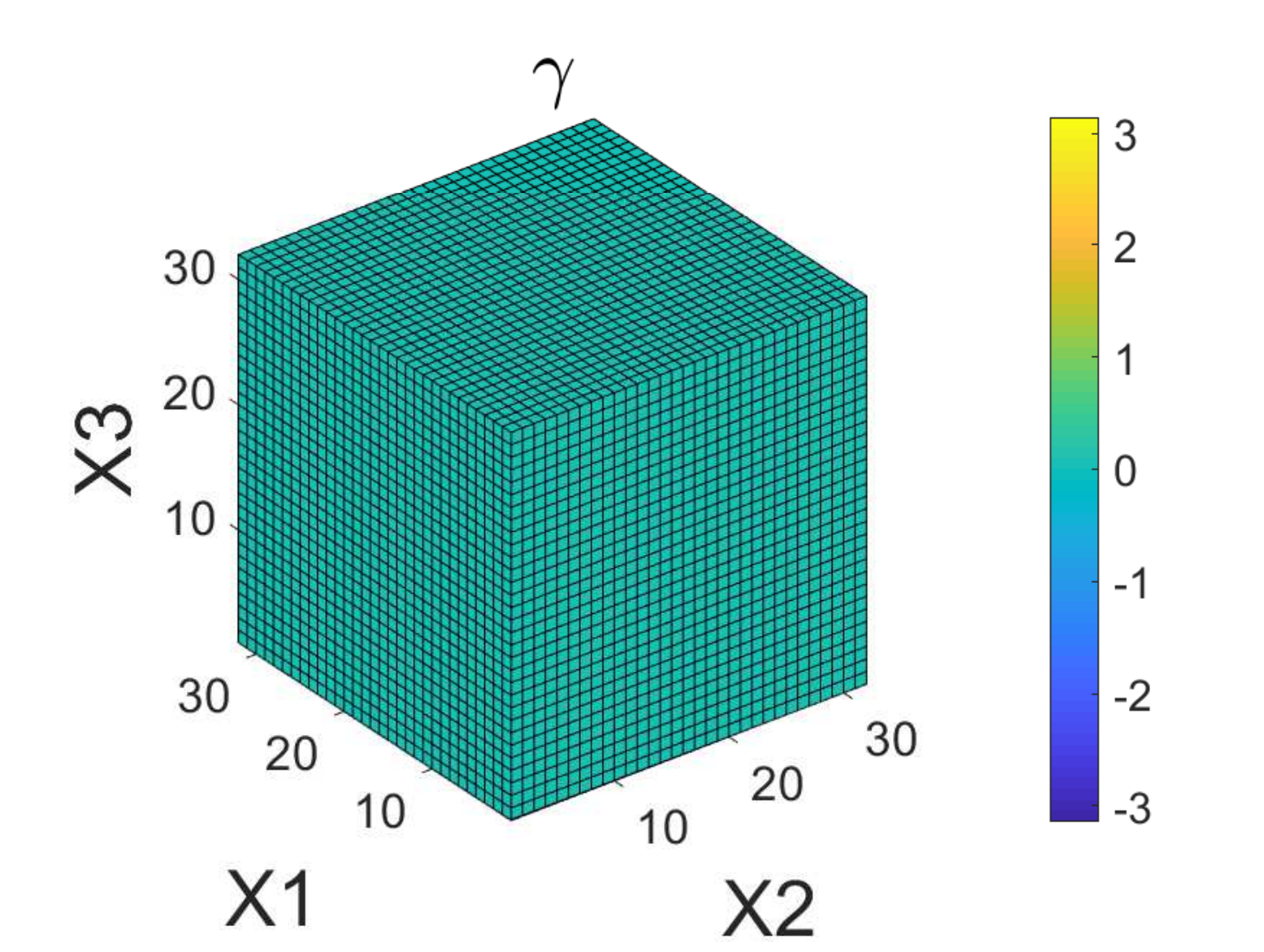} 
    \end{tabular}
  \end{center}
  \caption{Estimated Euler angles of the crystal orientation of the crystal image in Figure \ref{fig:is1} (left). From left to right: $\alpha(x)$, $\beta(x)$, and $\gamma(x)$, respectively.}
  \label{fig:is2}
\end{figure}
\begin{figure}[ht!]
  \begin{center}
    \begin{tabular}{ccc}
      \includegraphics[width=1.6in]{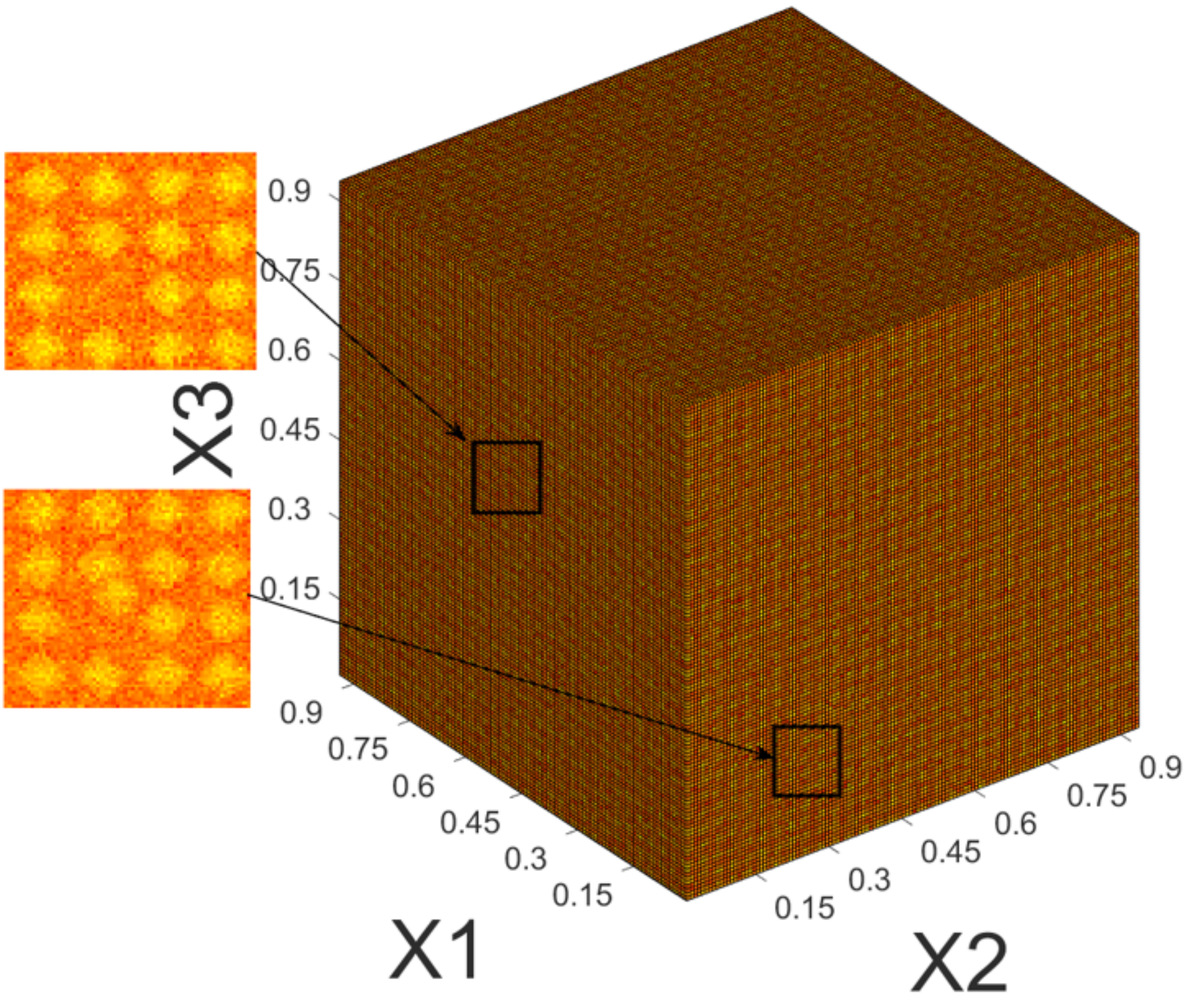} &
      \includegraphics[width=1.8in]{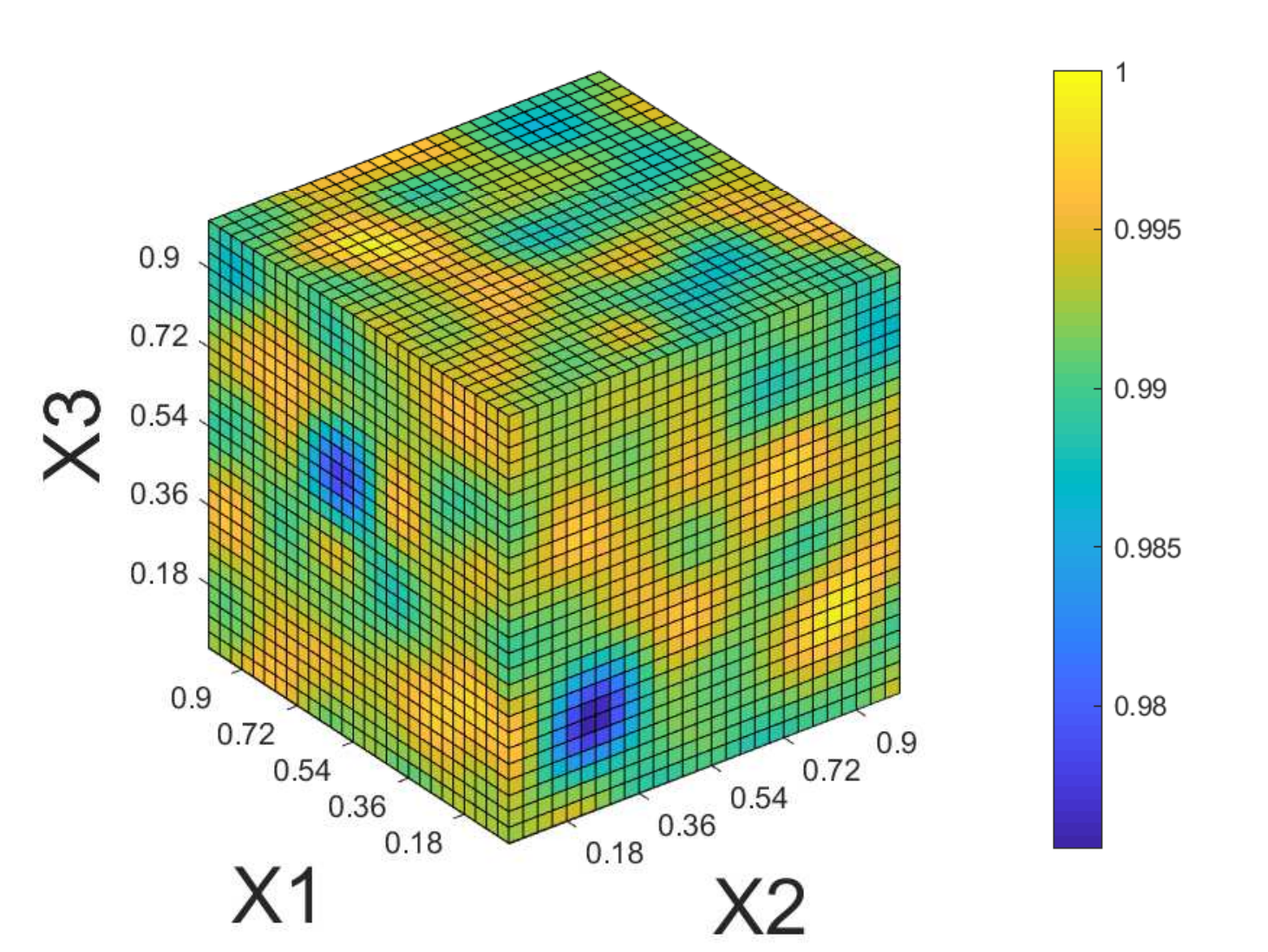} &
      \includegraphics[width=1.8in]{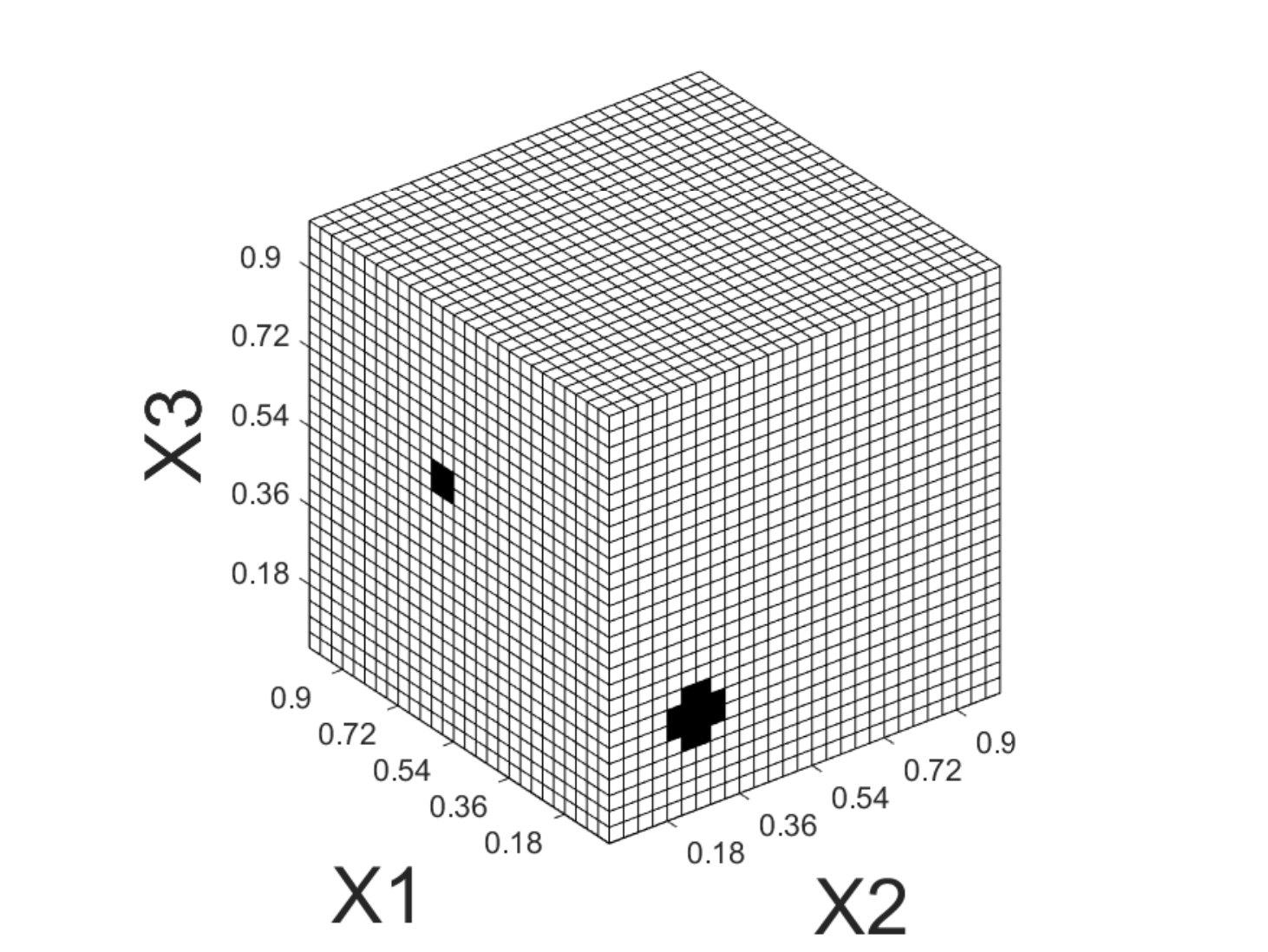} 
    \end{tabular}
  \end{center}
  \caption{Left: an undeformed cubic crystal image $f(x)$ with two isolated defects and additive Gaussian random noise $ns(x)$ with mean zero and variance $0.3$. Middle: $\text{mass}(x)$ of $f(x)+ns(x)$. Right: identified grain boundaries by thresholding $\text{mass}(x)$.}
  \label{fig:is3}
\end{figure}

\begin{figure}[ht!]
  \begin{center}
    \begin{tabular}{ccc}
      \includegraphics[width=1.8in]{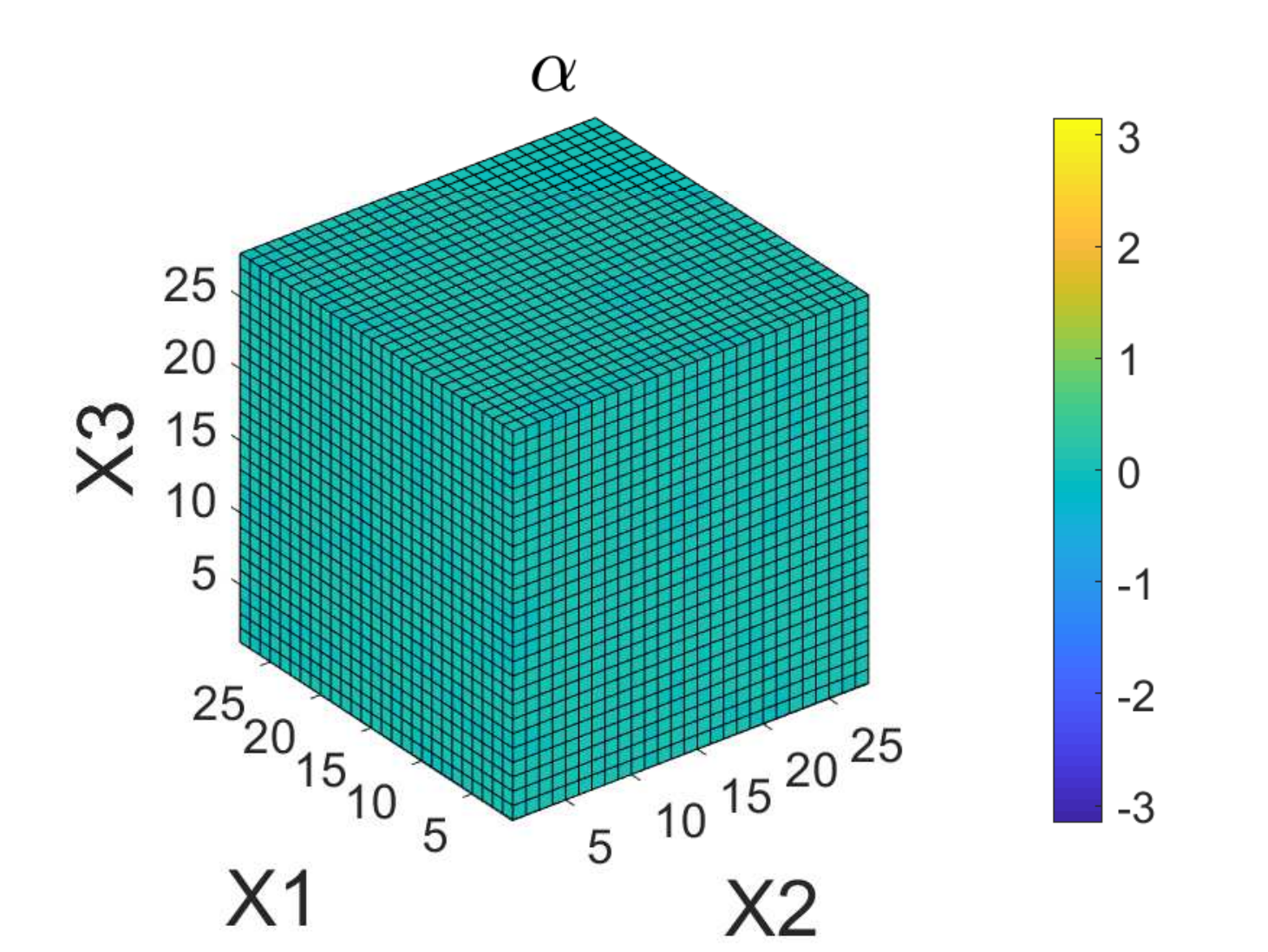} &
      \includegraphics[width=1.8in]{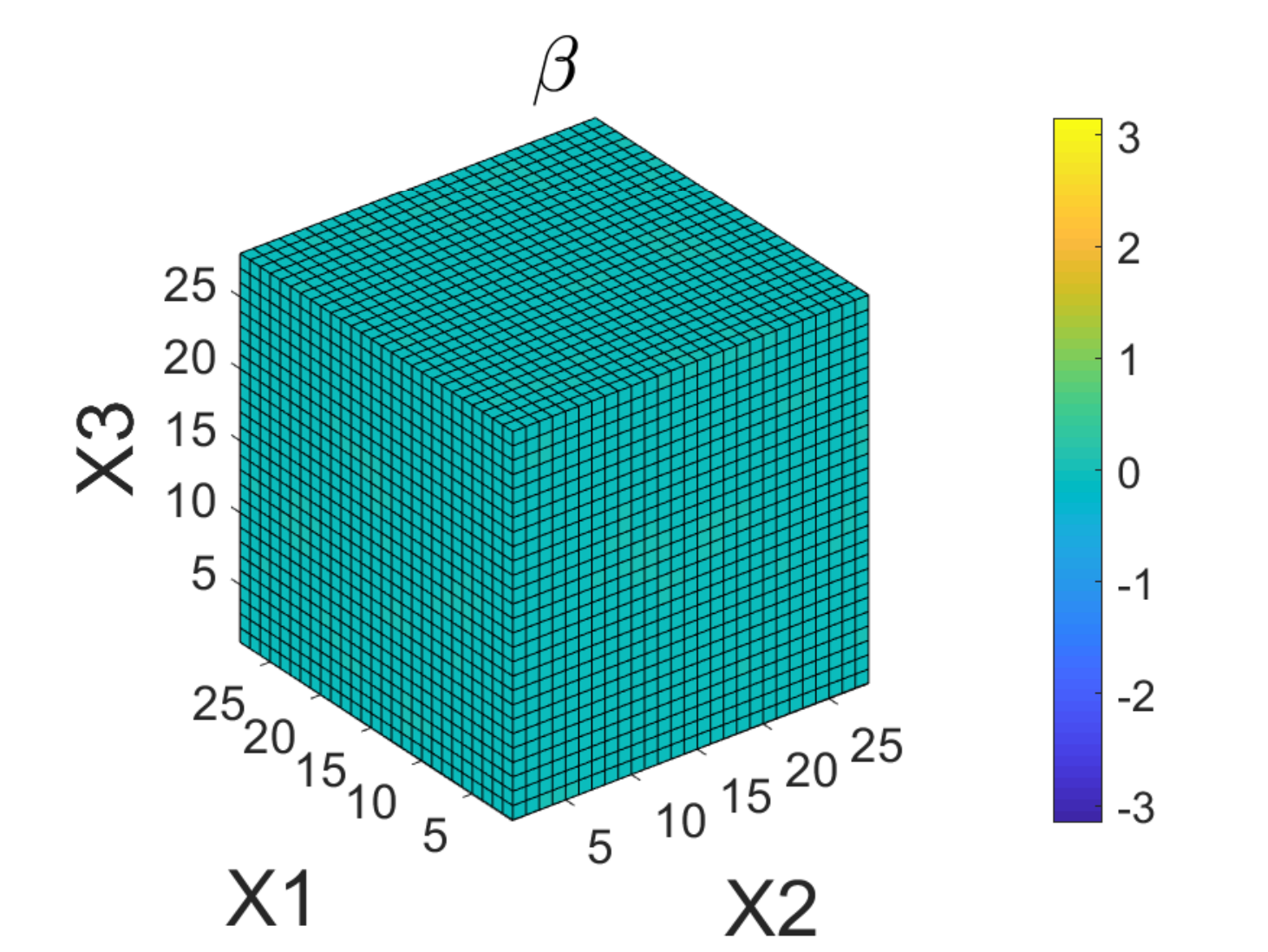} &
      \includegraphics[width=1.8in]{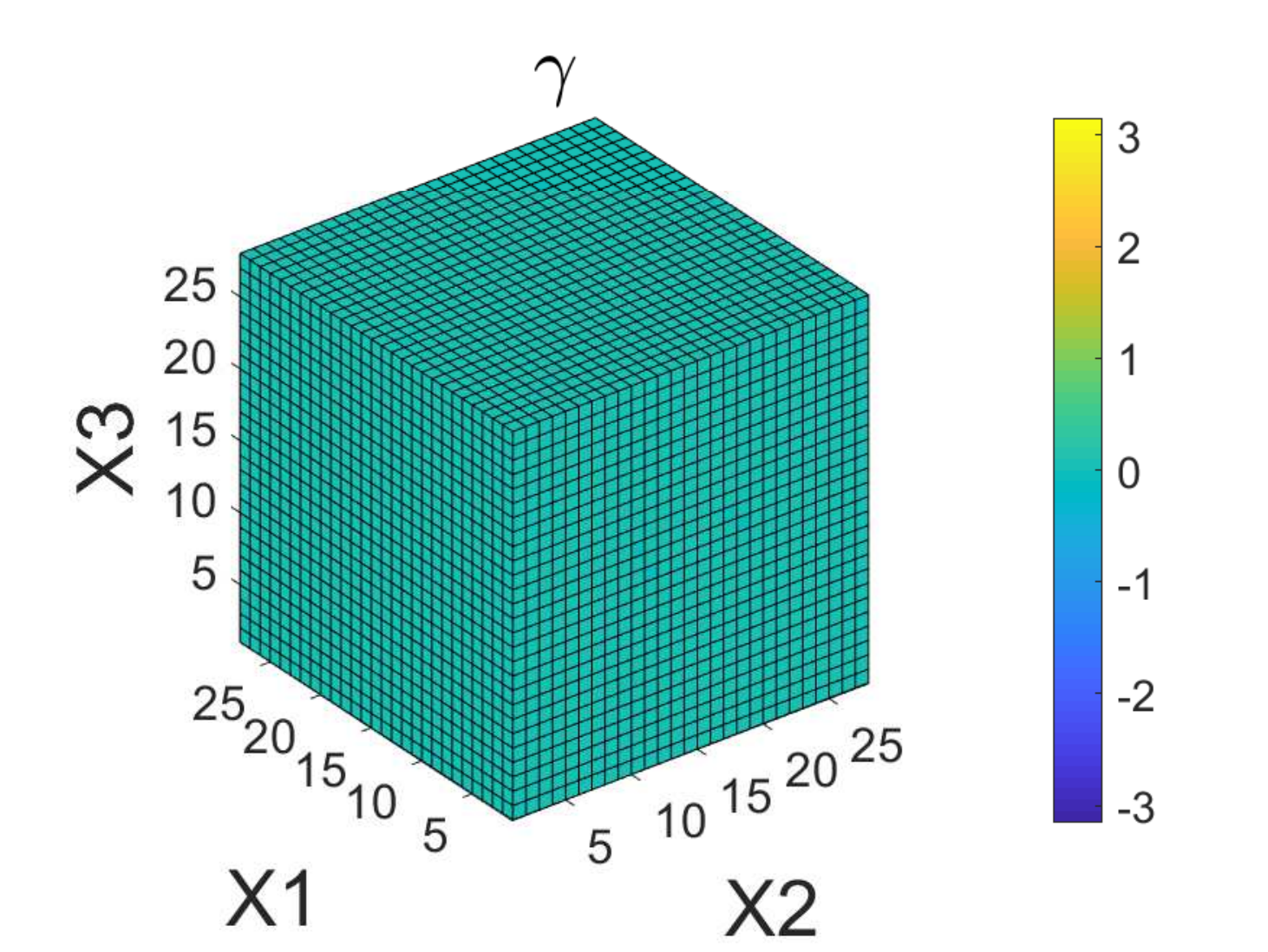} 
    \end{tabular}
  \end{center}
  \caption{Estimated Euler angles of the crystal orientation of the crystal image in Figure \ref{fig:is3} (left). From left to right: $\alpha(x)$, $\beta(x)$, and $\gamma(x)$, respectively.}
  \label{fig:is4}
\end{figure}

\subsection{Real atomic resolution crystal images}
Self-assembly plays a pivotal role in biologically controlled synthesis and in fabricating advanced engineering materials \cite{Na1,Na2}. For example, self-assembly of colloidal particles admits versatile fabrication of highly ordered $2D$ and $3D$ structures for various applications, e.g., photonics, sensing, catalysis, etc. \cite{Co1,Co2}. Data analysis of crystal images from self-assembly improves the mechanistic understanding of the self-assembly processes for accurate control in lattice types, crystallography, and defects, which could further facilitate the performance and functionality of fabricated structures.

Here we present an example of $3D$ images of colloidal particles. Figure \ref{fig:rd1} shows a real example and its estimated Euler angle $\gamma(x)$. To better visualize the results, we cut the $3D$ data into $2D$ slices and show one slice per eight pixels. The results in the estimated $\gamma(x)$ is able to reflect the crystal orientation in the $x_1$-$x_3$ plane. For example, the left-bottom panel of Figure \ref{fig:rd1} shows the zoomed-in image (top) and its corresponding $\gamma(x)$ (bottom), highlighting two examples (boxed regions) of noticeable fine scale variation of the crystal orientation in $\gamma(x)$, readily recognizable
also by visual inspection of the corresponding zones in the crystal image. The crystal images in these boxed regions squeeze slightly in the upper part, resulting in slightly different orientations in $\gamma(x)$.

\begin{figure}[ht!]
  \begin{center}
    \begin{tabular}{cc}
    \includegraphics[width=1.6in]{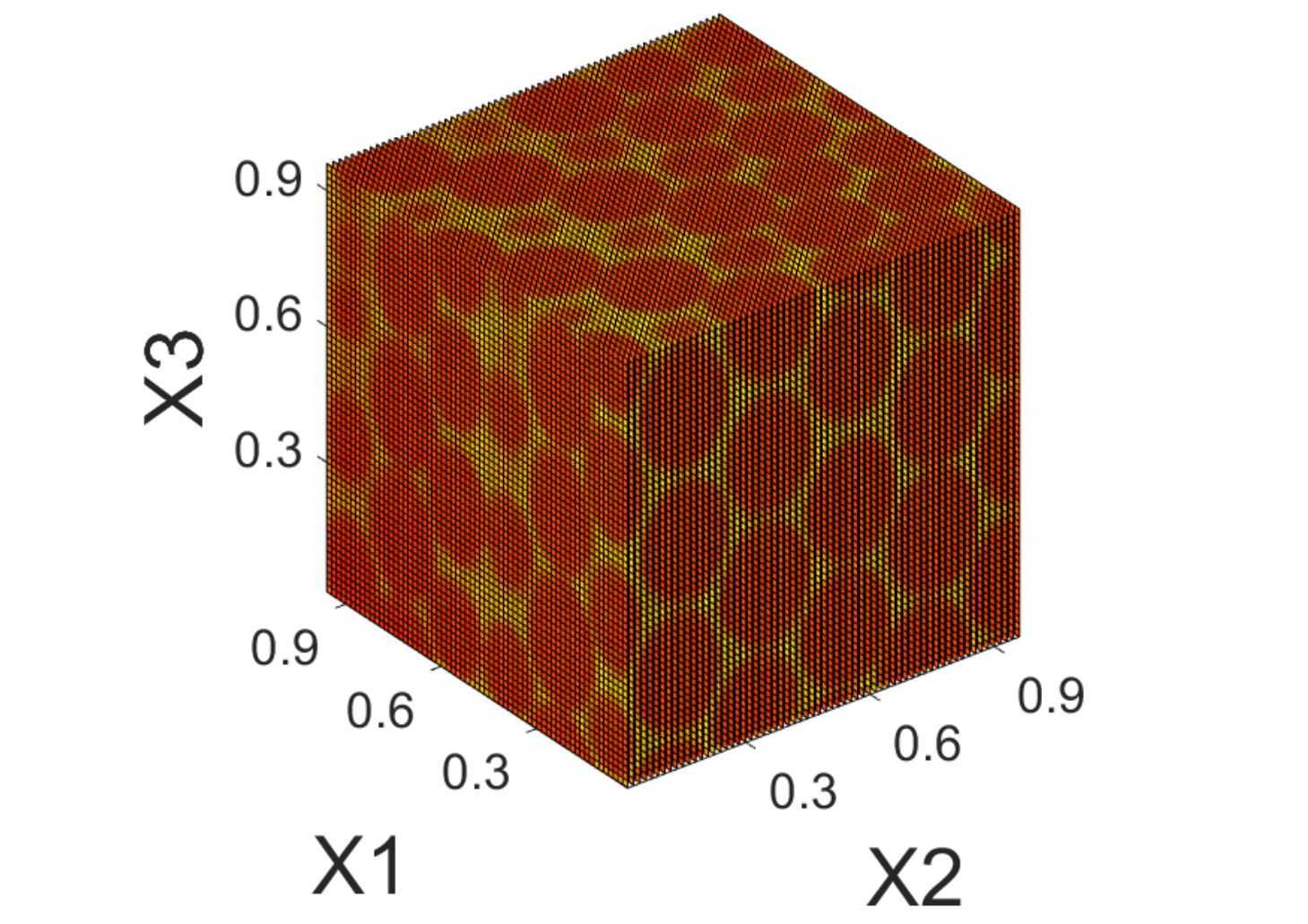}   &  \includegraphics[width=3.2in]{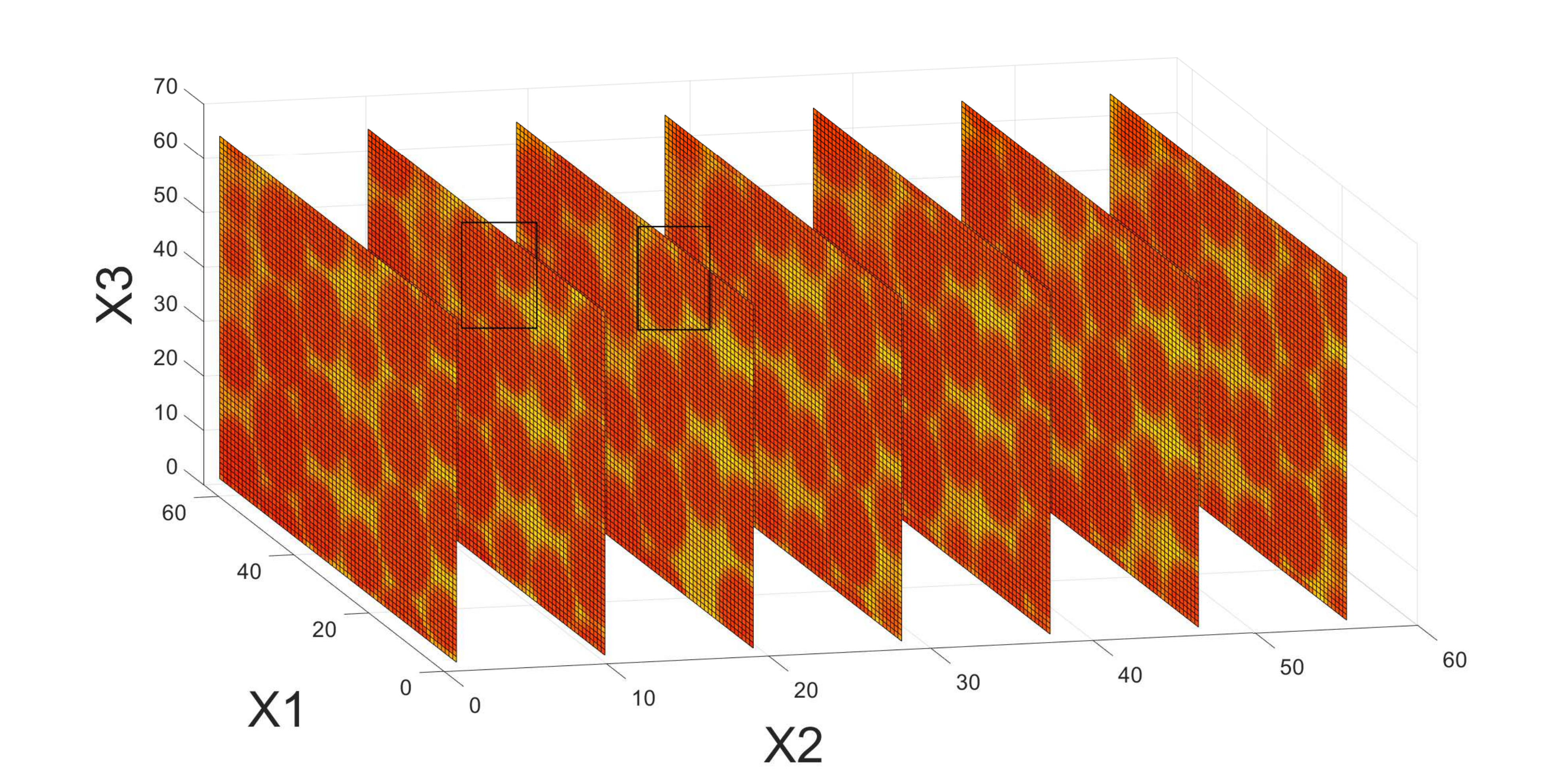}    \\
    (a) & (b)\\
 \includegraphics[width=1.4in]{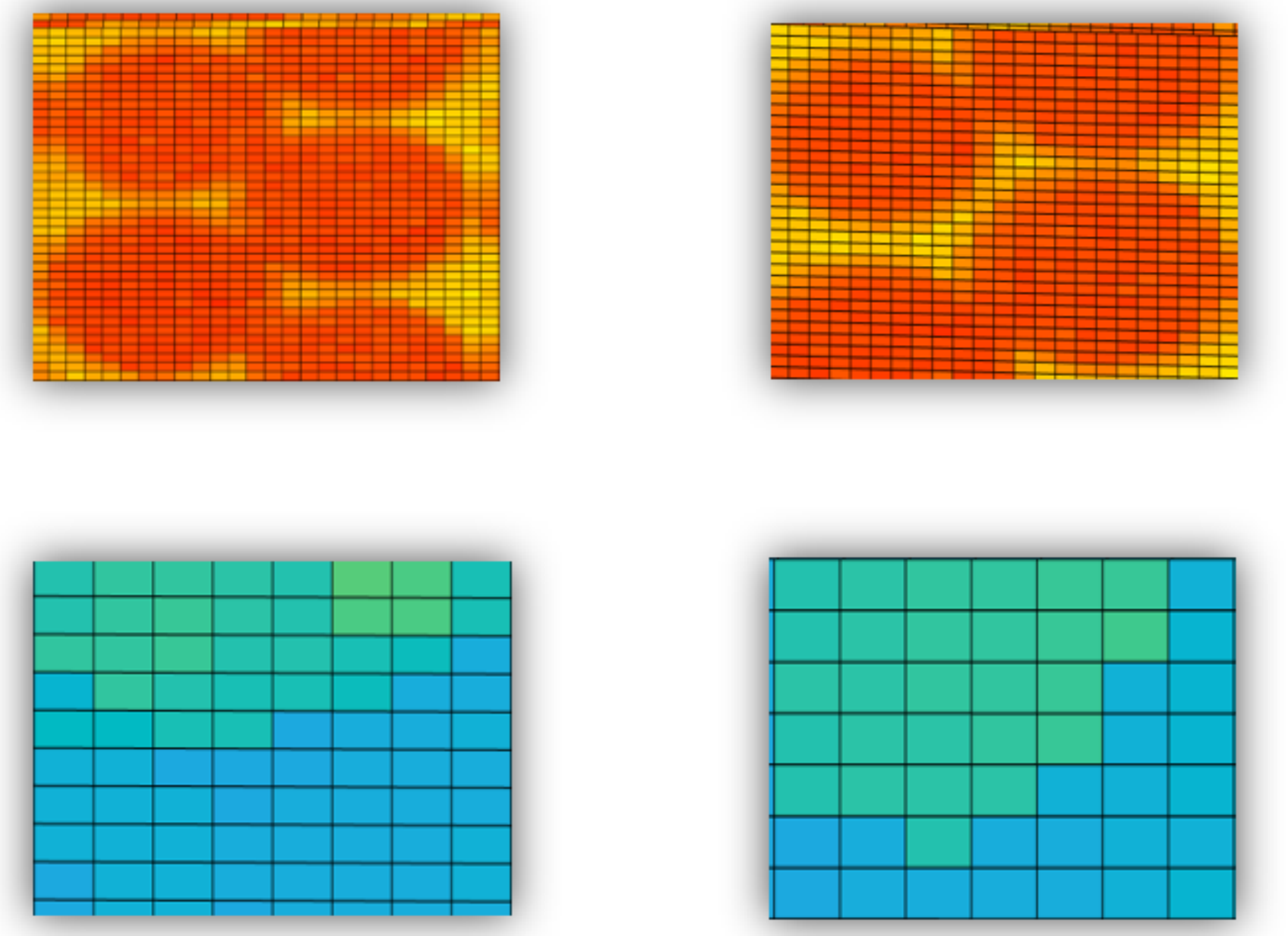} &     \includegraphics[width=3.2in]{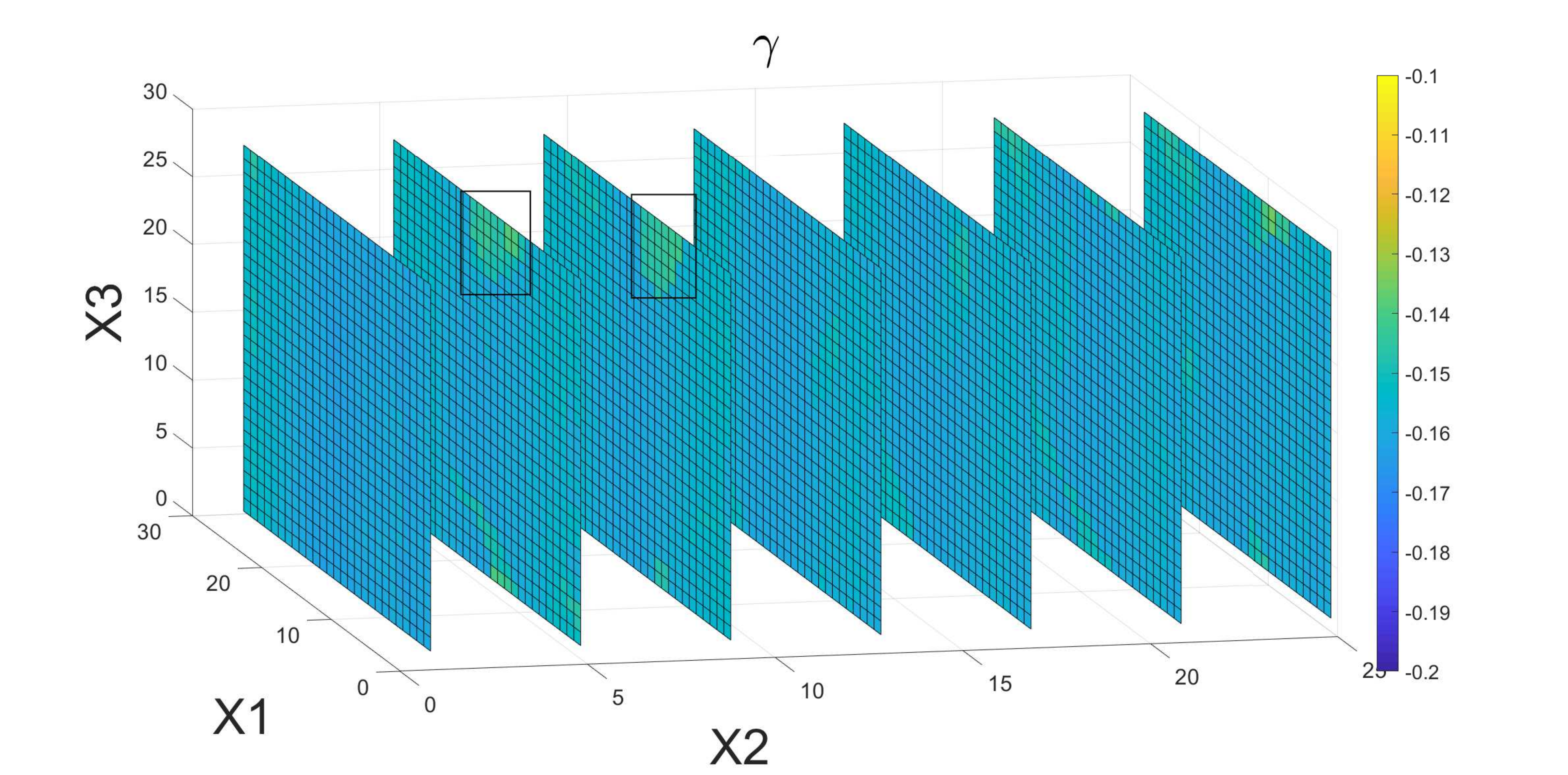} \\
 (c) & (d)
    \end{tabular}
  \end{center}
  \caption{(a) a colloidal particle image $f(x)$. (b) $2D$ slices of the $3D$ data. (c) zoomed-in images of the boxed regions of (b) and (d). Crystal images on top and their Euler angles $\gamma(x)$ at the botoom. (d) The Euler angles $\gamma(x)$ of the $3D$ colloidal particle image in (a) showed in the form of $2D$ slices. }
  \label{fig:rd1}
\end{figure}

\section{Conclusion}
\label{sec:con}

This paper has proposed a framework for  atomic resolution crystal image analysis in $3D$ based on a new $3D$ fast synchrosqueezed transform. It has been shown that the proposed methods are able to provide robust and reliable estimates of mesoscopic and microscopic properties, e.g., crystal defects, rotations, elastic deformations,
and grain boundaries in various synthetic and real data.  We focus on the analysis of images with the presence of only one type of known crystal lattice and without solid and
liquid interfaces in this paper. The extension is simple following the work in \cite{Lu2018}. The proposed method can be a standalone algorithm for crystal image analysis; it could also be applied to create training database for deep learning approaches because it is impractical (or even impossible) to create training database manually for $3D$ data.

{\bf Acknowledgments.} H. Yang thanks the support of the start-up grant by the Department of Mathematics at the National University of Singapore.

\bibliographystyle{unsrt}
\bibliography{Reference}
\end{document}